\documentclass[11pt, reqno]{amsart}

\usepackage{textcmds}  
\usepackage{amsmath, amssymb, amsfonts, amstext, verbatim, amsthm, mathrsfs}
\usepackage[mathcal]{eucal}
\usepackage{microtype}
\usepackage[all]{xy}
\usepackage[modulo]{lineno}
\usepackage[usenames]{color}
\usepackage{aliascnt}
\usepackage{enumitem}
\usepackage{xspace}

\usepackage[backref,colorlinks=true,linkcolor=blue,citecolor=blue,urlcolor=blue,citebordercolor={0 0 1},urlbordercolor={0 0 1},linkbordercolor={0 0 1}]{hyperref} 

\usepackage{amsfonts}
\usepackage{amssymb}
\usepackage[centertags]{amsmath}
\usepackage{amsthm}
\usepackage{graphics,graphicx}
\usepackage[margin=3cm]{geometry}
\usepackage{dsfont}
\usepackage{bm}
\usepackage{color}
\usepackage{subfigure}
\usepackage{amsmath}
\usepackage{array}
\usepackage[all]{xy}

\newtheorem{thm}{Theorem}[section]
\newtheorem{prop}[thm]{Proposition}
\newtheorem{theorem}[thm]{Theorem}
\newtheorem{assumption}[thm]{Assumption}
\newtheorem{proposition}[thm]{Proposition}
\newtheorem{lem}[thm]{Lemma}
\newtheorem{lemma}[thm]{Lemma}

\newtheorem{corollary}[thm]{Corollary}

\newtheorem{claim}[thm]{Claim}

\newtheorem{definition}[thm]{Definition}

\newtheorem{remark}[thm]{Remark}

\newtheorem*{thm*}{Theorem}
\newtheorem*{cor*}{Corollary}
\newtheorem*{prop*}{Proposition}

\newcommand{\e}{\epsilon}
\renewcommand{\a}{\alpha}
\renewcommand{\b}{\beta}
\newcommand{\g}{\gamma}
\renewcommand{\o}{\omega}
\newcommand{\T}{\mathcal{T}}
\newcommand{\Z}{\mathbb{Z}}

\newcommand{\R}{\mathbb{R}}

\newcommand{\C}{\mathbb{C}}
\newcommand{\M}{\mathcal{M}}

\newcommand{\Fuk}{\mathcal{F}uk}

\newcommand{\E}{\mathcal{E}}

\renewcommand{\P}{\mathbb{P}}

\renewcommand{\P}{\mathbb{P}}

\newcommand{\bdm}{\begin{displaymath}}
\newcommand{\edm}{\end{displaymath}}

\newcommand{\bq}{\begin{equation}}
\newcommand{\eq}{\end{equation}}

\numberwithin{equation}{section}

\title{Lagrangian tori in four-dimensional Milnor fibres}
\author{Ailsa Keating}
\date{October 15, 2015} 

\begin{document}

\begin{abstract}
The Milnor fibre of any isolated hypersurface singularity contains many exact Lagrangian spheres: the vanishing cycles associated to a Morsification of the singularity. Moreover, for simple singularities, it is known that the only possible exact Lagrangians are spheres. We construct exact Lagrangian tori in the Milnor fibres of all non-simple singularities of real dimension four. This gives examples of Milnor fibres whose Fukaya categories are not generated by vanishing cycles. Also, this allows progress towards mirror symmetry for unimodal singularities, which are one level of complexity up from the simple ones.
\end{abstract}

\maketitle

\tableofcontents

\section{Introduction}
\label{sec:introductionB}

Let $f: \C^{n+1} \to \C$ be a holomorphic function. Suppose 
its differential has an isolated zero at the origin: 
 $df|_0 = 0$, but $df \neq 0$ on some punctured open ball $B_r^\ast(0)$. An \emph{isolated hypersurface singularity} is the equivalence class of the germ of such an $f$, up to holomorphic re-parametrisation. 
 Assume for simplicity that $f(0)=0$.
The \emph{Milnor fibre} of $f$, studied in \cite{MilnorSingularities}, is the smooth manifold
\bq \label{eq:Milnorfibre}
M_f:= f^{-1}(\e_\delta) \cap B_\delta(0)
\eq
for suitably small $\delta$ and $\e_\delta$.  This carries an exact symplectic structure, say $\o = d \theta$, inherited from $\C^{n+1}$.\footnote{The exact symplectic manifold in equation \ref{eq:Milnorfibre} depends, a priori, on several choices. One can always attach cylindrical ends to its boundary, which gives a (non-compact) Liouville domain. We shall see that this Liouville domain is independent of choices, including holomorphic re-parametrisation -- see Lemma \ref{th:Milnorfibreindepreparametrisation}.}

Perturb $f$ generically, say to $\tilde{f}$. The singularity at zero splits into a collection of complex Morse singularities near zero. As we've chosen an isolated singularity, there are finitely many of them; their count is called the \emph{Milnor number} of $f$. 
 Fix a regular value of $\tilde{f}$, say $a$, and a collection of paths between each of the singular values and $a$. 
 The fibre of $\tilde{f}$ above $a$ is itself a copy of the Milnor fibre of $f$. Each path determines a Lagrangian sphere in it: the vanishing cycle associated to that path. 
 Topologically, this already gives much information \cite{MilnorSingularities}: indeed, $M_f$ is homotopic to a wedge of half-dimensional spheres:
 \bq
 M_f \cong \bigvee_\mu S^n
 \eq
 where $\mu$ is the Milnor number of $f$, 
   and a basis for $H_n(M_f; \Z)$ is given by a distinguished collection of vanishing cycles. One immediate consequence is that there is a large supply of Lagrangian spheres in Milnor fibres. 
   For $n \geq 2$, any such sphere $L$ is automatically exact: the cohomology class  $[ \theta|_{L}] = 0 \in H^1(L)$ vanishes. We are interested in the following question:
\begin{center}
\emph{What are the possible exact Lagrangian submanifolds in a Milnor fibre?}
\end{center}
We will focus on compact Lagrangians, and the case $n=2$, which means the Milnor fibre has real dimension four.
For a word about higher dimensions, see Section \ref{sec:higherdimensions}. 
Note that non--exact Lagrangians are comparatively easier to come by, notably tori in Darboux charts. 
 It could also be interesting to consider Lagrangians with different forms of rigidity requirements -- e.g.~ones with self--Floer cohomology defined and non-zero; however, we shall not address such questions here.

Particular attention has been paid to the symplectic geometry of a distinguished 
 collection of singularities, known as \emph{simple} or \emph{ADE--type} singularities -- for a far from exhaustive sample of the  flavour of questions studied, see e.g.~\cite{KhovanovSeidel, IUU, Evans, Wu, Chan}.
 Any isolated hypersurface singularity $f$ has an invariant called its \emph{intersection form}. In the case $n=2$, it agrees with the usual intersection form on $H_2(M_f; \Z)$; our orientations are chosen such that  for any compact Lagrangian $L \subset M_f$, we have $L \cdot L = -\chi(L)$.
 
Classically, one criterion that distinguishes simple singularities is that they are precisely the ones whose the intersection form is negative definite \cite{Tjurina}. 
Together with work of Ritter \cite{Ritter}, this implies that the only possible exact Lagrangians are spheres. 
For all other singularities, the intersection form  is semi-definite or indefinite. In particular, in the case $n=2$, this leaves room for tori. Our first result is that these exist:

\begin{thm*}(Theorem \ref{th:tori}.)
The Milnor fibre of any non-simple isolated hypersurface singularity of three variables contains an exact Lagrangian torus $T$, primitive in homology, and with vanishing Maslov class.
\end{thm*}
The vanishing of the Maslov class  is useful from the perspective of Floer theory, as it allows one to use a Fukaya category with absolute $\Z$--gradings. (See Section \ref{sec:Fukayabackground}.) 

The simple singularities have \emph{modality zero}, whereas all other singularities have positive modality. 
Loosely speaking, the modality of a singularity $f$ is the dimension of a parameter space 
covering a neighbourhood of $f$ in the space of singularities \emph{after} holomorphic reparametrization (see Definition \ref{def:modality}). 
This means that suitably interpreted, the non-simple singularities are generic (the simple ones correspond to points). 

Our approach is to construct tori explicitly in strategically chosen Milnor fibres, and use embeddings from these Milnor fibres to get tori in all others. These embeddings are geometric consequences of a phenomenon known as \emph{adjacency} of singularities. See Section \ref{sec:adjacency}.

The bulk of this article focuses on the singularities for which we construct $T$ explicitly. In the classification of Arnol'd \cite{Arnold6}, they are known as $T_{p,q,r}$ singularities. They are of the form
\bq
T_{p,q,r}(x,y,z) = x^p + y^q +z^r + axyz
\eq
where $p$, $q$ and $r$ are integers such that
\bq
\frac{1}{p}+\frac{1}{q} + \frac{1}{r} \leq 1
\eq
and $a \in \C$ is a complex parameter, which  is allowed to take all but finitely many values for each triple $(p,q,r)$. In the case where $\frac{1}{p}+\frac{1}{q} + \frac{1}{r} < 1$, the condition is $a \neq 0$. Note that while the holomorphic germ depends on the choice of $a$, its Milnor fibre, as an exact symplectic manifold, will not (Lemma \ref{th:indepofa}).
For all of the $T_{p,q,r}$ singularities, we have the following further properties:

\begin{itemize}
\item There exists an exact Lagrangian torus with vanishing Maslov class in every primitive class in the nullspace of the intersection form (Theorem \ref{th:allparabolic}).

\item We  compute Floer cohomology between our torus and every vanishing cycle in a distinguished collection (Proposition \ref{th:floercohomology}). In particular, if we equip our torus with any spin structure and a generic complex flat line bundle, all of those Floer groups vanish. 
\end{itemize}

\begin{remark}
If one only wanted to prove Theorem \ref{th:tori}, it would have been enough to consider the cases $(p,q,r) = (3,3,3)$, $(2,4,4)$ and $(2,3,6)$. Note these are the three triples of integers for which $\frac{1}{p}+ \frac{1}{q} + \frac{1}{r} =1$. 
\end{remark}

As a consequence of Proposition \ref{th:floercohomology}, we show the following:
\begin{thm*}(Theorem \ref{th:nofukayageneration}) 
The Fukaya category of the Milnor fibre of $T_{p,q,r}$ is not split-generated by any collection of vanishing cycles.
\end{thm*}

In contrast, for all the previously understood examples of Milnor fibres, the Fukaya category of closed exact Lagrangians is generated by vanishing cycles. These examples include simple singularities, and most weighted homogeneous singularities -- see Theorem \ref{th:Seidelweighted}, due to Seidel.

Our key technical result is a detailed geometric description of the Milnor fibre of $T_{p,q,r}$, together with a distinguished collection of vanishing cycles (Proposition \ref{th:Tpqr}). In particular, it also gives enough information to answer a mirror symmetric question about $T_{p,q,r}$:

\begin{thm*}(Theorem \ref{th:mirrorsymmetry})
There is an equivalence 
\bq
D^b \Fuk^{\to} (T_{p,q,r}) \cong D^b Coh (\P^1_{p,q,r})
\eq
where
\begin{itemize}
\item 
 the left-hand side is the bounded derived directed Fukaya category of the singularity $T_{p,q,r}$. (The category $\Fuk^{\to} (T_{p,q,r})$ is associated to a Lefschetz fibration on $\C^3$ given by a morsification of the singularity $T_{p,q,r}$, together with a distinguished collection of vanishing paths. After passing to the bounded derived closure, this is inpendent of choices.)
\item the right-hand side is the bounded derived category of coherent sheaves on an orbifold $\P^1$, with orbifold points with isotropies of order $p$, $q$ and $r$.
\end{itemize}
\end{thm*}
For more information, see Section \ref{sec:mirrorsymmetry}. 
The directed Fukaya category of $T_{p,q,r}$ contains less information than the `full' Fukaya category, which was the one that we considered  in Theorem \ref{th:nofukayageneration}. 
We hope our techniques will enable us to understand a version of mirror symmetry for this too. Already, Theorem  \ref{th:mirrorsymmetry} complements existing results in the literature. 
In particular, it provides an answer to Conjecture 1 of \cite{EbelingTakahashi11} and Conjecture 7.4 of \cite{Takahashi}; both of these articles consider mirror-symmetric questions for these singularities, studying algebraic invariants. 
  Among other works, Ueda \cite{Ueda} proves a related statement when $\frac{1}{p} + \frac{1}{q} + \frac{1}{r} =1$, using the fact that $T_{p,q,r}$ is weighted homogenerous in those cases.
Recent work of Cho, Hong, Kim and Lau \cite{ChoHongLau, ChoHongKimLau} studies the same spaces, but with the $A$ and $B$--sides swapped: they take $\P^1_{a,b,c}$ as the $A$--side. They show that its derived Fukaya category corresponds to matrix factorizations of a Landau--Ginzburg potential whose leading terms agree with $T_{p,q,r}$ (see also \cite{Seidel-genus-2} for a special 
case).

\subsection{Constructing the Lagrangian torus $T$ in the Milnor fibre of $T_{p,q,r}$: a sketch}

Proposition \ref{th:Tpqr} describes the Milnor fibre of $T_{p,q,r}$ as the result of smoothing the corners of the total space of a Lefschetz fibration; vanishing cycles for $T_{p,q,r}$ are given by matching paths for that fibration. In the example of $T_{3,4,5}$, here is what we get. 
The base of the fibration is given in Figure \ref{fig:basisforT345}, with matching paths given in colour. The smooth fibre is a three-punctured surface of genus four. In the fibre above $\star$, where most of the matching paths meet, the matching cycles restrict to the curves of Figure \ref{fig:fibreforT345}.

\begin{figure}[htb]
\begin{center}
\includegraphics[scale=0.85]{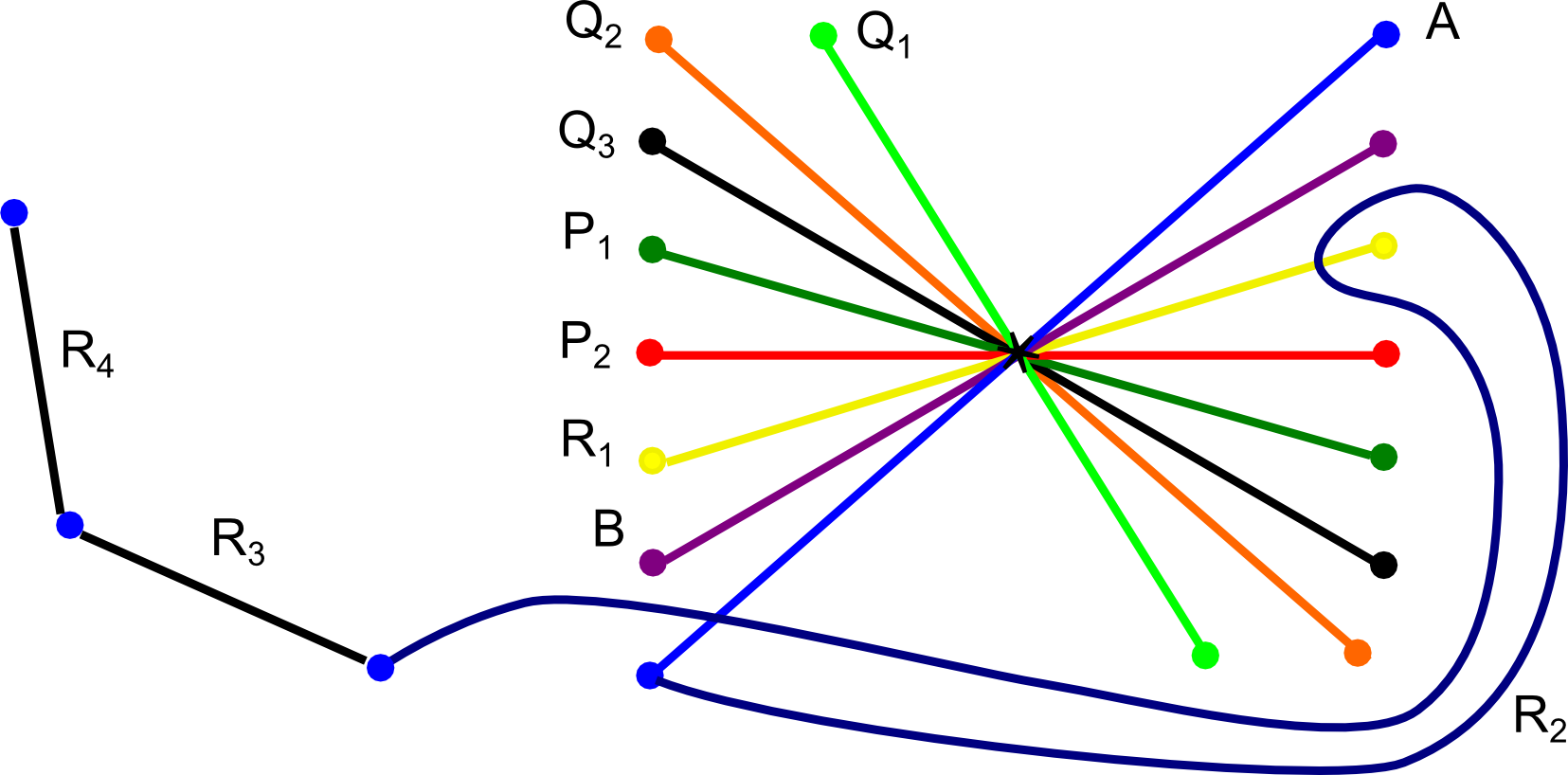}
\caption{Basis of the Lefschetz fibration used to describe $T_{3,4,5}$.
}
\label{fig:basisforT345}
\end{center}
\end{figure}

\begin{figure}[htb]
\begin{center}
\includegraphics[scale=0.85]{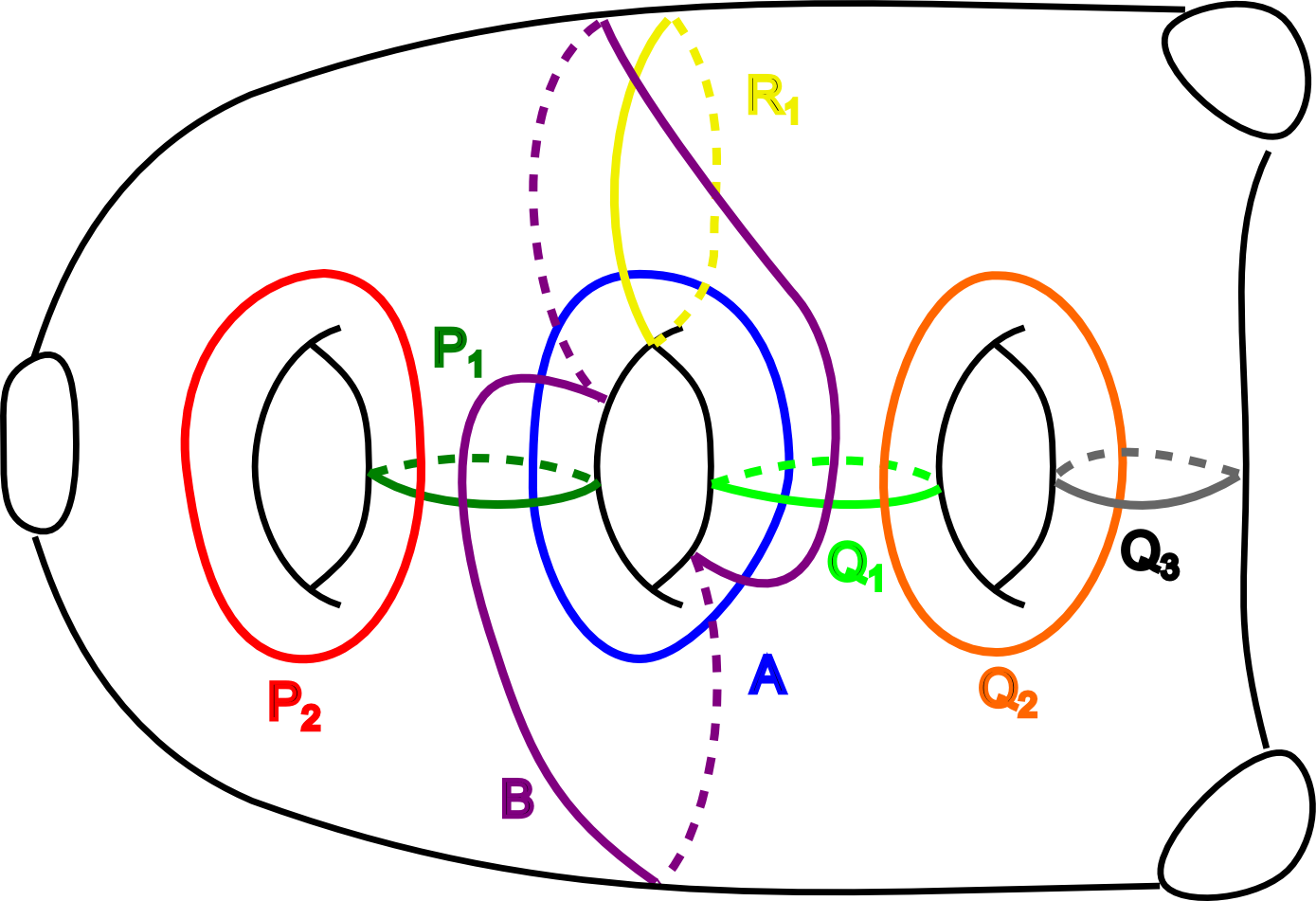}
\caption{Fibre above $\star$ in the Lefschetz fibration used to describe $T_{3,4,5}$.
}
\label{fig:fibreforT345}
\end{center}
\end{figure}
The $P_i$'s and $Q_i$'s form chains of lengths $p-1$ and $q-1$ in the fibre direction. The $R_i$'s form a chain in the base direction.
The spheres $A$ and $R_2$ intersect twice, with opposite orientations. 
After Hamiltonian isotopy, we can arrange for them not to intersect (Lemma \ref{th:AR2disjoint}); all other intersections are already minimal. 
This description also recovers Gabrielov's description of the intersection form of $T_{p,q,r}$ \cite{Gabrielov1, Gabrielov2}. We use three tools:
\begin{itemize}
\item A symplectic version of work of A'Campo \cite{ACampo99}, which gives an algorithm for describing the Milnor fibre of a function of two variables, together with a distinguished collection of vanishing cycles. See Section \ref{sec:ACampo}. 

\item A Thom-Sebastiani--type technique, based on work of Gabrielov \cite{Gabrielov1}, which, given a distinguished collection of vanishing paths and cycles for $f(x_0, \ldots, x_n)$, gives one for $f(x_0, \ldots, x_n)+ x_{n+1}^d$. See Section \ref{sec:Gabrielovcyclic}. 

\item Deformation arguments,  which allow us to embed the Milnor fibres we care about into the fibres of functions of the form $g(x,y) + z^d$, which are covered by the techniques above. We use deformation arguments when $r \geq 3$. See Sections \ref{sec:T333} and \ref{sec:Tpqr}.
\end{itemize}

The spheres $A$ and $B$ of Figure \ref{fig:basisforT345} and \ref{fig:fibreforT345} intersect in two points, with agreeing orientation. Our torus $T$ is obtained by performing Lagrangian surgery \cite{Polterovich} at each of those two points. Note that in general, such a construction would not preserve exactness. The fact that we have that, as well as the vanishing of the Maslov class of $T$, are consequences of different geometric features of the Milnor fibre. Details are in Section \ref{sec:torusconstruction}.

To calculate Floer cohomology groups between $T$ and the vanishing cycles, or the $A_\infty$--products in the directed Fukaya category $\Fuk^{\to} (T_{p,q,r})$, we reduce the problem to counting holomorphic curves on the Riemann surface that is the preimage of $\star$. In the case involving $T$, this requires a little care; see Section \ref{sec:Fukayaunimodal}.

\subsection{Extension to higher dimensions: Lagrangian $S^1 \times S^{n-1}$'s}\label{sec:higherdimensions}

The present article focuses on the case when $M_f$ has real dimension four. However, let us make the following remark about the higher-dimensional case: Starting with our description of the Milnor fibre of 
$$
x^p+y^q+z^r+axyz
$$
one gets a description of the Milnor fibre of its stabilization to a function of more variables, also known as $T_{p,q,r}$:
$$
x_0^p+x_1^q+x_2^r+ax_0 x_1 x_2 + x_3^2 + \ldots + x_{n-2}^2.
$$
In particular, there will be two $n$--dimensional vanishing cycles, $A'$ and $B'$, which intersect in two points with agreeing orientation. Performing surgery at those two points, one gets a Lagrangian $S^1 \times S^{n-1}$. One can check that the argument for exactness in the $n=2$ case readily extends here. Moreover, the result we use about adjacency of singularities also holds for singularities of more variables. Thus we have the following:

\begin{proposition}
In higher dimensions, the Milnor fibre of any singularity of positive modality contains an exact Lagrangian $S^1 \times S^{n-1}$, primitive in homology.
\end{proposition}


\subsection{Contents}
Section \ref{sec:singularity} collects material on singularity theory, principally from a symplectic perspective. In particular, we re-visit the methods of A'Campo (Subsection \ref{sec:ACampo}) and Gabrielov (Subsection \ref{sec:Gabrielovcyclic}) for studying Milnor fibres and vanishing cycles, and set them up in a symplectic framework. 
Section \ref{sec:Fukayabackground} gives some background on the version of the Fukaya category that we use, and relevant properties. 
The description of the Milnor fibre of $T_{p,q,r}$ together with a distinguished collection of vanishing cycles is in Section \ref{sec:vcycles}.
 Section \ref{sec:torusconstruction} constructs tori in these Milnor fibres, and proves Theorem \ref{th:tori} and Proposition \ref{th:allparabolic}; it also presents a useful local model for the construction. Section \ref{sec:Fukayaunimodal} gives the results relating to Floer cohomology between the torus we construct and the vanishing cycles of the $T_{p,q,r}$, including Theorem \ref{th:nofukayageneration} relating to generation of the Fukaya categories of the Milnor fibres of these singularities. Finally, Section \ref{sec:mirrorsymmetry} proves Theorem \ref{th:mirrorsymmetry} on homological mirror symmetry for $T_{p,q,r}$.


\subsection{Acknowledgements}

I thank my advisor, Paul Seidel, for many helpful conversations and suggestions. I first learnt of the techniques of A'Campo \cite{ACampo75, ACampo99} at a workshop in Symplectic and Contact Topology in Nantes in June 2011. I would like to thank the organisers for an enjoyable workshop, and Norbert A'Campo for stimulating lectures. 
   I am grateful to John Lesieutre and Tiankai Liu for explanations regarding del Pezzo surfaces, and to Bjorn Poonen for answering questions relating to properties of families of elliptic curves.  
   I also benefited from conversations with Mohammed Abouzaid, 
   Denis Auroux, Francesco Lin, Timothy Perutz and Umut Varolgunes.
   This project developed from a suggestion of Ivan Smith, whom I also thank for his continued interest in this work.

I was partially supported by NSF grants  DMS--1054622, DMS-1505798, and by a Junior Fellow award from the Simons Foundation.


\subsection{Notation}
All singular (co)homology groups have coefficients in $\Z$ unless otherwise specified.



\section{Background on singularity theory} \label{sec:singularity}

\subsection{Conventions: Lefschetz fibrations and parallel transport isomorphisms} 

\begin{definition}
A Lefschetz fibration $\pi: E \to \C^n$ is as follows:
\begin{itemize}
\item $E$ is a manifold with corners equipped with an exact symplectic form $\omega_E = d \theta_E$ and a compatible almost complex structure $J$, and $\partial E$ is weakly convex with respect to $\theta_E$. 
\item After a small smoothing of the corners of $E$, the resulting exact symplectic manifold $E'$ is a Liouville domain (i.e., has contact type boundary).
\item $\pi$ is proper and pseudoholomorphic with respect to $(J, i)$, and $\pi|\partial E$ is a submersion.
\item $\pi$ has finitely many critical points, at most one of which is in each fibre.
\item $J$ is integrable near each critical point, and the complex Hessian at each critical point is non-degenerate.
\item Each smooth fibre of $\pi$ (which is automatically symplectic, equipped with the restriction of $\omega_E = d \theta_E$), is a Liouville domain.
\end{itemize}
We will call $B = \pi(E)$ the base of the fibration.
\end{definition}

At any smooth point $x$, the symplectic form defines a prefered horizontal tangent space $TE^h_x$, given by taking the symplectic orthogonal to $\textrm{ker} (D\pi_x)$. Note that our definition is essentially the same as the one in \cite[Section 5]{Seidel-FukayaLefschetzI}, with the exception that we allow for higher dimensional bases, and that instead of requiring the horizontal tangent spaces $TE_x^h$ to be parallel to the boundary, we simply require the fibres to be Liouville domains.

Fix a path $\gamma(t)$, $t \in [0,1]$, in the base of $\pi$, avoiding the singular locus. The horizontal tangent spaces defines a symplectic parallel transport map along $\gamma$. However, as the horizontal tangents spaces needn't be parallel to the boundary, this will not in general be defined for all time. We can address this as follows.

Fix a fibre $F^o_{\ast}$. Using the Liouville flow on a collar neighbourhood of $\partial F^o_\ast$, we can attach conical ends to $F^o_\ast$, to get a non-compact exact symplectic manifold, say $({F}_\ast, \omega_\ast = d \theta_\ast)$. Varying $\ast$, we can attach conical ends to all of the fibres of $\pi$. (Note that this is naturally `coherent': no choice is involved when attaching the ends.) We have the following.

\begin{lemma}\label{lem:symplectic_parallel}
Fix a path $\gamma(t)$ $t \in [0,1]$ in the base of $\pi$, avoiding the singular locus. Then there is a family of  symplectomorphisms
\[
\Phi_t: {F}_{\gamma(0)} \to {F}_{\gamma(1)}
\]
with the following properties:
\begin{enumerate}
\item[(a)] up to compactly supported Hamiltonian isotopy, $\Phi_t$ recovers the symplectic parallel transport maps along $\gamma$, in so far as they are defined;
\item[(b)] $\Phi_t$ is an exact symplectomorphism: $\Phi^\ast (\theta_t) = \theta_0 + df_t$, some compactly supported $f_t$. 
\end{enumerate}
\end{lemma}

\begin{proof}
The key is \cite[Lemma 9.3]{Keating13}. This explains how to construct a symplectic form $\Omega'_E$ on $E$ such that
\begin{enumerate}
\item For any $\ast$, $\Omega_E |_{F^o_\ast} = \Omega_E ' |_{F^o_\ast}$;
\item $\Omega'_E$ and $\Omega_E$ agree outside a neighbourhood of $\partial E$;
\item The horizontal tangent space determined by $\Omega'_E$ in invariant under the fibre-wise Liouville flow in some collar neighbourhood of $\partial F^o_\ast$;
\end{enumerate}
Now consider the symplectic parallel transport maps defined by $\Omega'_E$; again, these will not in general be defined for all time. However, Property (3) allows us to extend the parallel transport to the completed fibres ${F}_t$. (Note that, by construction, no point can escape to infinity in finite time: for any compact set $U$ in the base and any choice of metric, the lifts to the horizontal tangent space of norm-one tangent vectors to points in $U$  have bounded norms.) Point (1) implies that this parallel transport is a symplectomorphism for the original symplectic structure on ${F}_t$, and point (2) implies (a). 
Claim (b) follows from the usual computation for symplectic parallel transport combined with (3) (outside a compact set, $\Phi_t$ preserves Liouville vector fields, and hence the $\theta_t$). 
\end{proof}

Modifying the path $\gamma$ by a smooth isotopy relative the the boundary (still avoiding singular values) changes the map $\Phi_1$ by pre-composing with an exact symplectic isotopy.

Fix a singular value $s$, a smooth regular value $\e$, and  and a smooth path between them that otherwise does not go through any singular value. This is called a \emph{vanishing path} for $x_s$. It determines a map from $F_\e$ to ${F}_{x_s}$.  The preimage of the singular point in ${F}_{x_s}$ is a Lagrangian sphere, called the \emph{vanishing cycle} associated to this singular value and path. The map on the fibres is an exact symplectomorphism from the complement of the vanishing cycle to the complement of the singular point.

      As well as a Lagrangian sphere in the fibre, a vanishing path gives a Lagrangian disc in the total space, by taking the union of all the vanishing cycles above that path.
      We will call such a Lagrangian disc a \emph{Lefschetz thimble}.    

Suppose that there are two vanishing paths $x_{s_1}$ and $x_{s_2}$, from $\e$ to critical values $s_1$ and $s_2$, which only intersect at the end-point $\e$. Suppose moreover that the corresponding two vanishing cycles are Hamiltonian isotopic. Then the union of the two $x_{s_i}$ is know as a \emph{matching path}, and, after  Hamiltonian isotopies, the two Lefschetz thimbles  can be glued together to form a Lagrangian sphere in the total space, called a \emph{matching cycle}. (We will want some control over the Hamiltonian isotopies -- see Section  \ref{sec:local_changes} for details.)

\subsection{Basic properties}

Unless otherwise specified, the singularity theory facts stated here are taken from the excellent survey \cite{Arnold6}.

\begin{definition}
Let $f: \C^{n+1} \to \C$ be a holomorphic function such that $f(0)=0$, $df|_0 =0$ and $df \neq 0$ on the punctured ball $B^\ast_r (0)$, some sufficiently small $r$.  An isolated hypersurface singularity at zero is the equivalence class of the germ of such an $f$, up to biholomorphic changes of coordinates (fixing $0$).
\end{definition}

Whenever we use the term \emph{singularity}, unless otherwise specified, we shall be referring to an isolated hypersurface singularity. We shall make repeated use of the following notion:

\begin{definition}
Let $(M, \omega_M, \theta_N)$ and $(N, \omega_N, \theta_N)$ be exact symplectic manifolds. An exact symplectic embedding $M \hookrightarrow N$ is an embedding $ f: M \to N$ such that 
\bq
f^\ast (\theta_N) = \theta_M + dh
\eq
where $h$ is a compactly supported function on the interior of $M$. 
An exact symplectomorphism is defined similarly.
\end{definition}

\subsubsection{Milnor fibres: definition and independence of choices} 

\paragraph{\textit{Smooth category.}}
Let $h: \C^{n+1} \to \R$ be the function given by
\bq
h(x_0, \ldots, x_n) = |x_0|^2 + \ldots + |x_{n+1}|^2.
\eq
Let $f$ be a singularity. 
The restriction $h: \,  f^{-1}(0) \to \R$ is a real algebraic function.
By the Curve Selection Lemma (\cite[Chap. 3]{MilnorSingularities}), it has isolated critical values. 
Let $\delta_f$ be the smallest positive one. 
For any $\delta < \delta_f$ and sufficiently small $\e_\delta$, we have $f^{-1} (\e_\delta) \pitchfork S_{\sqrt{\delta}}(0)$, where $S_{\sqrt{\delta}} (0)$ is the sphere of radius $\sqrt{\delta}$ about the origin. 

\begin{definition}\cite{MilnorSingularities} \label{def:Milnorfibre}
The Milnor fibre of $f$ is the smooth manifold with boundary $f^{-1}(\e_\delta) \cap \overline{B}_{\sqrt{\delta}}(0)$, for any such $\delta < \delta_f$ and $\epsilon_\delta \neq 0$. As a smooth manifold, this is independent of choices.
\end{definition} 
 Instead of $h$, consider any real algebraic function
\bq
\tilde{h}: \C^{n+1} \to [0 , \infty)
\eq
such that $\tilde{h}^{-1} (0) = 0$. 
The Curve Selection Lemma still applies: 
the restriction of $\tilde{h}$ to $f^{-1}(0)$ also has isolated critical values.
 For sufficiently small $\delta$ and $\e_\delta$, $f^{-1}(\e_\delta)$ is transverse to the $\delta$--level set of $\tilde{h}$, and 
  the manifolds  
  \bq
  f^{-1}(\e_\delta) \cap \{ \tilde{h} (x_0 ,\dots, x_{n+1}) \leq \delta  \}
  \eq
  are also (diffeomorphic) copies of the Milnor fibre. One can show this by linearly interpolating between $h$ and $\tilde{h}$, and noting that all the intermediate functions are real algebraic, with the same properties as used above.
  
  \paragraph{\textit{Symplectic category.}}
  The affine space $\C^{n+1}$ comes with a `standard' exact symplectic form, which is  the usual Kaehler form on $\C^{n+1}$: $\o = d \theta$, where $\theta = \frac{i}{4} \sum_{i} x_i d \bar{x}_i -  \bar{x}_i d x_i$. This restricts to an exact symplectic form on any of the Milnor fibres. By construction, the associated negative Liouville flow is the gradient flow of $h(x)$ with respect to the standard Kaehler metric. Suppose we use the cut-off function
  \bq
  h_A (\mathbf{x} ) = || A \mathbf{x} ||^2
  \eq
  for some $A \in GL_{n+1}(\C)$. On $\C^{n+1}$, the negative gradient flow of $h$ points strictly inwards at any point of the real hypersurface  $ || A \mathbf{x} ||^2 = \delta$. (This is of course true for any $\delta$, though we only need it for $\delta < \delta_{f,A}$.) In particular, for sufficiently small $\delta$ and $\e_\delta$, the copy of the Milnor fibre given by
  \bq
  f^{-1}(\e_\delta) \cap \{  h_A ( \mathbf{x} ) \leq  \delta  \}  
  \eq
is an exact symplectic manifold with contact type boundary (i.e.~a Liouville domain), and we can attach cylindrical ends to it using the Liouville flow on a collar neighbourhood of the boundary.  We call this the \emph{completed Milnor fibre} of $f$, and denote it $M_f$.

\paragraph{\textit{Independence of choices.}}
To define $M_f$, we chose $A \in GL_{n+1}(\C)$, then $\delta$,  then $\epsilon$. We show below that $M_f$ is independent of those choices (Lemma \ref{th:Milnorfibreindepchoices}) and of the holomorphic representative of the singularity (Lemma \ref{th:Milnorfibreindepreparametrisation}).

\begin{lemma}\label{th:Milnorfibreindepchoices}
Given a holomorphic function $f: \C^{n+1} \to \C$, the completed Milnor fibre $M_f$ is independent of the choice of $A$, $\delta$ and $\epsilon$ up to exact symplectomorphism.
\end{lemma}
 
 \begin{proof}
 One can use symplectic parallel transport arguments, with varying choices of total spaces and bases, appealing to Lemma \ref{lem:symplectic_parallel}.
 To understand independence of $\delta$ and $\e$ in that framework, see  \cite[Section 9]{Keating13}, where it is addressed explicitly.   Let's look at independence of $A$. Consider the map: 
    \begin{eqnarray}
\pi: \C^{n+1} \times GL_{n+1}(\C) & \to & \C \times GL_{n+1}(\C)  \\
 ( \mathbf{x} , A ) & \mapsto  & (f (\mathbf{x}), A)
\end{eqnarray}
where $GL_{n+1} (\C) \subset \C^{(n+1)^2}$ inherits the Kaehler structure. Fix a smooth path of matrices $A(t) \in GL_{n+1}(\C)$, for $t \in [0,1]$. Now choose a smooth family $\delta(t)$ with $0 < \delta(t) < \delta_{f, A(t)}$. We can find a smooth family of sufficiently small $\e(t)$ such that 
\bq
f^{-1} (\e(t)) \pitchfork \{ h_{A(t)} (\mathbf{x}) = \delta(t) \}
\eq
and for each $t$, the space 
\bq
f^{-1} (\e(t)) \cap \{ h_{A(t)} (\mathbf{x}) \leq \delta(t) \}
\eq
is a copy of the Milnor fibre of $f$. 
Consider the path $\gamma(t) = (\e(t), A(t))$ in the base. 
Choose an open neighbourhood $U$ of that path, such that all the fibres above points of $U$ are smooth. 
We already have smooth choices of cutoffs for the fibres above $\gamma(t)$; extend this smoothly to all fibres above $U$. 
Call the resulting total space (that is, the union of truncated fibres) $E$. This is now the total space of a Lefschetz fibration. We now appeal to Lemma \ref{lem:symplectic_parallel}.
\end{proof}

\begin{lemma}\label{th:Milnorfibreindepreparametrisation}
The completed Milnor fibre $M_f$ does not depend on the choice of holomorphic representative of $f$. More precisely, suppose  $f = g \circ \rho$, some holomorphic change of coordinates $\rho$. Then there is an exact symplectomorphism from $M_f$ to $M_g$. 
\end{lemma}

\begin{proof} We'll use symplectic parallel transport again, and proceed in two steps. First, assume that $\rho$ is linear: $\rho (\mathbf{x}) = D \rho|_0 (\mathbf{x})$, with $D \rho|_0  \in GL_{n+1}(\C)$. Consider the map:
\begin{eqnarray}
\sigma: \C^{n+1} \times GL_{n+1}& \to & \C \times GL_{n+1}(\C) \\
( \mathbf{x}, A) & \mapsto & ( f(A \mathbf{x} ), A ).
\end{eqnarray}
Pick a path $A(t) \in GL_{n+1}(\C)$, with $t \in [0,1]$, such that $A(0) = Id$, $A(1) =  D \rho|_0$. Suppose that $\delta$ and $\e$ are such that 
\bq
f^{-1} (\e) \pitchfork \{ || \mathbf{x} ||^2 = \delta \}
\eq
and the space
$
f^{-1} (\e) \cap \{ || \mathbf{x} ||^2 \leq \delta \}
$
is a copy of the Milnor fibre of $f$. Then for all $t$, we have that
\bq
(f \circ A(t))^{-1} (\e) \pitchfork \{ || A(t) \mathbf{x} ||^2 = \delta \}
\eq
and, moreover, the space 
\bq
(f \circ A(t))^{-1} (\e) \cap \{ || A(t) \mathbf{x} ||^2 \leq \delta \}
\eq
is a copy of the Milnor fibre of $f \circ A(t)$, which we know to be smoothly isomorphic to a Milnor fibre of $f$. This gives smooth choices of cutoffs along the path $(\e, A(t))$ in the base of $\sigma$. Proceeding similarly to the previous Lemma, we see that the completed Milnor fibres of $f$ and $f \circ D \rho |_0$ are exact symplectomorphic. 

We're left with the case of a general $\rho$. Consider the following map:
\begin{eqnarray}
\tau: \C^{n+1} \times B_2 (0) & \to & \C \times B_2(0) \\
( \mathbf{x}, c) & \mapsto & ( f \circ (D \rho|_0 + c (\rho - D \rho|_0))(\mathbf{x}), c)
\end{eqnarray}
where $B_2(0) \subset \C$ is the disc of radius 2. There exists a (sufficiently small) $\delta>0$ such that
\bq
(f \circ (D \rho|_0 + c (\rho - D \rho|_0)))^{-1} (0) \pitchfork \{ || D \rho|_0 \mathbf{x} ||^2 = \delta \}
\eq
for all $c \in [0,1]$. This allows us find a path between copies of the Milnor fibres of $f \circ \rho$ and $f \circ D \rho|_0$, with smooth choices of cutoffs along those paths. We can then proceed with a symplectic parallel transport construction as before, again appealing to Lemma \ref{lem:symplectic_parallel}.
\end{proof}

As much as possible, we have tried to refer to  the exact manifold with cylindrical ends as the  \emph{completed} Milnor fibre of $f$, and to the manifolds with boundary simply as copies of the Milnor fibre of $f$. (The reader should be aware that there is little practical difference for our purposes -- in particular, we will only be considering exact compact Lagrangians.)
Finally, it will be useful to note the following:

\begin{remark}\label{rk:weightedcutoffs}
Suppose that the function $f$ has a single isolated singularity at $0$, and is weighted homogeneous. Then the completed Milnor fibre $M_f$ is exact symplectomorphic to any of the hypersurfaces $f^{-1}(\epsilon)$, for any $\e \in \C^\ast$, equipped with the standard Kaehler exact symplectic form. This can be done, for instance, by starting with the trivial identification of $f^{-1}(\e) \cap B_\delta$ with itself, and then using the Liouville flows one each of $M_f$ and $f^{-1}(\e)$ to identify the two infinite ends. (This uses weighted homogeneity of $f$, which implies that the negative Liouville flow will take any point in $f^{-1}(\e) \backslash B_\delta$ into $f^{-1}(\e) \cap B_\delta$ in finite time.)
\end{remark}

\subsubsection{Milnor numbers, intersection forms and sufficient jets}

Fix a singularity $f$. A generic small perturbation to a Morse function $\tilde{f}$ is called a \emph{Morsification} of $f$. It has a collection of non-degenerate critical points near $0$. 
The number of these critical points is independent of the Morse perturbation, and finite. 
It is called the \emph{Milnor number} $\mu$ of $f$. 
Moreover, $M_f$ is homotopy-equivalent to a wedge of $\mu$ half-dimensional spheres (\cite{MilnorSingularities}).

One can equip the middle-homology of $M_f$, $H_n(M_f) \cong \Z^{\mu}$, with an \emph{intersection form} (for instance, formally, using cohomology with compact support). When the fibre has real dimension four ($n=2$), the form is symmetric, and for any compact Lagrangian $L$, we have
\bq
L \cdot L = -\chi (L).
\eq
One can obtain a singularity in $n+2$ variables from a singularity in $n+1$ variables by adding a $z_{n+2}^2$ term, a process known as \emph{stabilization}. This changes the intersection form; the resulting sequence of intersection forms has period four \cite[Chapter II, Section 1.7]{Arnold6}.

Milnor numbers have another important application, as follows. Consider the polynomial expansion of a singularity $f$ at the singular point $0$. The $k$-jet of $f$ is called \emph{sufficient} if any two functions with that jet are equivalent (that is, there exist a biholomorphic change of coordinates between them). 

\begin{theorem}\cite{Tougeron, Arnold68}\label{th:tougeron}
The $(\mu+1)$--jet of a function at an isolated critical point with Milnor number $\mu$ is sufficient.
\end{theorem}


In particular, all of the singularities we consider are equivalent to polynomials.

\subsubsection{Modality}

Suppose you have a Lie group $G$ acting on a manifold $\mathcal{M}$, and $p\in \mathcal{M}$. 

\begin{definition}\cite[Section 1.6]{Arnold6} \label{def:modality}
The modality of $p$ under the action of $G$ is the least integer $m$ such that a sufficiently small neighbourhood of $p$ is covered by a finite number of $m$--parameter families of orbits. 
\end{definition}

The group of biholomorphic coordinate changes $(\C^{n+1},0) \to (\C^{n+1}, 0)$ acts on the space of holomorphic functions $f: (\C^{n+1}, 0) \to (\C, 0)$. This induces an action on the $k$--jet space, for any fixed $k$.
\begin{definition}
The modality of a singularity $f$ is the modality of any of its $k$--jets, for $k \geq \mu(f)+1$.
\end{definition}
Loosely, the reader can think of modality as the ``number of complex parameters'' of the family that the singularity belongs to. For example, a modality one singularity is
\bq
x^4+y^4+ax^2y^2
\eq
where $a$ is a complex parameter such that $a^2\neq 4$, associated to the cross-ratio of four lines.

Modality is unchanged under stabilization; this follows for instance from the discussion at the end of \cite[Chapter I, Section 1.9]{Arnold6}.

\subsubsection{Adjacency and embeddings}\label{sec:adjacency}

Suppose $[f]$ and $[g]$ are isolated hypersurface singularities, described as equivalences classes of germs. We say that $[f]$ is \emph{adjacent} to $[g]$ if there exists an
arbitrarily small perturbation $p$ such that $[f + p] = [g]$. For instance, $[x^m]$ is adjacent to $[x^n]$ for all $m \geq n$. Symplectically, we have that:

\begin{lemma}\cite[Lemma 9.9]{Keating13}\label{th:adjacentembedding}
Suppose that $[f]$ and $[g]$ are isolated hypersurface singularities such that $[f]$ is adjacent to $[g]$. Then there exists an exact symplectic embedding from a non-completed Milnor fibre of $g$ into a completed Milnor fibre of $f$.
\end{lemma}

Moreover, suppose $f+p= \widetilde{g}$, some representative $\widetilde{g}$ of the class of $[g]$. A Morsification of $f$ can be obtained by picking a Morsification of the function $\widetilde{g}$ on some neighbourhood of $0$. In particular, it follows that:

\begin{lemma}\label{th:adjacentembeddingcycles}
Suppose that $[f]$ and $[g]$ are singularities such that $[f]$ is adjacent to $[g]$. Then, under the exact symplectic embedding from a non-completed Milnor fibre of $g$ into a completed Milnor fibre of $f$, vanishing cycles for $g$ get mapped to (Hamiltonian displacements of) vanishing cycles for $f$. 
\end{lemma}

In particular, using Liouville flows at both ends, we can find an exact symplectic embedding of any compact subset of the (completed) Milnor fibre of $g$ into the Milnor fibre of $f$.
We shall make use of the following:

\begin{theorem}\cite[Theorem 10.1]{Durfee} \label{th:Durfee}
Any positive modality singularity is adjacent to a modality one singularity.
\end{theorem}

\subsubsection{Singularities of modality zero and one}  \label{sec:simpleunimodalintro}
\label{sec:unimodalreps}

Modality zero singularities are also known as \emph{simple} singularities. They are also sometimes known as \emph{elliptic} singularities.  Their intersection forms correspond to Dynkin diagrams of the form $A_m$, $D_m$ and $E_6$, $E_7$ and $E_8$. They are characterized by the following property.

\begin{theorem}\cite{Tjurina} For $n = 2 \,(mod \, 4)$, simple singularities have a negative definite intersection form. Moreover, they are the only singularities whose intersection form is definite.

\end{theorem}

 Modality one singularities are also known as \emph{unimodal}. They are the ones we shall be primarily concerned with. They are classified into three families \cite[Chapter I, Section 2.3]{Arnold6}. Three-variable representatives are as follows.
 
 \begin{itemize}
\item Three parabolic singularities:

\vspace{3pt}
\begin{tabular}{c}
 \qquad 
$\left\{\begin{array}{l}
 T_{3,3,3}: \quad x^3+y^3+z^3+axyz  \qquad \quad  \hspace{7pc} a^3+27 \neq 0 \\
 T_{4,4,2}: \quad  x^4+y^4+z^2+axyz \qquad \quad a \in \C \text{ such that }\,\,\, a^2-9 \neq 0 \\
T_{6,3,2}: \quad   x^6+y^3+z^2+axyz \qquad \quad\hspace{7pc}a^6-432 \neq 0 
  \end{array}\right. $
\end{tabular}
\vspace{3pt}
\item The hyperbolic series $T_{p,q,r}$:
$$x^p+y^q+z^r+axyz$$
 where $a \in \C^\ast$, and  $p,q,r \in \mathbb{N}$ are such that $1/p+1/q+1/r <1$. 
 
 \vspace{3pt}
\item 14 `exceptional' singularities. These are the objects, for instance, of strange or Arnol'd duality. We shall not consider them further here.
\end{itemize}
Unimodal singularities of more variables are stabilizations of the above; unimodal singularities of one or two variables are the singularities which stabilize to one of the above.

 The singularity  $T_{p,q,r}$ has Milnor number $p+q+r-1$ \cite[Theorem 2]{Gabrielov2}.

\begin{lemma}\label{th:indepofa}
The Milnor fibre of each $T_{p,q,r}$ is independent of $a$.
\end{lemma}

\begin{proof}
 We want to use a symplectic parallel transport argument, with the map
\begin{eqnarray}
\C^{4} & \to & \C^2 \\
(x,y,z, a) & \mapsto & (x^p + y^q+z^r+axyz, a)
\end{eqnarray}
The singular locus $\mathcal{C} \subset \C^2$ is algebraic, so its complement is connected. In order to use symplectic parallel transport, one just needs to be able to make choices of cutoffs  that vary smoothly over a path in the smooth locus. In the case of the three parabolic singularities, whose standard representative germs (the weighted homogeneous polynomials chosen here) have no singularity apart from the origin, one can readily make such a choice. 

For the other singularities, this follows from the following observation: suppose that for some $r>0$, 
 \begin{equation}
 \{ x^p+y^q+z^r + xyz =0 \} \pitchfork \{ |x|^2 + |y|^2 + |z|^2 =r \}.
 \end{equation}
 Then, for any complex constant $\xi \neq 0$, and choices of $p^{th}$, $q^{th}$ and $r^{th}$ roots for $\xi$, we have
  \begin{equation}
 \{ x^p+y^q+z^r +\xi^{\frac{1}{p} + \frac{1}{q} + \frac{1}{r} -1} xyz =0 \} \pitchfork \{ |\xi|^{\frac{2}{p}}|x|^2 + |\xi|^{\frac{2}{q}}|y|^2 + |\xi|^{\frac{2}{r}}|z|^2 =r \}.
 \end{equation}
 This gives choices of cut-off functions $h_A$, for diagonal matrices $A$. Now again appeal to Lemma \ref{lem:symplectic_parallel}.
\end{proof}

There are many alternative standard representatives for these singularities; notably, in the case of the form $T_{p,q,2}$, we have \cite{Arnold73, Arnold75}: 
\begin{eqnarray}
T_{4,4,2}: &    x^4+y^4+z^2+a x^2 y^2 & \quad a^2 \neq 4  \\
 T_{6,3,2}:&  x^6+y^3+z^2+ a x^2 y^2 &  \quad  4a^3+ 27 \neq 0 \\
T_{p,q,2 }:  &  x^p + y^q + z^2 + a x^2 y^2 &  \quad  a \neq 0 \quad  (1/p+1/q < 1/2)
\end{eqnarray}
(The attentive reader might wonder why the numbers of values of $a$ that are excluded are different for the previous presentations; this is due to redundancies for some of the descriptions.)
Using Tougeron's result on sufficient jets (Theorem \ref{th:tougeron}), we are free to add monomials of degree at least $p+q+1$ to these. Thanks to classification results of Arnol'd \cite[Section 14]{Arnold75}, one can actually do even more than that. We shall use the representation
\bq
T_{p,q,2}: (x^{p-2}-y^2) (x^2-\lambda y^{q-2})+ z^2
\eq
for any $\lambda \in \C$ such that the resulting polynomial has an isolated singularity. 

Whenever $p' \geq p$, $q' \geq q$ and $r' \geq r$, the singularity $T_{p',q',r'}$ is adjacent to $T_{p,q,r}$. Moreover, each of the fourteen exceptional singularities is adjacent to a parabolic -- see e.g.~\cite[Chapter I, Section 2.7]{Arnold6}. An immediate consequence of Theorem \ref{th:Durfee} is the following.

\begin{corollary}\label{th:adjacenttoparabolic}
Any positive modality singularity is adjacent to at least one of the parabolic singularities. 
\end{corollary}

For the intersection forms of parabolic and hyperbolic singularities, see Section \ref{sec:GabrielovDynkin}.

\subsection{Some Picard-Lefschetz theory}\label{sec:PicardLefschetz}

References for this section are the survey book  of Arnol'd and collaborators \cite{Arnold6} for classical material (in the smooth category), and Seidel's book \cite{Seidel08} for a detailed exposition in the symplectic category. We use a small variation on his framework.


Fix a singularity $f$, and a Morsification $\tilde{f}$.  Unless otherwise stated, Morsifications are assumed to have distinct critical values.
  Pick a regular value $\e$ of $\tilde{f}$.  

\begin{lemma}\label{lem:fibre_Morse}
Let $\delta$ and $\e_\delta$ be as in Definition \ref{def:Milnorfibre}. Then provided the perturbation is sufficiently small,   $\tilde{f}^{-1}(\e_\delta)$ intersects $S_{\sqrt{\delta}}(0)$ transversely, $\tilde{f}^{-1}(\e) \cap \overline{B}_{\sqrt{\delta}}(0)$ is a Liouville domain, and moreover, after attaching conical ends, this is exact symplectomorphic to the Milnor fibre of $f$. Call this space $F_{\e_\delta}$.
\end{lemma}

\begin{proof}
Smoothly, this was shown by Siersma -- see e.g.~the account in \cite[Lemma 3.3]{Dimca}. For a symplectic argument, you could for instance  use symplectic parallel transport with a piece of the fibration $\C^{n+2} \to \C^2$ given by $(\mathbf{x}, \gamma) \mapsto \big( f(\mathbf{x})+\gamma \big(\tilde{f}(\mathbf{x})- f(\mathbf{x})  \big) , \gamma \big) $.
  \end{proof}

    \begin{definition}  Fix a regular value $\e$ of the Lefschetz fibration $\tilde{f}$.
A distinguished collection of vanishing paths is a cyclically ordered family of  vanishing paths $\gamma_i$ between $\e$ and the singular values $x_s$, with $i=1, \ldots, \mu(f)$, such that:

\begin{itemize}

\item  They only intersect at $\e$.

\item Their starting directions $\R_+ \cdot \gamma_i'(0)$  are distinct.

\item They are cyclically ordered by clockwise exiting angle at $\e$.

\end{itemize}

\end{definition}

\begin{remark} This differs slightly from tradition: it is more common to use an absolute ordering compatible with the one described above. This is necessary if the smooth value  lies on the boundary of a domain. For the mutations we consider in the subsequent section, we only need the cyclic ordering.
\end{remark}

The resulting vanishing cycles give a cyclically ordered, so-called \emph{distinguished} basis for the middle-dimension homology of the Milnor fibre.
    
    \begin{remark}
Vanishing cycles for $f$ are independent of the choices of made for defining the Milnor fibre of $f$, in the following sense: under any of the parallel transports used to construct symplectomorphisms between representatives of the Milnor fibre (Lemmas \ref{th:Milnorfibreindepchoices} and \ref{th:Milnorfibreindepreparametrisation}), vanishing cycles get mapped to vanishing cycles. Moreover, under the embeddings provided by adjacencies (Lemma \ref{th:adjacentembedding}), vanishing cycles also get taken to vanishing cycles. This can be shown in all cases by enriching the Lefschetz fibrations used for the parallel transport argument in the proof of the lemmas with a parameter to introduce the Morsification. In the case of Lemma \ref{th:Milnorfibreindepchoices}, we use the same Morsification for the different cut-off choices; in the case of Lemma \ref{th:Milnorfibreindepreparametrisation}, we reparametrize the Morsification along with $f$. If $f$ is adjacent to $g$, we can use a Morsification of $f$ given by picking small, generic perfurbation of germs $p$ and $g$, where $f = g-p$. 

    \end{remark}

   \subsubsection{Dehn twists and mutations}
   
 Fix a  distinguished collection of vanishing paths, say $\gamma_1, \ldots, \gamma_\mu$. 
 Pre-compose $\gamma_i$ with a clockwise loop around $\gamma_{i-1}$. Call the resulting vanishing path $\tau_{i-1} (\gamma_i)$. See Figure \ref{fig:mutation}. 
  \begin{figure}[htb]
\begin{center}
\includegraphics[scale=0.9]{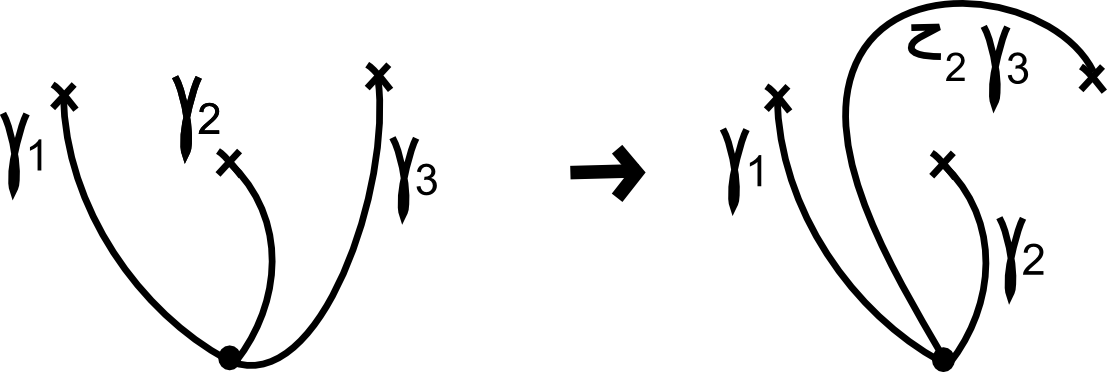}
\caption{ A mutation of vanishing paths
}
\label{fig:mutation}
\end{center}
\end{figure}
  The ordered collection of paths 
 \bq
 \gamma_1, \ldots, \gamma_{i-2}, \tau_{i-1}( \gamma_i), \gamma_{i-1}, \gamma_{i+1}, \ldots, \gamma_{\mu}
 \eq
  is a new distinguished collection of vanishing paths. One could also consider the inverse operation: pre-compose $\gamma_{i-1}$ with an anti-clockwise loop around $\gamma_{i}$. We call the resulting vanishing path $\tau^{-1}_i (\gamma_{i-1})$. Similarly, the ordered collection
 \bq
   \gamma_1, \ldots, \gamma_{i-2}, \gamma_{i}, \tau^{-1}_i( \gamma_{i-1}), \gamma_{i+1}, \ldots, \gamma_{\mu}
  \eq
  is a new distinguished collection. (Remember that we only have a cyclic ordering, so you can do this for $\gamma_1$ and $\gamma_\mu$ too. However, if you try using non-consecutive paths, you get intersection points.) These two operations are called \emph{mutations}. They are defined up to isotopy relative to marked points. What is the effect on vanishing cycles?
  
  \begin{theorem}\cite[Chapter 3]{Seidel08}
  Let $L_j$ be the vanishing cycle associated to $
  \gamma_j$. Then the vanishing cycle associated to $\tau_{i-1} (\gamma_i)$, say $L'_{i-1}$, is given by
  \bq
  L'_{i-1} = \tau_{L_{i-1}} (L_i)
  \eq
  where $ \tau_{L_{i-1}}$ is the Dehn twist about $L_{i-1}$.
  \end{theorem}  
    
 It turns out that these are the only moves that one needs:
 
 \begin{lemma}\label{th:mutationseq}
Given a Morsification of $f$, and a marked regular value $\e$, one can get from any collection of distinguished vanishing paths to any other through a sequence of mutations.
 \end{lemma}

\begin{proof}
This follows from the fact that the group of isotopy classes of diffeomorphisms of the disc preserving a collection of marked points is given by the braid group on those points; one can arrange for each elementary braid to correspond to a mutation on the initial distinguished collection of vanishing paths (and this to a composition of mutations for subsequent collections).  For a detailed account, see \cite[Chapter II, Section 1.9]{Arnold6}.
\end{proof}
 
 What if we had chosen a different Morsification? W.l.o.g.~assume you want to compare two Morsifications given by polynomial expressions. Consider the parameter space for the coefficients of these polynomials, a complex vector space $\C^N$. Those deformations which are not Morse belong to  a space of the parameter space $\C^N$ that is itself cut out by finitely many (complex) polynomial equations. In particular, its complement is path connected. Thus one can deform any Morsification to any other through Morsifications.  In particular, for the purpose of studying vanishing cycles, it does not matter which Morsification we pick.

\subsubsection{Distinguished bases for unimodal singularities}\label{sec:GabrielovDynkin}

The intersection forms, calculated by Gabrielov, are given by the Dynkin diagrams of Figure \ref{fig:Tpqr'} \cite{Gabrielov1, Gabrielov2}. The numbered dots represent an ordered basis of vanishing cycles, which give a basis of $\Z^\mu$. Each vanishing cycle is a Lagrangian sphere, and so has self-intersection $-2$. Full lines represent an intersection of $+1$, and double dashed lines an intersection of $-2$. 

\begin{figure}[htb] \label{fig:Tpqr'}
\begin{center}
\includegraphics[scale=0.4]{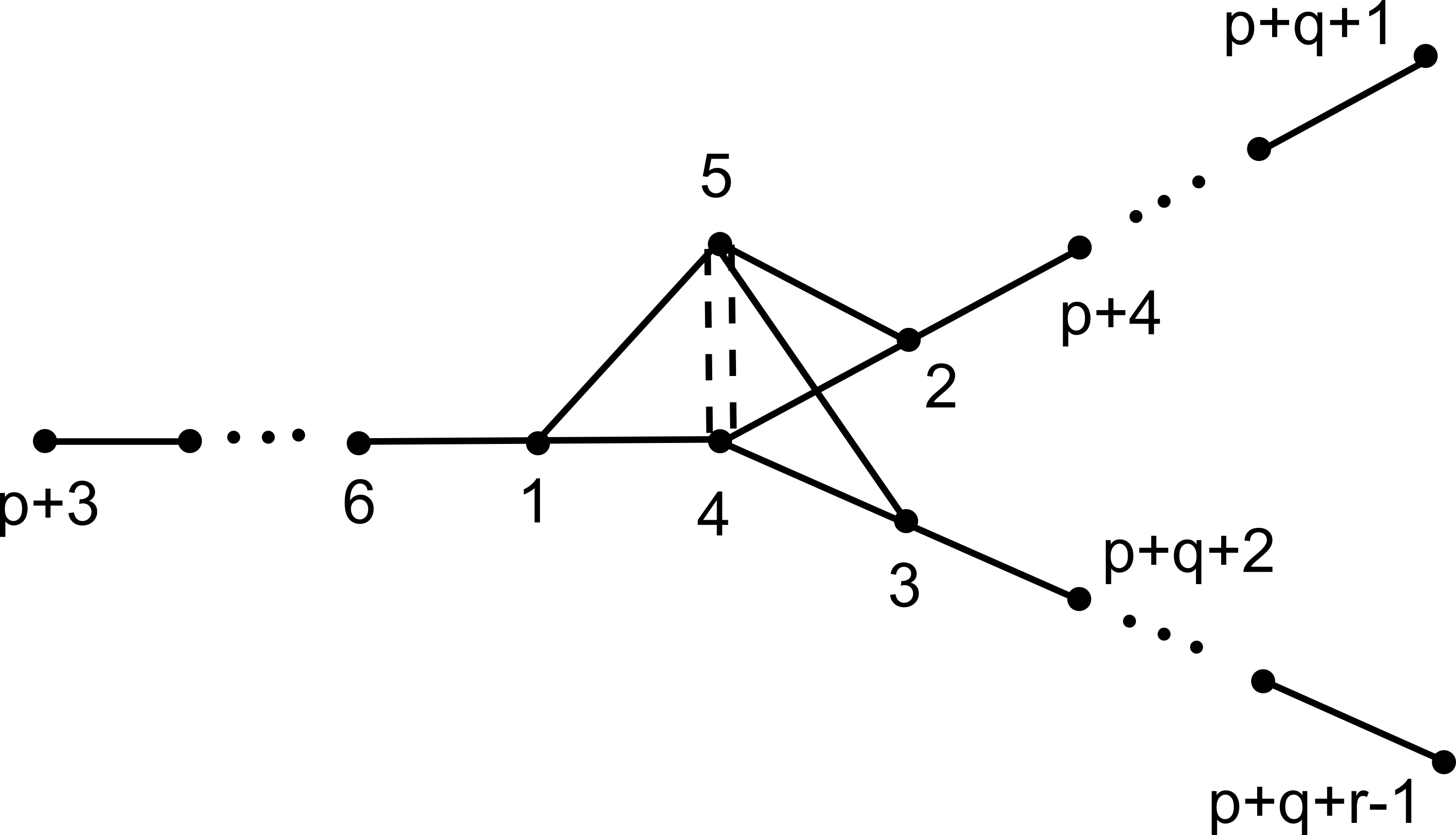}
\caption{ Dynkin diagram for the intersection form of $T_{p,q,r}$
}
\label{fig:Tpqr'}
\end{center}
\end{figure}

The intersection forms of the parabolic singularities in semi-definite, with a rank two null-space. The intersection forms of the hyperbolics is indefinite, with a rank one null-space.

 \subsubsection{Essentially local changes of the symplectic form}\label{sec:local_changes}
 
 First, an observation:
 \begin{claim}
 Let $\pi: \C^{n+1} \to \C$ be a holomorphic complex Morse function, with distinct critical values. Let $\o$ be the usual Kaehler symplectic form on $\C^{n+1}$, and $\o_b$ the one on the base $\C$. Let $c>0$ be any positive constant. Then
 \bq
 \o' = \o + c \pi^{\ast} \o_b
 \eq
 is also a symplectic form. Moreover, restricted to any smooth fibre of $\pi$, it agrees with $\o$.
 \end{claim}
 
 \begin{proof}
 We need to check that $\o'$ is non-degenerate at each point $x \in \C^{n+1}$. If $x$ is a critical point of $\pi$, we simply have $\o'_x = \o_x$. Now suppose  $x \in \C^{n+1}$ is distinct from the critical points.
 The two-form $c\pi^{\ast} \o_b$ vanishes on the `vertical' tangent space at $x$, and gives a symplectic form on the `horizontal' tangent space determined by $\o$. In this case, the horizontal tangent space is just two-dimensional, and $\o$ and $\pi^\ast \o_b$ give area forms with the same sign, which implies the claim about $\o'$.
 \end{proof}

 We shall later use the following technical result.
 
 \begin{lemma}\label{th:changeomega}
 Suppose $f: \C^{n+1} \to \C$ is a singularity, $\tilde{f}$ a Morsification of it, and $\e$ a regular value. 
Fix a vanishing path $\gamma_i$
between $F_\e$ and a critical value $x_s$, and any
  compactly supported Hamiltonian isotopy of the Milnor fibre $F_\e$, say  $\phi_t$. Pick any open set $U$ in the base, intersecting $\gamma_i$ and not containing $x_s$. Then we can modify the symplectic form to get a new symplectic form $\omega'$ such that:
  
  \begin{itemize}
  \item $\omega' = \omega + c\tilde{f}^\ast \o_b$ outside a compact set, where $c \geq 0$ is some constant, and $\o_b$ the standard symplectic form on $\C$. 
  \item $\omega'|_{\tilde{f}^{-1}(a)} = \omega|_{\tilde{f}^{-1}(a)}$ for all regular values $a$.
  \item The parallel transport on $\gamma_i$ from $\e$ to any point past $U$ (using $\omega'$)  is obtained by pre-composing the parallel transport for $\omega$ with $\phi_1$.
  
  \end{itemize}

 \end{lemma}
 
 \begin{proof} 
 
 This is essentially Lemma 15.3 of \cite{Seidel08}.
 Assume without loss of generality that $U$ is a disc centered on a point of $\gamma_i$, and does not contain a singular value. 
  
 Let $U_{1/2}$ be a disc with same centre, and half radius. 
  Let $H(t): F_\e \to \R$ be a smooth family of functions whose Hamiltonian flow is $\phi_t$. We assume that it has compact support. Using a suitable chosen smooth family of paths (including $\gamma_i$), this gives a Hamiltonian function on each fibre above $U$. Let $t$ be the coordinate along $\gamma_i$, smoothly extended to $U$. Pick a bump function $\beta$ on $U_{1/2}$ whose integral along $\gamma_i$ is one. Let $\omega_b$ be the symplectic form on the base. For a sufficiently large constant $k$,   
   \bq
   \omega + d(\beta H(t) \cdot dt) + k \tilde{f}^{\ast} \omega_b
    \eq 
    is a suitable symplectic form  on $\tilde{f}^{-1}(U_{1/2})$. 
 \end{proof} 
 
 The following is then immediate.
 
 \begin{corollary}\label{th:movevcycle}
 Changing $\omega$ in ways prescribed by the previous lemma, we can replace any vanishing cycle by a Hamiltonian isotopic Lagrangian sphere.
 \end{corollary}
Whenever we modify a symplectic form in such a way that $\o'-\o$ is exact, and $\o'-\o = c \pi^\ast \o_b$ outside a compact set, we shall refer to the process as an \emph{essentially local change of the symplectic form}.
We shall want to modify the almost-complex structure accordingly.

\begin{lemma}\label{th:changej}
Suppose $\omega'$ is obtained from $\omega$ as in lemma \ref{th:changeomega}, and that $J$ is an $\omega$--compatible almost complex structure such that $\pi:=\tilde{f}$ is $(J,i)$--holomorphic. Then we can find an $\omega'$--compatible almost complex structure $J'$ such that:
\begin{itemize}
\item $J$ and $J'$ agree outside a compact set, which does not contain any critical point. 

\item $J$ and $J'$ agree when restricted to any fibre of $\pi$.

\item $\pi$ is also $(J', i)$--holomorphic.

\end{itemize}

\end{lemma}

\begin{proof}
Decompose the tangent space at any point $x$ in the total space $M$ using symplectic orthogonal complements:
\bq
T_xM  = (\text{ker} \pi ) \oplus  (\text{ker} \pi )^{\perp \omega'}.
\eq
With respect to such a basis for this decomposition, $J'$ must be of the form
\bq
J'  = \begin{pmatrix} 
\ast & 0 \\
0 & i
\end{pmatrix}
\eq
and as $J$ and $J'$ agree on fibres, $\ast$ is uniquely determined. Moreover, the structure $J'$ thus defined has the required properties.
\end{proof}

One can also modify the symplectic form so that it becomes a `product' in the neighbourhood of a fixed fibre. More precisely, we have the following:

\begin{lemma}\label{th:changeomegatoproduct}
Let $\pi: \C^{n+1} \to \C$ be a holomorphic complex Morse function, with distinct critical values. Let $p$ be a regular value of $\pi$, $\Sigma_p$ the fibre above $p$, and $B_r(p) \subset \C$ an open ball of radius $r$ about $p$, whose closure does not contain any critical values. As before, let $\o$ be the usual Kaehler form on $\C^{n+1}$, and $\o_b$ be the one on the base. 
Let $\phi$ be the diffeomorphism 
\begin{equation}
\phi: \, \pi^{-1}(B_r(p)) \to   \Sigma_p \times  B_r(p)
\end{equation}
given by using symplectic parallel transport with respect to $\o$ along straight-line segments starting at $p$. (Note $\phi$ is compatible with projection to $B_r(p)$.) Fix a compact subset $K \subset \Sigma_p$. We claim that we can find a symplectic form $\o'$ on $\C^{n+1}$ such that:
\begin{itemize}
\item on $\phi^{-1} \big( K \times B_{r/2}(p) \big)$, $\o'$ is a `product' symplectic form:
\bq
\o' = \phi^\ast \big( (\o|_{\Sigma_p} , c\, \o_b) \big)
\eq
for some constant $c>0$;
\item $\o' = \o + c\pi^\ast \o_b$ outside a compact set;
\item $\o'$ and $\o$ agree when restricted to any smooth fibre of $\pi$.
\end{itemize}

\end{lemma}

\begin{proof}
Let $\theta_p \in \Omega^1(\Sigma_p)$ be the restriction of the standard one-form $\theta$ to $\Sigma_p$. Fix $q \in B_r(p)$, and let $\Sigma_q$ and $\theta_q$ be defined analogously. Let $\phi_q: \Sigma_q \to \Sigma_p$ be the result of symplectic parallel transport along a straight-line segment. This has the feature that
\begin{equation}
\phi_q^\ast (\theta_p) = \theta_q + d \rho_q
\end{equation}
for some compactly supported smooth function $\rho_q$ on $\Sigma_q$. Varying over $q$, these give a smooth function on $\pi^{-1}(B_r(p))$, with bounded support, say $\rho$.

Let $\psi$ be a bump function on $\C^{n+1}$ with bounded support inside $\pi^{-1} (B_r(p))$, and such that $\psi=1$ on $\phi^{-1} \big( K \times   B_{r/2}(p) \big)$. 
Let $\o^{pr} = \phi^\ast \big( (  \o|_{\Sigma_p} , 0) \big)$ be a degenerate `product' form on $\pi^{-1}(B_r(p))$. We have  $\o^{pr} = d \theta^{pr}$, where $\theta^{pr}$ is defined to be:
\bq
\theta^{pr} : = \phi^\ast \big( ( \theta|_{\Sigma_p}, 0) \big) - d\rho . 
\eq
Now set
\bq
\o' = \o + d \big(\psi (\theta^{pr} - \theta) \big) + c \pi^\ast \o_b
\eq
for some constant $c$ to be determined. This is certainly closed. As $\theta^{pr} - \theta$ vanishes fibre-wise, $\o'$ agrees with $\o$ when restricted to any smooth fibre. Thus we just need to pick a constant $c$ large enough to ensure that the form is non-degenerate at each point.
\end{proof}

Analogously to Lemma \ref{th:changej}, one can show the following.

\begin{lemma}\label{th:changejtoproduct}

Suppose $\omega'$ is obtained from $\omega$ as in Lemma \ref{th:changeomegatoproduct}, and that $J$ is an $\omega$--compatible almost complex structure such that $\pi$ is $(J,i)$--holomorphic. Then we can find an $\omega'$--compatible almost complex structure $J'$ such that:
\begin{itemize}
\item $J$ and $J'$ agree outside a compact set.

\item $J$ and $J'$ agree when restricted to any fibre of $\pi$.

\item $\pi$ is also $(J', i)$--holomorphic.

\end{itemize}
These requirements determine $J'$ uniquely. Moreover, on the set $\phi^{-1} \big( K \times B_{r/2}(p) \big)$, on which $\o'$ is a product, $J'$ will correspond to the product $( J|_{\Sigma_p} , i)$.
\end{lemma}

\subsubsection{Remarks on more general Lefschetz fibrations}\label{sec:moregeneralLefschetzfibrations}
Most of the features described previously in Section \ref{sec:PicardLefschetz} are also found in more general Lefschetz fibrations (we shall only be concerned with some whose total spaces are complex hypersurfaces, or open subsets thereof). One can for instance consult \cite[Chapter III]{Seidel08}. 
 In particular, we will use the concepts of symplectic parallel transport, distinguished collections of vanishing paths and vanishing cycles, and Lefschetz thimbles in such settings. 
They are defined completely analogously, and their properties are similar, with the caveat that vanishing cycles may not generate the homology of a smooth fibre, and that two different singular values may give the same vanishing cycle. 
One can still perform essentially local changes to the symplectic form to realise Hamiltonian isotopies of vanishing cycles, with prescriptions analogous to Lemmas \ref{th:changeomega} or \ref{th:changeomegatoproduct}.

\subsection{Milnor fibres of functions of two variables, following A'Campo}\label{sec:ACampo}

Given a singularity of two variables, A'Campo \cite{ACampo99} presents a way of describing the Milnor fibre of that singularity, together with favourite vanishing cycles and paths. This uses certain totally real deformations of the singularity; such deformations had earlier been considered by A'Campo \cite{ACampo75} and independently Gusein-Zade \cite{GuseinZade74} to calculate intersection forms of singularities. We give a brief outline here, while highlighting some features we shall use.

\subsubsection{Divides and real deformations}

\begin{definition}\cite{ACampo75}
Let $R$ be the disjoint union of $r$ copies of $[0,1]$, and $D_\e \subset \R^2$ the disc of radius $\e$ centred at 0. An $r$--branched divide (`partage') of $D_\e$ is an immersion $\alpha: R \to D_\e$ such that:
\begin{itemize}
\item  $\alpha(\partial R) \subset \partial D_\e$, $\alpha(\mathring{R}) \subset \mathring{D_\e}$, and $\alpha(R)$ is connected.
\item $\alpha(R)$ only has ordinary double points, none of which lie on $\partial D_\e$.
\item A \emph{region} is a connected component of $D_\e \backslash \alpha(R)$ that does not intersect the boundary of $D_\e$. For any two regions $A$ and $B$, we have $\bar{A} \cap \bar{B} = \alpha(I)$, where $I \subset R$ is a connected segment (possibly empty, or a point).
\end{itemize}

\end{definition}

A'Campo considers real deformations of certain singularities whose zero--loci give divides.

\begin{proposition}\cite[Th\'eor\`eme 1]{ACampo75}\label{th:gooddivides}
Let $f(x,y)$ be a polynomial such that
\begin{itemize}
\item $f(0)=0$, and $f$ has an isolated singularity at the origin.
\item $f$ decomposes into a product of irreducible factors at $0$ (over $\C$), each of which are polynomials with real coefficients.
\end{itemize}
Then $f$ has a real polynomial deformation $f(x,y; t)$, $t \in \R$, such that $f(x,y;0)=f(x,y)$ and for all sufficiently small $t \neq 0$, we have that:
\begin{itemize}
\item The real curve $C_t=\{ (x,y) \, | \, f(x,y;t) =0 \}$ is an $r$--branched divide.
\item  The number $k$ of double points of $C_t$ satisfies $2k-r+1 = \mu$, where $\mu$ is the Milnor number of $f$ at zero.
\end{itemize}
\end{proposition}

Using that fact that $C_t$ is a divide, the second condition is equivalent to:
\bq
\# \text{  (regions of the divide } C_t) + k = \mu. \label{eq:divide}
\eq
What is the significance of this? Let $\tilde{f}(x,y) := f(x,y;1)$. Consider the zero-locus of $\tilde{f}$ in the \emph{real} $x-y$ plane. Each crossing corresponds to a saddle--type critical value for \emph{real} variables, and thus, also, to a critical value for $\tilde{f}$ as a function of complex variables. In each region of the divide, the real function $\tilde{f}$ must attain at least one maximum or minimum. Similarly, these extrema also give critical values of the complex function. Thus to each region corresponds at least one critical point of $\tilde{f}$. On the other hand, because of equation \ref{eq:divide}, it must be that each region actually corresponds to exactly one critical point of $\tilde{f}$. Informally speaking, this real deformation of $f$ ``sees'' a full Milnor--number's worth of critical points. We shall call such deformations \emph{good real deformations}. 

\begin{remark}
While $\tilde{f}$ is non-degenerate, it does not in general have pairwise distinct critical values. If needed, one can remedy this by making a small further perturbation.
\end{remark}

\begin{remark}
The reader might be concerned about the second assumption in Proposition \ref{th:gooddivides}. However, A'Campo notes that for the purposes of understanding the topological type of a singularity (including the construction of Milnor fibres below), it is not actually restrictive. For instance, one can use work of L\^e D\~ung Tr\'ang and Ramanujam \cite{LeRamanujam}, who show that within a smooth one-parameter family of isolated hypersurfaces singularities in two variables with constant Milnor numbers,  the topological type of the singularity does not change.
\end{remark}

\subsubsection{Associating Milnor fibres to divides}

Consider the polynomial 
\bq f(x,y) = xy(x-y)(x+y). \eq
 A good real deformation is given by:
\bq
f(x,y;t)= xy(x-y+2t)(x+y+t).
\eq
Let $\tilde{f}(x,y) = f(x,y;1)$. The associated divide is given in Figure \ref{fig:44divide}. The $+$ regions in the figure correspond to maxima, the $-$ regions to minima, and the double intersection points correspond to saddle points.
\begin{figure}[htb]
\begin{center}
\includegraphics[scale=0.6]{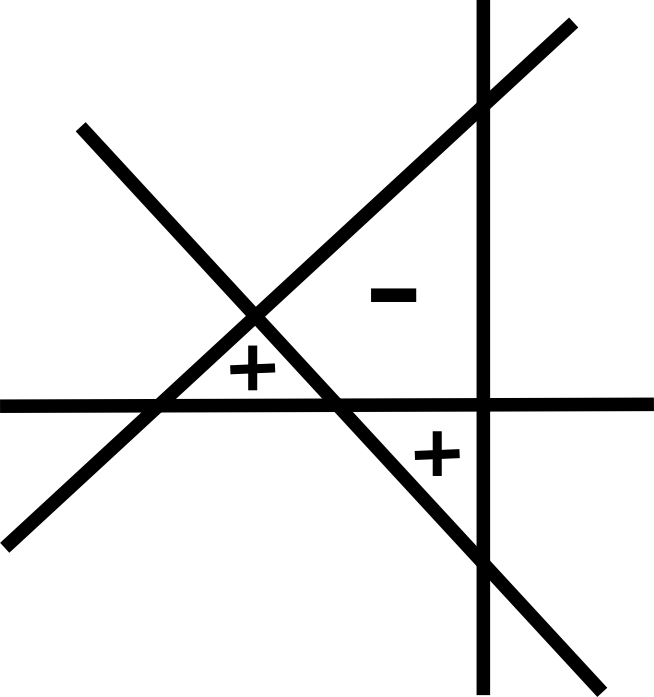}
\caption{A divide for the product of four linear functions
}
\label{fig:44divide}
\end{center}
\end{figure}
Given a divide, A'Campo algorithmically associates to it an oriented (topological) Riemann surface $S$, which is the Milnor fibre of $f$. (See also Example 1 in \cite{ACampo75}, which explains how to get a Dynkin diagram for the intersection form from this example.) Consider the divide as a planar graph, and proceed as follows:
\begin{itemize}
\item To each edge of the graph, associate a ribbon--like strip with one half twist.
\item Replace each of the intersections of the graph by a cylinder, embedded into $\R^3$ with four half twists, as inside the dotted regions in  Figure \ref{fig:44embedded}, which continues our example. Given this embedding, the projection to $\R^2$ given in the figure presents the cylinder with four outer boundary segments. To these we attached the ends of the incoming ribbons (corresponding to edges), of which there are up to four.  This is also performed in Figure \ref{fig:44embedded}.
\end{itemize}
Note that changing the direction of a half twist alters the embedding into $\R^3$, but not the actual surface.

The surface $S$ has Euler characteristic $1-\mu$, i.e.
\bq
\chi (S) = 1- \# \text{  (regions of the divide }) - k 
\eq

\begin{figure}[htb]\label{fig:fibre44}
\begin{center}
\includegraphics[scale=0.7]{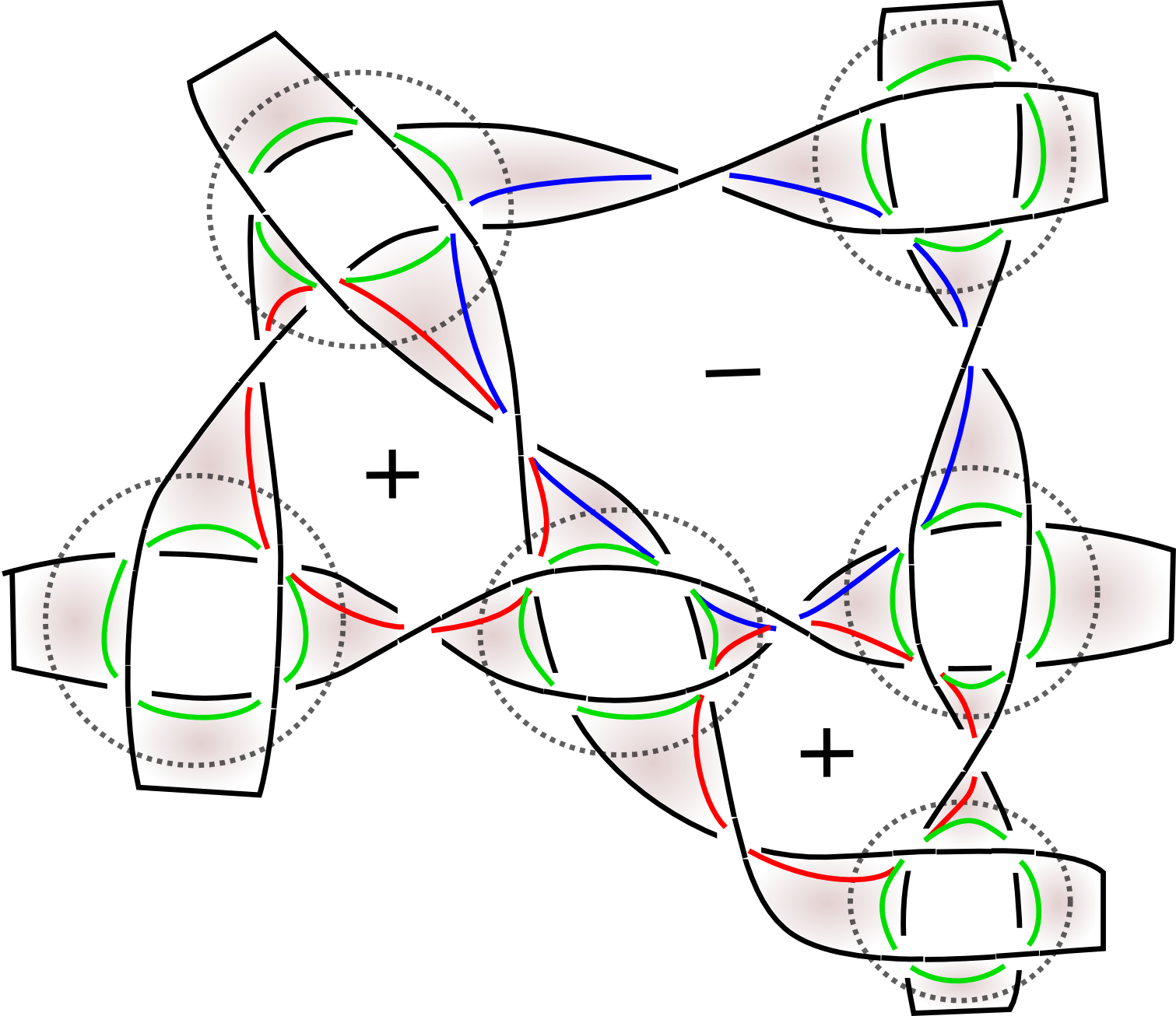}
\caption{Milnor fibre of $x^4 + y^4$. Cylinders associated to crossings are inside the dotted circles. The vanishing cycles corresponding to the minimum, saddles, and maxima, are, respectively, in blue, green and red.
}
\label{fig:44embedded}
\end{center}
\end{figure} 

Pick a regular value of the form $-i \eta$, where $\eta \in \R_+$ is sufficiently small. As vanishing paths, take straight lines between $-i \eta$ and the critical values. (To be rigorous, you need to pick a very small Morse perturbation of $\tilde{f}$ such that all the critical values are distinct, and, w.l.o.g., real -- as the are perturbations of real critical values, since by assumption we start with a good real deformation. It will turn out that it does not matter what the resulting orderings are within any of the types of critical values: it only matter that all the minima be first, then saddles, then maxima.) For these choices,  vanishing cycles are given as follows:
\begin{itemize}
\item The cycles corresponding to real saddles are given by meridional curves of the cylinders introduced for the associated double points.
\item The cycles corresponding to real extrema are given by going along the ribbons corresponding to the boundary of the region the extremum lies over.
\end{itemize}
(For our example, see Figure \ref{fig:44embedded}.)

\subsubsection{Properties of A'Campo's construction}

This description gives preferred orientations for each of the vanishing cycles: anticlockwise in the plane that the projection of the surface lives in (e.g., the paper). The Milnor fibre itself carries a natural orientation, as a complex curve. 

\begin{proposition}\cite[p.~4]{ACampo75}
Let $V_1, \ldots, V_\mu$ be the vanishing cycles given by A'Campo's algorithm. We have the following intersection numbers:
\begin{itemize}
\item $V_i \cdot V_j = 1$ if $V_i$ corresponds to a region with a maximum, and $V_j$ to a saddle in the boundary of that region.
\item $V_i \cdot V_j = -1$ if $V_i$ corresponds to a region with a minimum, and $V_j$ to a saddle in the boundary of that region.
\item $V_i \cdot V_j = 1$ if $V_i$ corresponds to a maximum, $V_j$ to a minimum, and the two regions share an edge.
\item $V_i \cdot V_j = 0$ otherwise. (In particular, if $V_i$ and $V_j$ correspond to distinct critical points of the same real type.)
\end{itemize}

\end{proposition}
In practice, one often uses the preferred orientations of the vanishing cycles, together with the signs of the intersection numbers described above, to determine this orientation for the surface given by A'Campo's algorithm.

Now consider the Milnor fibre equipped with its exact symplectic structure. Examining A'Campo's argument, one notices the following.

\begin{proposition}\label{th:minintersection} \label{th:minimalintersection}
After compactly supported Hamiltonian isotopies, we can arrange for the vanishing cycles given by A'Campo's algorithm to intersect minimally. (In this case, in either zero or one point.)
\end{proposition}
(It is now immediately clear that the ordering within each real--type class of singular points does not matter, as none of the corresponding cycles intersect  -- so one can trivially swap them using mutations.)

\begin{remark}
While the algorithm described above is inherently restricted to functions of two variables, there exist some generalisations of A'Campo's work in higher dimensions. For instance, one can relate the flow category of a real Morse function with the directed Fukaya category of its complexification -- see \cite{Johns}.
\end{remark}

\subsection{Distinguished bases for $P(x_0,\ldots, x_n) + x_{n+1}^{d+1}$} \label{sec:Gabrielovcyclic}

\subsubsection{Background}

Suppose you have two singularities, say $P(x)$  and $Q(y)$, with $x = (x_0,$ \ldots, $x_n)$ and $y=(y_0, \ldots y_m)$. 
We will want to use theorems   that describe vanishing cycles for the join singularity, $P(x) + Q(y)$, following Thom and Sebastiani \cite{ThomSebastiani}.  
Let $\mu$ be the Milnor number of $P$, and $\nu$ that of $Q$. 
Suppose you are given Morsifications $\widetilde{P}$ and $\widetilde{Q}$. Let $p_1, \ldots, p_\mu$ be the critical values of $\widetilde{P}$, and $q_1, \ldots, q_\nu$ the critical values of $\widetilde{Q}$.  
Suppose also that you have chosen two regular values, say, respectively, $p_\ast$ and $q_\ast$, and made choices of distinguished collections of vanishing paths between $p_\ast$ and the $p_i$ (respectively, $q_\ast$ and the $q_j$).
 Gabrielov \cite{Gabrielov1} describes how to use this data to construct a distinguished basis of vanishing paths and cycles for $P(x) + Q(y)$, in the smooth category. 
We shall later use such a construction, in the case that $Q$ is simply a function of one variable, in the symplectic category. (When $Q(y)=y^2$, the process of adding $Q$ is known as a \emph{stabilization}; 
the Milnor number is unchanged, and the corresponding operation for vanishing cycles is well-known to experts; see e.g.~\cite[Chapter 3]{Seidel08}.) 
This is also considered by Futaki and Ueda \cite{FutakiUeda}, Section 2. For a more general case, with functions of arbitrarily many variables, the interested reader might consult \cite[Section 6.3]{AKO}.

\subsubsection{Assumptions and vanishing paths}

Let $\widetilde{P}$, $\widetilde{Q}$, $p_i$ and $q_i$ be as in the introductory paragraph. Assume that the sets $\{ p_i \}_{1 \leq i \leq \mu}$ and $\{ q_j \}_{1 \leq j \leq \nu}$ are disjoint. 
Denote the vanishing path from $p_\ast$ to $p_i$ by  $\gamma_i$ ($1 \leq i \leq \mu$), and the vanishing path from 
from $q_\ast$ to $q_j$ by  $\zeta_j$ ($1 \leq j \leq \nu$). 
We make the following (non-restrictive) assumptions:
\begin{assumption}
There exists positive constants $\e$ and $r$ such that 
\begin{itemize}
\item $\gamma_i(t) \in B_\e(p_\ast)$ for all $i$ and $t$.
\item $\zeta_j(t) \in B_r(0)$ for all $j$ and $t$.
\item $B_\e(q_j)$ doesn't intersect any of the $\zeta_k$, for $k \neq j$.
\item There are no other critical values of $\widetilde{P} + \widetilde{Q}$ in $B_{r+\e}(0)$. (The functions $P$ and $Q$ are representatives of germs: they could have critical points -- and values -- away from zero, which their Morsifications $\widetilde{P}$ and $\widetilde{Q}$ would inherit. We are assuming that these remain sufficiently far away.)
\end{itemize}

\end{assumption}
The critical values of $\widetilde{P}+ \widetilde{Q}$ are given by $\tau_{ij} = p_i+q_j$, and our assumption implies that they are distinct. 
Further deform the paths $\zeta_j$ such that:
\begin{itemize}
\item $\zeta_j + p_\ast $ only intersects the path $q_j+\gamma_i$ (where $(q_j+\gamma_i)(t) := q_j+\gamma_i(t)$) at $q_j + p_\ast$;
\item Ordered anti-clockwise, $\zeta_j + p_\ast$ arrives at $q_j +p_\ast$ between the paths $q_j+\gamma_\mu$ and $q_j+ \gamma_1$.
\end{itemize}
(See Figure \ref{fig:joinvpaths}.)
 We define a vanishing path between $p_\ast+ q_\ast$ and $p_i+q_j$ as follows. First, concatenate the paths $p_\ast+\zeta_j$ (from $p_\ast + q_\ast$ to $p_\ast + q_j$) and $\gamma_i+q_j$ (from $p_\ast + q_j$ to $p_i + q_j$). Second, smooth out the results so that they only intersect at their starting points $p_\ast+q_\ast$, and the results still lie in $B_{r+\e}(0)$. These form a distinguished collection of vanishing paths for $P+Q$. See Figure  \ref{fig:joinvpaths}.
 We shall describe vanishing cycles for this collection below. 
 
 \begin{figure}[htb]
\begin{center}
\includegraphics[scale=0.9]{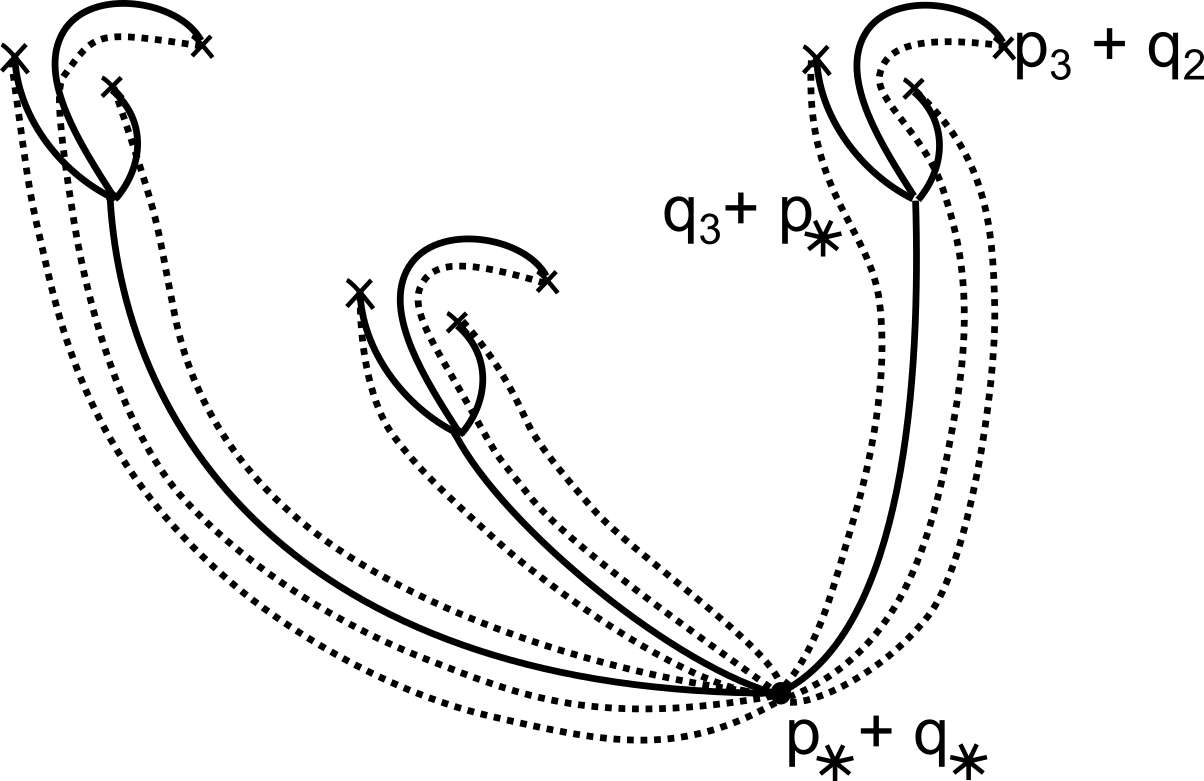}
\caption{ Vanishing paths for $P(x) + y^{d+1}$: concatenations of existing paths (full lines) and smoothings (dotted lines)
}
\label{fig:joinvpaths}
\end{center}
\end{figure}
 
 What if the original data did not satisfy the assumptions we made? Well, we can certainly deform it (in particular, scaling the perturbations of $P$ and $Q$) until it does. Then take the vanishing paths described above, and deform them back.

\subsubsection{Matching paths and vanishing cycles}
We will treat $\widetilde{Q}$ as a branched cover from $\C$ to $\C$ (say $\C_{\mathrm{cover}}$ to $\C_{\mathrm{base}}$), branched over the critical values $q_i \in \C_{\mathrm{base}}$. Let $\bar{q}_i$ denote the corresponding critical points in $\C_{\mathrm{cover}}$: $\widetilde{Q}(\bar{q}_i) = q_i$.

There exists some $\delta > 0$ such that the ball $B_\delta (q_\ast) \subset \C_{\mathrm{base}}$ does not intersect the branch points.  
We assume that we have chosen deformations such that for all $i$, the path $\gamma_i + q_\ast - p_\ast$, from $q_\ast$ to the translated critical value $p_i +q_\ast -p_\ast$, lies entirely in $B_\delta (q_\ast)$. 
Let the pre-images of $q_\ast$ under $\widetilde{Q}$ be $\mathfrak{q}_1, \ldots, \mathfrak{q}_{d+1}$. We have now arranged that for each $j$, each of the paths $\gamma_i + q_\ast - p_\ast$ has a unique, well-defined lift starting at $\mathfrak{q}_j$. Calls its other end-point $\mathfrak{p}_{ij}$. 

Consider the map:
\begin{eqnarray*}
\pi:  \{ (x,y) \, | \, \widetilde{P}(x)+ \widetilde{Q}(y)= p_\ast + q_\ast \} & \to & \C \\
(x,y) \qquad \qquad & \mapsto & y.
\end{eqnarray*}
By starting with the preimage of a small ball about a smooth value (truncated to copies of the Milnore fibre of $P$) and extending, we can find a smooth manifold with corners $V \subset \C^{n+2}$ such that, when restricted to $V$, $\pi$ give a Lefschetz fibration, 
with smooth fibre the Milnor fibre of $P$, exactly $d\mu$ critical points, and critical values $y$ such that 
$$
p_i + \widetilde{Q}(y) = q_\ast + p_\ast
$$
some $i$, $1 \leq i \leq \mu$. 
Moreover, smoothing the corners of $V$ to $V'$, we have that 
\[
 \{ (x,y) \, | \, \widetilde{P}(x)+ \widetilde{Q}(y)= p_\ast + q_\ast \} \cap V'
\]
is a copy of the Milnor fibre of $P+Q$. (This can be understood by deforming $\widetilde{P}$ and $\widetilde{Q}$ back to $P$ and $Q$ and simultaneously deforming $V'$ back to a standard ball, while keeping the Louiville flow on the Milnor fibre transverse to the boundary of the cut-off set.)

Consider the vanishing path $\zeta_j$ from $q_\ast$ to $q_j$.
Now $\zeta_j$ determines a (unique) path through $\bar{q}_j$ between two of the pre-images of $q_\ast$, say $\mathfrak{q}_{j_1}$ and $\mathfrak{q}_{j_2}$. Moreover, each pre-image of $q_\ast$ lies at the end of at most two such paths, and any two paths intersect at most at one end-point. 
See Figure \ref{fig:joingeneral} for an example. 

\begin{figure}[htb]
\begin{center}
\includegraphics[scale=0.9]{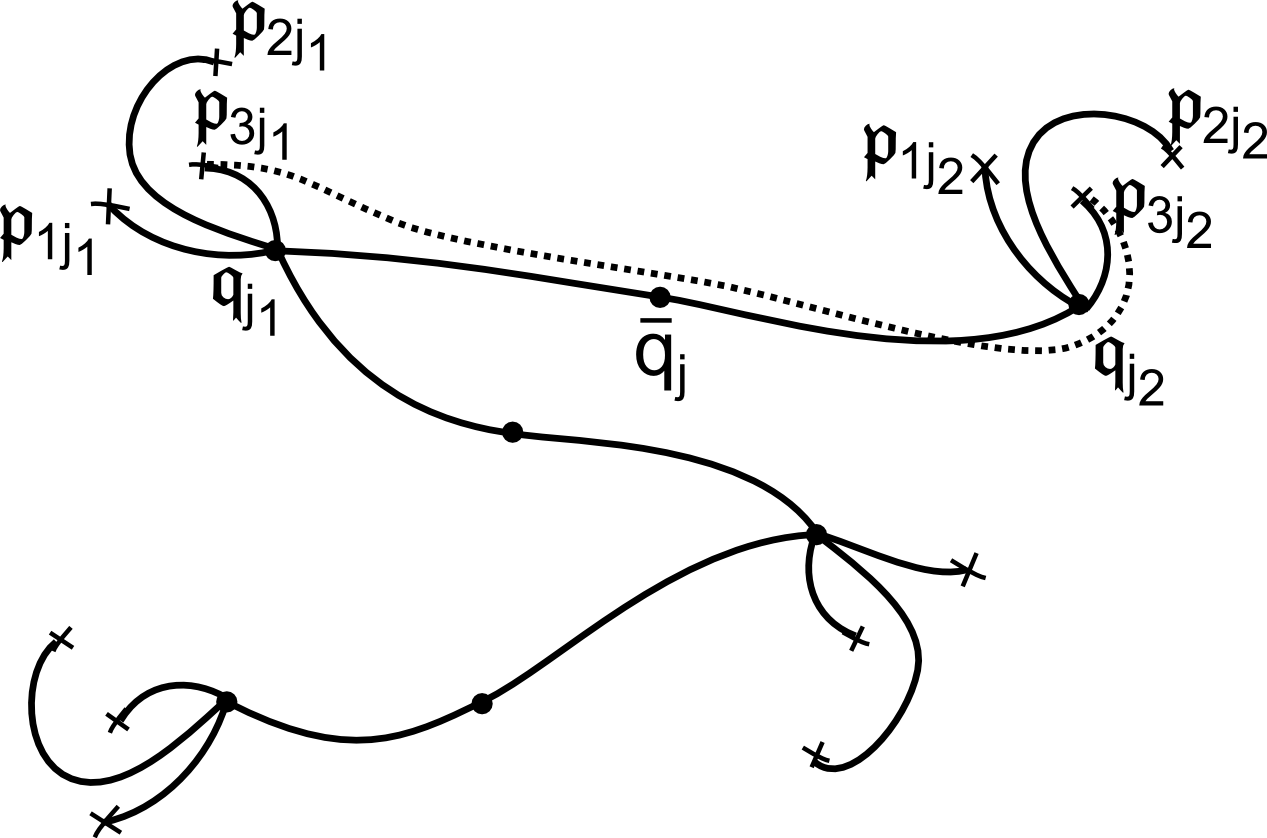}
\caption{ Singular values and matching paths for $\widetilde{P}+ \widetilde{Q}$
}
\label{fig:joingeneral}
\end{center}
\end{figure}

Fix $i$ and $j$, and consider the path from $\mathfrak{p}_{i j_1}$ to $\mathfrak{p}_{i j_2}$ obtained by concatenating the paths we have described from $\mathfrak{p}_{i j_1}$ to $\mathfrak{q}_{j_1}$,  $\mathfrak{q}_{j_1}$ to  $\mathfrak{q}_{j_2}$, and  $\mathfrak{q}_{j_2}$ to  $\mathfrak{p}_{i j_2}$, and making a very small perturbation to get a smooth path. 
(See Figure \ref{fig:joingeneral}; the smoothing should not intersect any of the other critical values.) This is a matching cycle: if you start with the fibre above a point in the interior of the path, the vanishing cycle obtained by deforming the fibre to  $\mathfrak{p}_{i j_1}$ is Hamiltonian isotopic to the vanishing cycle obtained by deforming to  $\mathfrak{p}_{i j_2}$.  
This follows from the construction: the path is obtained by concatenating two paths which are lifts of the same $\gamma_i + q_\ast - p_\ast$ with, in between then a path which is a double cover of a path in the base.
After an essentially local change of the symplectic form (as in Corollary \ref{th:movevcycle}), can arrange for the vanishing cycles for both directions to agree exactly. One can then glue together the corresponding Lefschetz thimbles to obtain Lagrangian spheres in the total space.  The sphere above the path between $\mathfrak{p}_{i j_1}$ and $\mathfrak{p}_{i j_2}$ is a vanishing cycle for $p_i+q_j$, with the vanishing path described in the previous subsection.

\begin{remark}\label{rem:twodescriptions}
 In \cite{Gabrielov1}, Gabrielov uses the projection to $\widetilde{Q}(y)$ -- in general, $y$ consists of multiple variables -- and constructs the vanishing cycle in $\{\widetilde{P}(x)+\widetilde{Q}(y)=p_\ast +q_\ast \}$ by using the product of the cycles for $P$ and $Q$ above a preferred path in the range of $Q(y)$. We only care about the case where $y$ is a single variable, in which case a vanishing cycle for $Q$ is just the union of two points. One can then check that the two descriptions are equivalent. In particular, the equivalence between the two constructions makes it immediate to check that the paths we describe are indeed matching paths.
\end{remark}

\subsubsection{A cyclically symmetric scenario}\label{sec:cyclicsymmetry}
In general, it might be complicated to study the branched cover $\widetilde{Q}: \C \to \C$ to understand what the paths between $\mathfrak{q}_{j_1}$ and $\mathfrak{q}_{j_2}$ are. However, notice that all choices of Morsifications and vanishing cycles for $Q$ are actually equivalent after an isotopy of the base. In particular, one can choose only to think about perturbations of $Q$ of the form 
 \bq
 \widetilde{Q}(y) = y^{d+1} -ky
 \eq
for some constant $k$, with $q_\ast=0$, and, as vanishing paths, straight lines between the origin and the scaled roots of unity. 

 We shall use the following case: suppose that all of the $p_i$ are real, $p_\ast=-i \eta$ (some $\eta \in \R_+$), and vanishing paths are given by straight line segments. (This is the scenario handed to us by A'Campo's construction.) Then
 \bq
  \{ (x,y) \, | \, \widetilde{P}(x)+ y^{d+1}= c \} \cap B_{R_c}
 \eq
is a copy of the Milnor fibre of $P+Q$, for any suitably chosen $c$ and $R_c$. The projection to $y$ has a cyclic symmetry of order $d+1$, and we can arrange for singular values to be positive scalings of the $(d+1)^{th}$ roots of unity. See Figure \ref{fig:joincyclic}.

 \begin{figure}[htb]
\begin{center}
\includegraphics[height=1.6in, width=1.5in,angle=0]{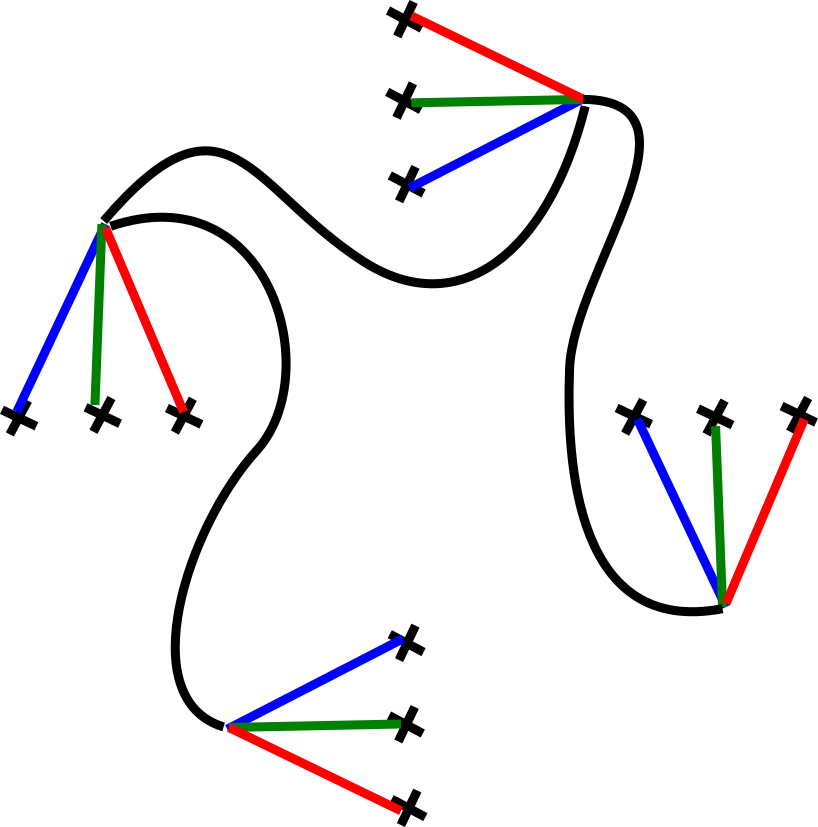} \qquad \quad
\includegraphics[height=1.6in, width=1.5in,angle=0]{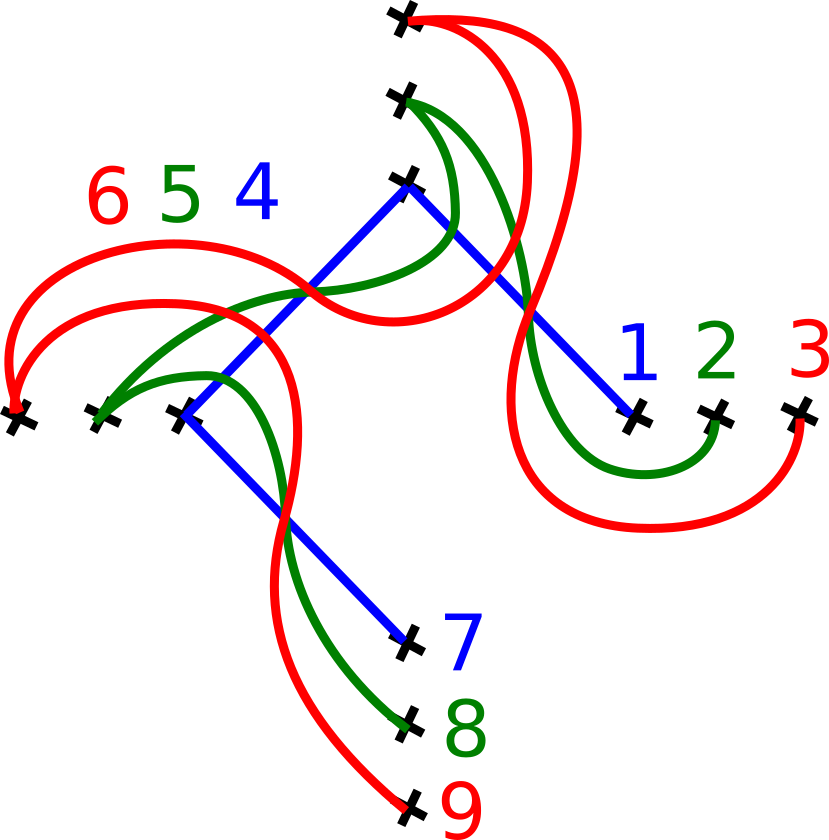}
\caption{Cyclically symmetric scenario: $P(x)+y^{4}$, case $\mu=3$. The left-hand side gives the equivalent of Figure \ref{fig:joingeneral} for this case. The right-hand side gives matching paths and order of the corresponding vanishing cycles.
}
\label{fig:joincyclic}
\end{center}
\end{figure}
 Deforming to $  \{ (x,y) | \widetilde{P}(x)+ y^{d+1}-ky= c' \} $
we get that the paths drawn on that figure are matching paths, and give our distinguished collection of vanishing cycles.

Notice that after performing an isotopy of the base, the desciption of Figure \ref{fig:joingeneral} will always take this form.

 \begin{remark}\label{rmk:changinglexicographicordering}
 The ordering on the $\tau_{ij}$ is a lexicographic ordering induced by those on the $q_j$ and $p_i$. While we usually assign a cyclic ordering to distinguished collections of paths, this construction actually requires an auxiliary absolute ordering for the $p_i$; one can of course interpolate between different choices through series of mutations, e.g.~by Lemma \ref{th:mutationseq}.
 
 The reader might also be wondering about the apparent lack of symmetry between the roles of the $p_i$ and $q_j$. This is an artefact of presentation;
  in particular, it turns out that one can make a series of trivial mutations (the twists involved are between spheres that do not intersect) to get a distinguished basis of vanishing cycles with the other lexicographic ordering. How? This follows from the fact that for $i_1<i_2$, and $j_1 > j_2$, the vanishing cycles assciated to $\tau_{i_1 j_1}$ and $\tau_{i_2 j_2}$ do not intersect. For instance, with the notation of Figure \ref{fig:joincyclic}, we can make the series of trivial mutations:
\begin{multline}
1,2,3,4,5,6,7,8,9 \to 1,4,2,3,5,6,7,8,9 \to 1,4,2,5,3,6,7,8,9 \\
\to 1,4,7,2,5,3,6,8,9 \to 1,4,7,2,5,8,3,6,9
\end{multline}

 \end{remark}

\subsubsection{Dehn twists revisited}\label{sec:matchingmutations}
Suppose you have a Lefschtez fibration with a matching path in the base. The Dehn twist in the corresponding Lagrangian sphere can be described as an automorphism of the fibration: after a Hamiltonian isotopy, the Dehn twist can be assumed to respect the fibration (i.e.~to map fibres to fibres, possibly above different points). The automorphism of the base is given as follows. Let $B_\e$ be a small neighoubrhood of the matching path which does not contain any other critical values (say, all points at distance at most $\e$ from a point in the matching path). Positively rotate $B_{\e/2}$ by a half-twist (so that the two critical values which are the ends of the matching path are swapped),  fix the complement of $B_\e$, and interpolated smoothly between the two on $B_\e \backslash B_{\e/2}$. The automorphisms of each fibre are determined by requiring that the automorphism of the total space be compactly supported; for fibres of points outside $B_\e$, it is the identity, and for the fibre above the mid-point of the matching path (which can be taken to be fixed), it is given by the Dehn twist in the vanishing cycle associated to the matching path. One can prove this using the local model of the standard Lefschetz fibration on the total space of the $A_1$ Milnor fibre (which has two critical points with one matching path between them, and fibre the $A_1$ Milnor fibre with one fewer variable), and e.g.~imposing suitable equivariance conditions, similar to the argument in \cite[Section 6]{Thomas-Yau}. 

We shall make use of the following two  consequences. First, suppose that you have a Lefschetz fibration with two matching paths which intersect at one point, away from the critical values. Let $L$ and $L'$ be the corresponding Lagrangian spheres. Then $\tau_L L'$ can also be described by a matching path; see Figure \ref{fig:matchingmutation} for the local model (compare with Figure \ref{fig:mutation}). Note that the change to the matching path for $L'$ happens in an arbitrarily small neighbourhood of the matching path for $L$. See \cite[Remark 16.14]{Seidel08}.


 \begin{figure}[htb]
\begin{center}
\includegraphics[scale=0.85]{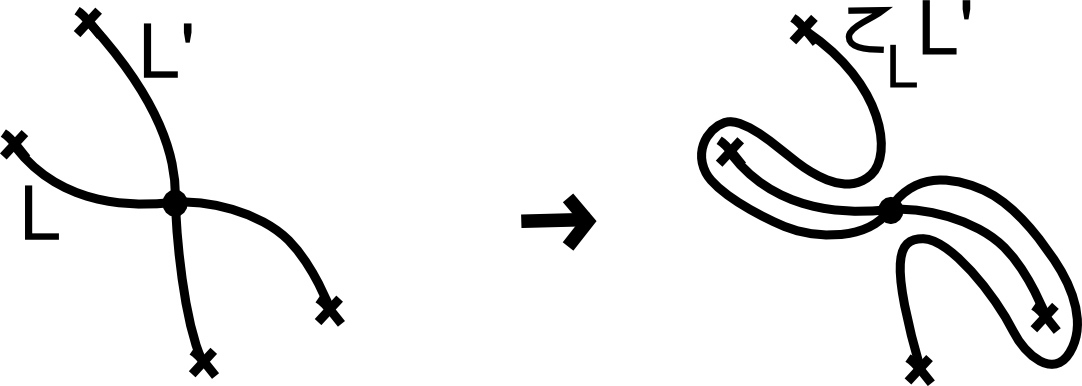}
\caption{Dehn twist of a matching cycle by another: the paths intersect in their interior.
}
\label{fig:matchingmutation}
\end{center}
\end{figure}

Second, suppose instead that the matching paths share a critical value (but do not otherwise intersect). Then $\tau_L L'$ can again be described by a matching path, as in Figure \ref{fig:matchingmutationends}. See \cite[Lemma 16.13]{Seidel08}.

\begin{figure}[htb]
\begin{center}
\includegraphics[scale=0.85]{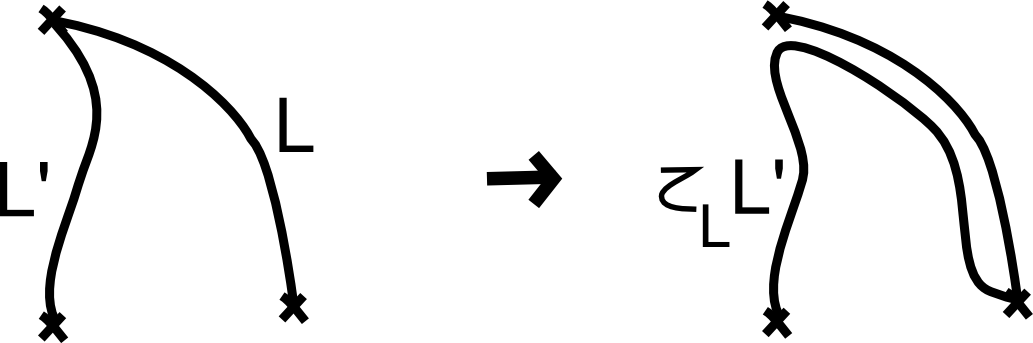}
\caption{Dehn twist of a matching cycle by another: the paths intersect at an end.
}
\label{fig:matchingmutationends}
\end{center}
\end{figure}

\subsection{Some extensions}\label{sec:generalizedextensions}

Many of the results and techniques described in previous parts of this section extend to somewhat wider frameworks. We collect here those extensions that we shall later make use of. These will be used in Section \ref{sec:T333}  where we  describe the Milnor fibre of $T_{3,3,3}$ with vanishing cycles, and in Section \ref{sec:Tpqr}, which considers the general $T_{p,q,r}$.

\subsubsection{Generalized Milnor fibres}\label{sec:generalizedMilnorfibre}
Suppose we have a polynomial $f: \C^{n+1} \to \C$ with finitely many isolated critical points, and distinct critical values. Let $f^{-1} (c)$ be a smooth fibre. For all sufficiently large $R_c$, the sphere $S_{R_c}$ intersects $f^{-1} (c)$ transversally (see e.g.~\cite[Corollary 2.8]{MilnorSingularities}). We shall call this the \emph{generalized Milnor fibre} of $f$. (As with actual Milnor fibres, one could also attach conical ends to get an open exact symplectic manifold.) If $f$ is a Morsification of the representative of a singularity with a single critical point, this is the same as the Milnor fibre. There's a  Lefschetz fibration 
\bq
f: V \subset  \C^{n+1} \to D \subset \C
\eq
for a suitably  chosen manifold with corners $V$, and manifold with boundary $D$ (containing all the critical values), with fibre the generalized Milnor fibre. This can be obtained by first picking $D$ to contain all of the critical values, and the picking $V$ by taking the union over $f^{-1} \cap D_{R_c}$, $c\in D$, noting that we can choose the $R_c$ to be varying smoothly with $c$ and be bounded. We shall consider such generalized Milnor fibres later, starting in Section \ref{sec:T333}. Among other things, they allow us to consider some functions that are not  Morsifications of any singularity. 

\subsubsection{More on A'Campo's techniques}
Suppose we have a polynomial $f$ as above, and that $n=1$. Moreover, suppose $f$ is real, and has a real polynomial deformation $f(x,y;t)$, where $t \in \R$, with the same properties as good real deformations, as described in Proposition \ref{th:gooddivides}: $f(x,y;0)=f(x,y)$, and for all sufficiently small $t\neq 0$, we have that:
\begin{itemize}
\item The real curve $C_t= \{ (x,y) | f(x,y;t)=0 \} $ is an $r$-branched divide.
\item  The number $k$ of double points of $C-t$ satisfies $2k-r+1=\mu$, where $\mu$ is the \emph{total} number of critical points of a Morse perturbation of $f$.
 \end{itemize} 
The reader may wish to keep in mind that often the zero locus of $f$ will itself be a suitable $r$-branched divide, and one can just take the constant deformation.
 A'Campo's techniques extend to understanding how to construct a copy of the generalized Milnor fibre of $f$, together with preferred vanishing paths and corresponding vanishing cycles. The algorithm is the same as previously; one simply needs to check that the construction of Section 3 and proof of Theorem 1 in \cite{ACampo99} do not use the fact that the divide originated from a single isolated singularity.

\subsubsection{Extending the Gabrielov construction to other Lefschetz fibrations}

Suppose $U \subset \C^{n+1}$, and $g: U \to \C$ is a Lefschetz fibration with $\mu$ critical points. (We are not assuming that $g$ is the Morsification of a singularity.) Let $Q$ and $\tilde{Q}$ be as in Section \ref{sec:Gabrielovcyclic}: for instance, $Q(y) = y^{d+1}$, and $\tilde{Q}$ is a Morsification of $Q$.   We can define a new Lefschetz fibration by taking
\begin{eqnarray}
g + \tilde{Q} :   U' \subset \C^{n+2}& \to & \quad \C \\
 (x,y) & \mapsto & g(x) + \tilde{Q}(y)
\end{eqnarray}
for a suitable open $U' \subset \C^{n+2}$, depending on $U$. This is a Lefschetz fibration with $d \mu$ critical points. 
Gabrielov's construction in \cite[pp.184-186]{Gabrielov1}, which we strengthened in Section \ref{sec:Gabrielovcyclic}, does not require $g$ to be a Morsification: it only requires the data of a Lefschetz fibration on an open $U \subset \C^{n+1}$, together with the auxiliary choices of a smooth base-point and a distinguished choice of vanishing paths.
In particular, the descriptions of Section \ref{sec:Gabrielovcyclic} extend to this setting, giving us choices of vanishing paths, and explicit descriptions of the associated cycles, as matching cycles in an auxiliary fibration.  (With care, it seems one could also apply the discussion to even wider collections of Lefschetz fibrations, though we will not need that.)

\section{Background on Floer theory and Fukaya categories} \label{sec:Fukayabackground}

\subsection{What flavour of the Fukaya category do we use?}

We will only consider real four-dimensional exact symplectic manifolds that are \emph{Liouville domains} (compact, contact type boundary), or Liouville domains with an infinite cylindrical end attached to the contact boundary. Moreover, these will all have vanishing $c_1$. (This is true of any Milnor fibre, and more generally of any smooth hypersurface.) We will use variations of the Fukaya category set up in Seidel's book \cite[Chapter II; Section 12]{Seidel08}.

\subsubsection{$\Z/2$--graded version} The objects of the Fukaya category, \emph{Lagrangian branes}, consist of triples $(L, \mathfrak{s}_L, E)$ where:
\begin{itemize}
\item $L$ is a compact, exact, orientable Lagrangian surface (notice any such $L$ will be spin);
\item $\mathfrak{s}_L$ is a spin structure on $L$;
\item $E$ is a flat complex line bundle on $L$.
\end{itemize}
We will see after defining the $A_\infty$--maps that there are redundancies within these choices. As a preliminary, notice the following: Suppose $L$, as above, has genus $g$. Fix an ordered basis for $H_1(L)$ (recall our homology groups have $\Z$--coefficients unless otherwise specified). This determines an isomorphism between the space of flat complex line bundles on $L$ and $(\C^{\ast})^{2g}$, given by holonomies about oriented simple closed curves representing the basis. On the other hand, this choice of basis also determines an identification of the space of spin structures on $L$ and $\{\pm 1\}^{2g}$: any spin structure on $L$ is determined by its restriction to the simple closed curves, which can either be trivial (to which we associated $1$) or non-trivial ($-1$). For more details, the reader could consult e.g.~\cite{Johnson}. Note that there will actually be some redundancy in the data $(L, \mathfrak{s}_L, E)$, as further detailed at the end of this subsection.

The difference with the set-up in \cite{Seidel08} is that we are also using flat complex line bundles. 
We'll explain how to define Floer chain complexes and the $A_\infty$--structure in our case by twisting the maps in \cite{Seidel08}. 
Suppose we have made universal choices of strip-like ends and consistent choices of regular Floer and perturbation data, as in \cite{Seidel08}.
Fix 
compact exact orientable Lagrangians $L_0$ and $L_1$, with spin structures $\mathfrak{s}_0$ and $\mathfrak{s}_1$.  To keep notation simple, assume that $L_0 \pitchfork L_1$, and that no Hamiltonian perturbation is made for that pair. We take our coefficient field to be $\C$. Using Seidel's notation, we have
$$
CF\big( (L_0, \mathfrak{s}_0), (L_1, \mathfrak{s}_1) \big) = \bigoplus_{x \in L_0 \cap L_1} |o_x| \cdot \C.
$$
Each of the pseudo-holomorphic discs $u$ contributing to the differential
give a map
$$
d_u: |o_x| \cdot \C \to |o_y| \cdot \C
$$
for some $x, y \in CF(L_0, L_1)$, and similarly for higher $A_\infty$--products. In the twisted case, we can use whatever regular data choices we have already made for the untwisted case. (In particular, we shall use the same pseudo-holomorphic discs to calculate differentials and $A_\infty$--operations.) Suppose $L_0$ and $L_1$ are further decorated with flat complex line bundles $E_0$ and $E_1$. We define the Floer complex to be:
\bq \label{eq:Floer_differential}
CF\big( (L_0, \mathfrak{s}_0, E_0) , (L_1, \mathfrak{s}_1, E_1) \big)  = \bigoplus_{x \in L_0 \cap L_1} |o_x| \otimes Hom (E_0 |_x, E_1|x).
\eq
To define $A_\infty$--products, we need to use the parallel transport maps associated to the $E_i$: 
$$
\pi_i^{\partial u}: E_i |_x \mapsto E_i |_y
$$
for $i=0,1$. The differential is given by:
$$
d(|o_x| \otimes \phi) = \sum_{u} d_u |o_x| \otimes \pi^{\partial u}_1 \circ \phi \circ 
(\pi^{\partial u}_0)^{-1}.
$$
The higher $A_\infty$--maps are defined analogously. 
To see that we do indeed get well-defined operations, and an $A_\infty$--category, one could go through the construction in \cite{Seidel08} and decorate the Banach bundles and operators that appear with the appropriate twists. 
Alternatively, notice the following: in the untwisted case, each of the $A_\infty$--relations, including $d^2=0$, holds homotopy class by homotopy class: each $A_\infty$--relation is obtained by considering the boundary of a moduli space of holomorphic discs with boundary conditions on a collection of Lagrangians, and there is a different component for each homotopy class of discs with such boundary conditions. 

Finally, an observation: if you start with a brane $(L, \mathfrak{s}_L, E)$, modifying $E$ by the action of an element of $\{  \pm 1 \} ^{2g} \subset (\C^{\ast})^{2g}$ is equivalent to keeping $E$ fixed and modifying the spin structure.
Essentially, this is because of the tensor products on the right-hand side of Equation \ref{eq:Floer_differential}: a change of sign in the orientation $|o_x|$ is equivalent to a change of sign in $Hom(E_0|_x, E_1|_x)$. 
 When carrying out computations later, we shall fix a basis for $H_1(L)$, assume that we are using the corresponding favourite spin structure (i.e.~the one that restricts to the trivial spin structure for each curve in the basis; we will suppress it from the notation), and specify $E$ by giving monodromy with respect to this basis.

\subsubsection{Absolutely $\Z$--graded version} 

We are considering exact symplectic manifolds $M$ with $c_1= 0$. Choose a lift
\bq
LGr(M) \to \widetilde{LGr}(M)
\eq
where $LGr(M)$ is the Lagrangian Grassmanian of $M$, and $\widetilde{LGr}(M)$ its universal cover. If $L$ is a Lagrangian with vanishing Maslov class, then we can find a consistent lift of the tangent planes to $L$. We denote such a lift by $\widetilde{L}$. 

Consider the category with objects given by triples $( \widetilde{L}, \mathfrak{s}_L , E)$, where:
\begin{itemize}
\item $L$ is a compact, exact, orientable Lagrangian of Maslov class zero, and $\widetilde{L}$ as above;
\item $\mathfrak{s}_L$ is a spin structure on $L$;
\item $E$ is a flat complex line bundle on $L$.
\end{itemize}
Again, one can define morphisms spaces and $A_\infty$--operations by starting with the set-up of \cite{Seidel08}, and twisting morphisms spaces and $A_\infty$--operations by the contributions of the flat line bundles. We get an $A_\infty$--category with an absolute $\Z$--grading, which recovers the $\Z_2$--grading of the previous section. 

Given two graded Lagrangians $\widetilde{L}_0$ and $\widetilde{L}_1$, and transverse intersection point $x$ naturally comes equipped with an absolute grading, which we denote by $\widetilde{I}(\widetilde{L}_0, \widetilde{L}_1, x)$, following \cite[Section 2d]{Seidel00}.

\subsection{Generation and split-generation}

The Fukaya category $\Fuk (M)$ with either gradings is an $A_\infty$ category. One can extend this to the category of twisted complexes, $Tw \Fuk (M)$; see e.g~\cite{Hasegawa} for a detailed account.
We shall use the following two definitions:

\begin{definition}
A collection of objects $B_1, \ldots B_m$ in $\Fuk(M)$ are said to \emph{generate} $\Fuk (M)$ if in $Tw \Fuk(M)$, every object of $\Fuk(M)$ is quasi-isomorphic to a twisted complex built from copies of $B_1, \ldots, B_m$. This means that every object in $\Fuk(M)$ can be obtained from $B_1, \ldots, B_m$ by using iterated mapping cones, and, if gradings are involved, grading shifts.

The collection $B_1, \ldots, B_m$ is said to  \emph{split-generate} $ \Fuk (M)$  if every object of $\Fuk(M)$ is quasi-isomorphic to a direct summand of a twisted complex built from copies of $B_1, \ldots, B_m$.

\end{definition}

We shall use the following:
\begin{lemma}\label{th:nosplitgeneration}
Suppose that $L \in \text{Ob } \Fuk(M)$ is such that 
\begin{itemize}
\item $HF(L,L) \neq 0$, and
\item $HF(L,B_i) =0$ for all $i=1, \ldots, m$.
\end{itemize}
Then the collection $B_1, \ldots, B_m$ cannot generate or even split-generate $\Fuk(M)$. 
\end{lemma}

Seidel studied the Fukaya category of Milnor fibres of weighted homogeneous singularities. 

\begin{theorem}\label{th:Seidelweighted}[Seidel]
Let $f$ be a weighted homogeneous polynomial with weights $w_1, \ldots, w_d$, such that
$$
1/w_1 + \ldots + 1/w_d \neq 1.
$$
Then the absolutely graded Fukaya category of the Milnor fibre of $f$ is generated by vanishing cycles.
\end{theorem}

\begin{proof}
This follows from combining \cite[Section 4c]{Seidel00} and \cite[Proposition 18.17]{Seidel08}: the first of these shows that a power of the global monodromy of $f$ acts by a non-trivial grading shift on the Fukaya category; the second gives generation in such circumstances.
The results are stated for the Fukaya category with objects Lagrangian branes consisting of an absolutely graded Lagrangian submanifold together with a $Pin$ structure. (For our set-up, this corresponds to only allowing line bundles with monodromies of the form $\{ \pm 1 \} ^{2g} \subset (\C^\ast)^{2g}$ -- we only consider orientable Lagrangian surfaces, so we think about spin structures instead of $Pin$ structures.) However, the proofs extend to the case of Lagrangian submanifolds decorated with flat complex line bundles.
\end{proof}

\subsection{Convexity arguments}

In the descriptions of Section \ref{sec:vcycles}, we will usually present the Milnor fibre of $T_{p,q,r}$ as an open subset of a larger exact symplectic manifold. To calculate $A_\infty$--products between compact Lagrangians inside the Milnor fibre, we can use the following:

\begin{lem}{(Abouzaid, see e.g.~\cite{AbouzaidSeidel} or \cite[Lemma 7.5]{Seidel08})}
   Suppose $(N, \omega_N, \theta_N)$ is an arbitrary exact symplectic manifold, $(U, \omega_U, \theta_U)$ an exact symplectic manifold of the same dimension with contact type boundary, and $\iota: U \hookrightarrow$ int($N$) an exact sympletic embedding. Suppose that we also have an $\omega_N$-compatible almost complex structure $J$, whose restriction to $U$ is of contact type near $\partial U$.
Let $R$ be a compact connected Riemann surface with boundary. Suppose $u: R \to N$ a $J$-holomorphic map such that $u(\partial R) \subset$ int($U$). Then we have that $u(R) \subset$ int($U$) as well.
Moreover, this only requires $u$ to be $U$-holomorphic in a neighbourhood of $\partial U$.
\end{lem}


\section{Vanishing cycles for $T_{p,q,r}$} \label{sec:Tpqrvcycles} \label{sec:vcycles}

\subsection{Case of $T_{p,q,2}$}\label{sec:Tpq2}

We shall take the following defining equation for $T_{p,q,2}$ (given in Section \ref{sec:unimodalreps}):
$$
\tau_{p,q,2} (x,y,z)=  (x^{p-2}-y^2)(x^2-\lambda y^{q-2}) +z^2
$$
where $\lambda \in \C$ is any constant such that the polynomial has an isolated singularity at zero. (In particular, we can choose $\lambda =1$, except when $p=q=4$, in which case $\lambda=2$ works.)

We can choose a real deformation $m_p(x,y)$ of $x^{p-2}-y^2$ and a real deformation $n_q(x,y)$ of $x^2-y^{q-2}$ such that 
$$
h_{p,q} (x,y) := m_p(x,y) n_q(x,y)
$$
has only non-degenerate critical points, and is a good real deformation of $\sigma_{p,q}(x,y):=(x^{p-2}-y^2)(x^2-y^{q-2})$, in the sense of section  \ref{sec:ACampo}, with divide given by Figure \ref{fig:Tpq2Morse}. (This was also considered by Guse\u \i n- Zade -- see \cite[Figure 3]{GuseinZade74}.)

\begin{figure}[htb]
\begin{center}
\includegraphics[scale=0.85]{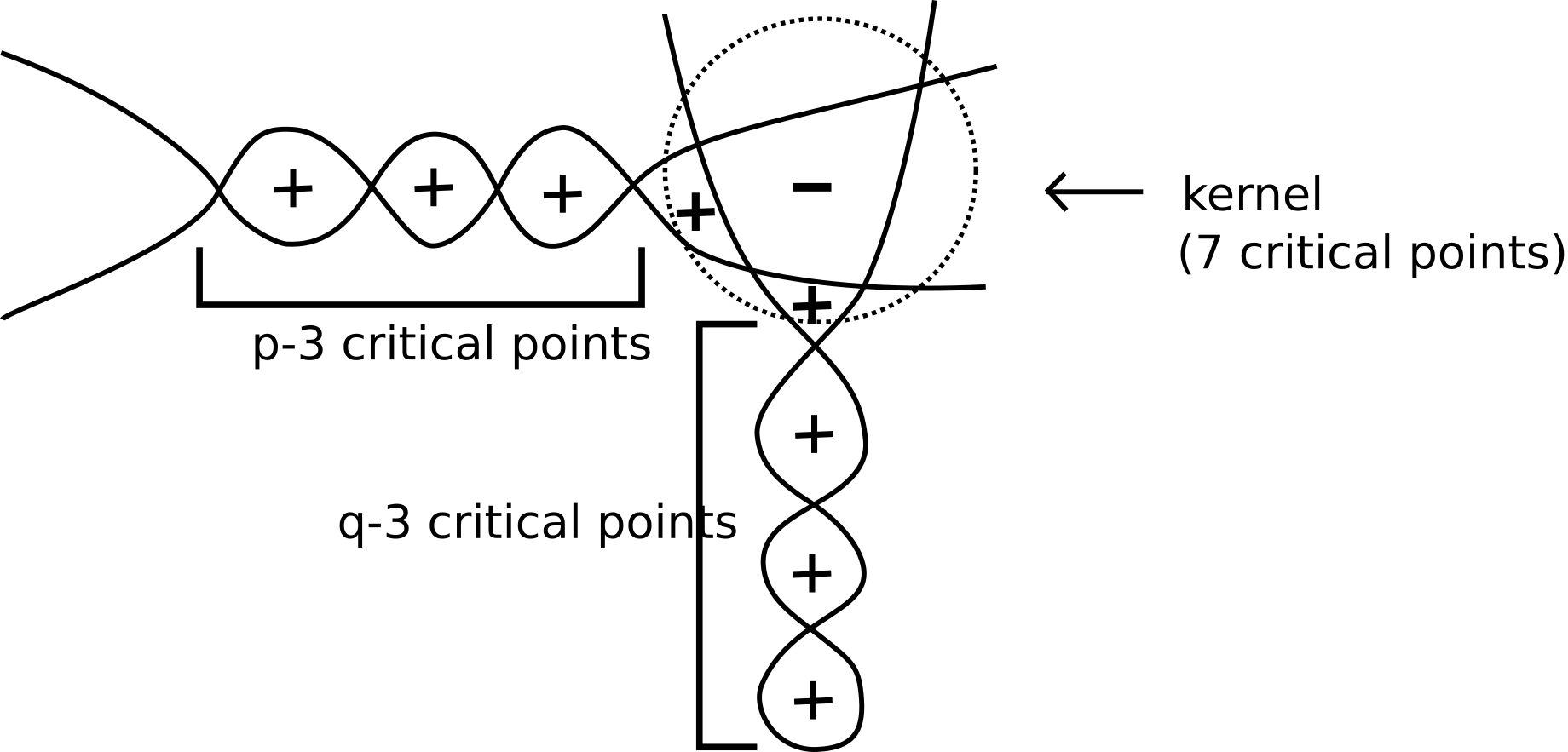}
\caption{The divide associated to $h_{p,q}$. Whether each branch has one or two strands depends on the parity of $p$ or $q$.}

\label{fig:Tpq2Morse}
\end{center}
\end{figure}

Regardless of the values of $p$ and $q$, we can arrange for the two branches to intersect so as to give a fixed `kernel' of seven critical points; the remainder of the critical points are arranged along two $(A_m)$--like chains. 

Using A'Campo's algorithm, we associate to this divide a copy of the Milnor fibre of $\sigma$, with vanishing paths and cycles. See Figures \ref{fig:45embedded} and \ref{fig:RS245ACampo} 
for the case where $p=4$ and $q=5$. The seven coloured cycles (labelled $a$ through $g$) correspond to the kernel; the three grey and black ones (labelled $p_3$, $q_3$ and $q_4$), to the two chains of length $p-3$ and $q-3$. The orientation is as follows: the face closest to the reader is oriented anti-clockwise. (Deleting $q_4$ and considering the remaining configuration of curves on a four-punctured genus three surface, one precisely gets the configuration of Figure \ref{fig:fibre44}.)

\begin{figure}[htb]
\begin{center}
\includegraphics[scale=0.60]{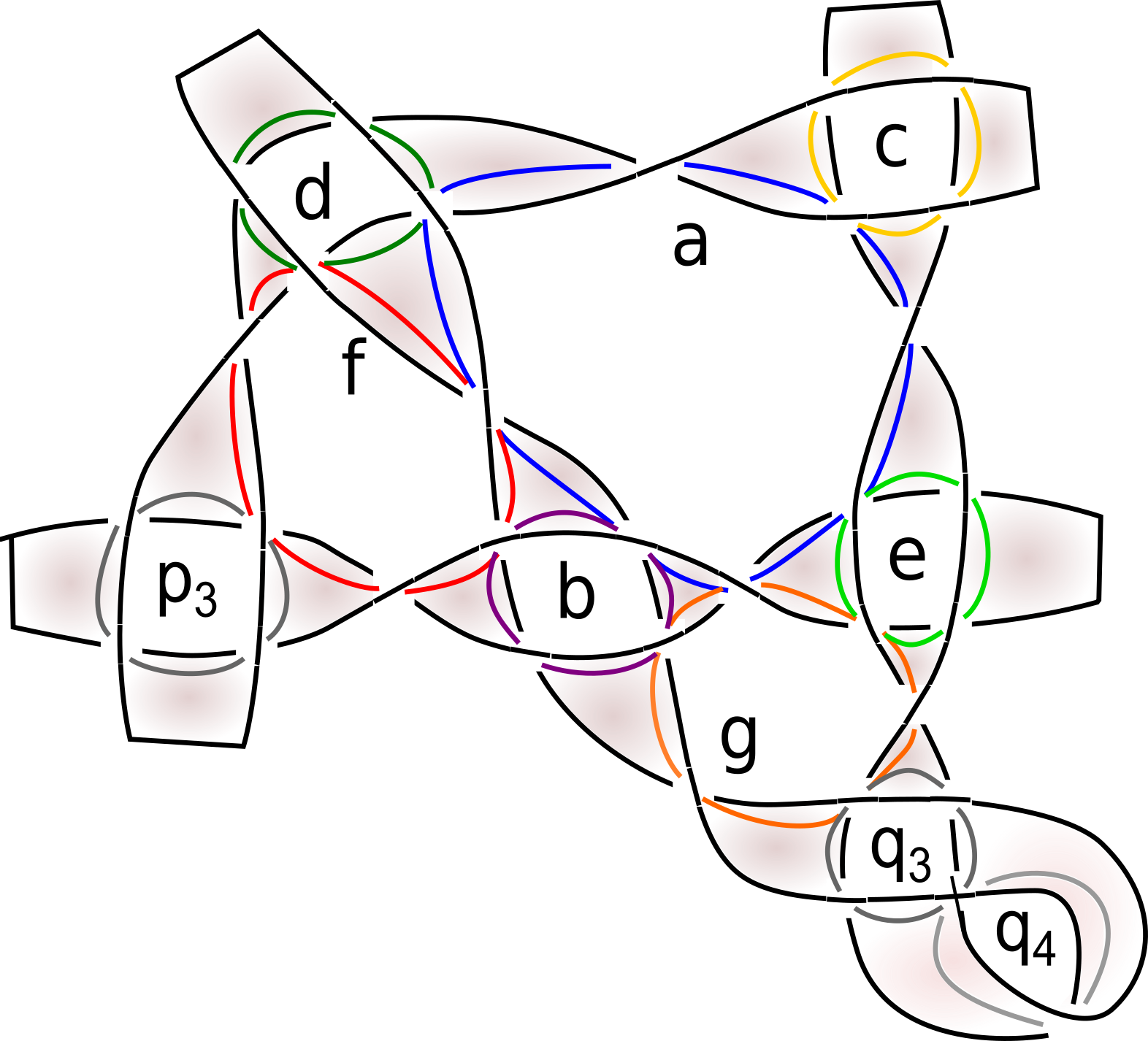}
\caption{Milnor fibre of $\sigma_{4,5}$, with the embedding coming from the A'Campo algorithm.}
\label{fig:45embedded}
\end{center}
\end{figure}

\begin{figure}[htb]
\begin{center}
\includegraphics[scale=0.85]{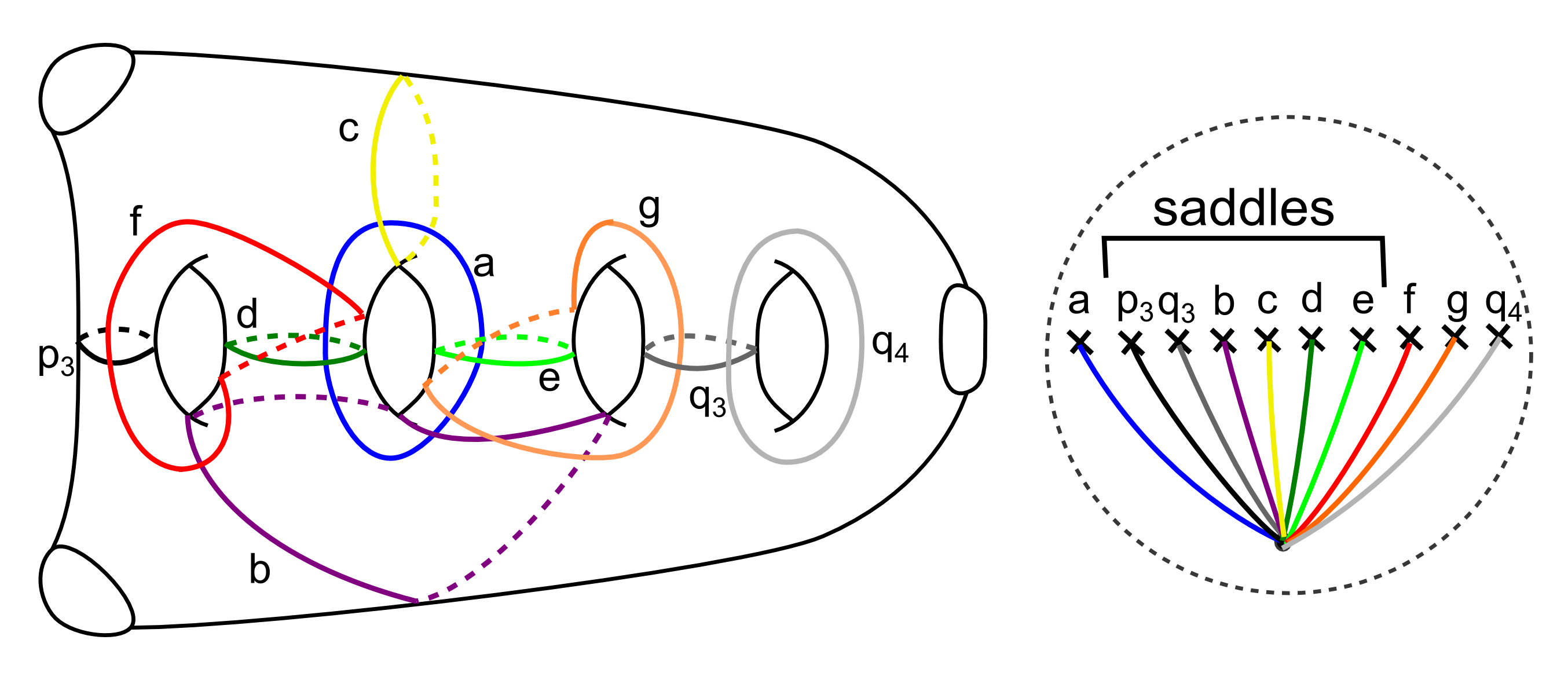}
\caption{Milnor fibre of $\sigma_{4,5}$, vanishing paths and cycles.}
\label{fig:RS245ACampo}
\end{center}
\end{figure}

\subsubsection{Mutations on the Milnor fibre of $\sigma_{p,q}$}\label{sec:spqmutations}\label{sec:Tpq2mutations}
We perform mutations on the Milnor fibre of $\sigma_{p,q}$. As per Proposition \ref{th:minimalintersection}, after an essentially local change to the symplectic form, we assume that to start with, the vanishing cycles intersect minimally. Moreover, after each mutation, we can also arrange for the new vanishing cycles cycles to intersect minimally  after Hamiltonian isotopy. (This is easy to check by hand as one performs the mutations, both in this section and in later ones, in which we find descriptions of more general $\mathcal{T}_{p,q,r}$. We will hereafter assume we do so.\footnote{It seems one might hope to prove a more general result about vanishing cycles for a two-dimensional Milnor fibre and minimal intersections, perhaps using hyperbolic geometry.}) Performing essentially local changes of the symplectic form as needed, we shall always assume that this is the case.

We shall include trivial mutations (involving vanishing cycles that do not intersect) for the case of $p=4$ and $q=5$. These are described in parenthesis. The reader might prefer to think of them as permissible re-orderings of the vanishing cycles. We make the following mutations. 
\begin{eqnarray}
f \mapsto ( f^1 =  \tau_e f  )\mapsto  f^2 =\tau_d f^1 
\\
g \mapsto g^1 = \tau_e g  
\\ 
b \mapsto  ( b^1 = \tau_{q_3} b)
\mapsto (  b^2 = \tau_{p_3} b^1)
\mapsto  b^3 = \tau_a b^2
\mapsto ( b^4 = \tau_{q_4} b^3)
  \mapsto  b^5 = \tau_e b^4\\
b^5  \mapsto (  b^6 = \tau_{g^1} b^5)
   \mapsto b^7 =  \tau_d b^6
   \mapsto ( b^8 = \tau_{f^2} b^7)
     \mapsto  b^9 = \tau_c b^8 
\end{eqnarray}

Note that all the effective (that is, non-trivial) mutations just involve the seven cycles in the kernel. For different values of $p$ and $q$, there will be trivial mutations following the same pattern as above. 

Relabelling, the resulting collection of vanishing paths and cycles is given by Figure \ref{fig:RS245Mutated}. We only keep track of the order of the vanishing paths, and not the exact trajectories. (After isotopy of the base, this is the only data that matters.) 
\begin{figure}[htb]
\begin{center}
\includegraphics[scale=0.85]{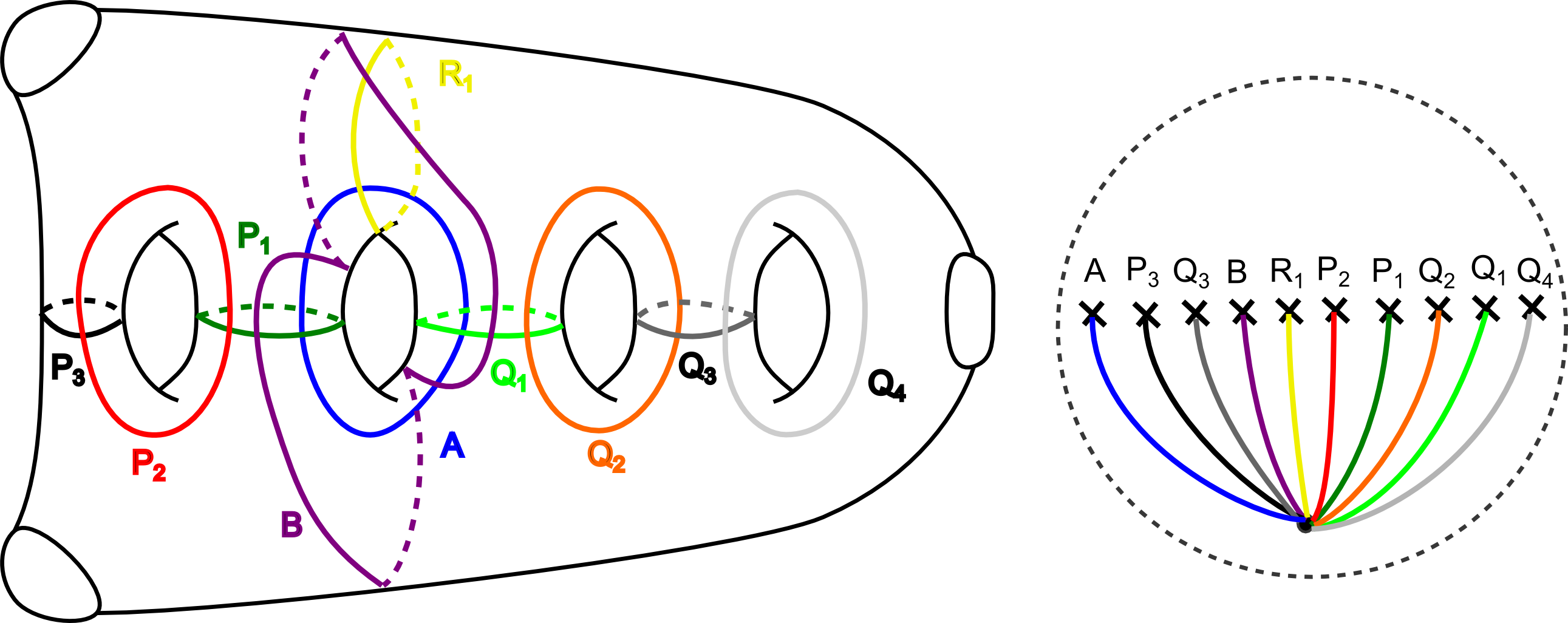}
\caption{Milnor fibre of $\sigma_{4,5}$, vanishing paths and cycles after some mutations.}
\label{fig:RS245Mutated}
\end{center}
\end{figure}
One can then make a sequence of trivial mutations, for instance as follows, to get the ordering of vanishing cycles given in Figure \ref{fig:RS245MutatedBase}. 

\begin{eqnarray}
Q_1 \mapsto \tau^{-1}_{Q_4} Q_1 \\
P_1 \mapsto \tau^{-1}_{Q_4} \tau^{-1}_{Q_2} P_1 \\
P_3 \mapsto \tau_{Q_2}\tau_{Q_4} \tau_{P_1}\tau_{Q_1}\tau_A P_3 \\
Q_3 \mapsto \tau_{P_1}\tau_{Q_1}\tau_A Q_3 \\
Q_4 \mapsto \tau_{P_1}\tau_{Q_1}\tau_{A}\tau_{B}
\tau_{R_1}\tau_{P_2}\tau_{P_3}\tau_{Q_2} Q_4 \\
R_1 \mapsto \tau^{-1}_{Q_1}\tau^{-1}_{P_1}\tau^{-1}_{Q_4}\tau^{-1}_{Q_3}
\tau^{-1}_{Q_2}\tau^{-1}_{P_3}\tau^{-1}_{P_2} R_1
\end{eqnarray}

\begin{figure}[htb]
\begin{center}
\includegraphics[scale=0.85]{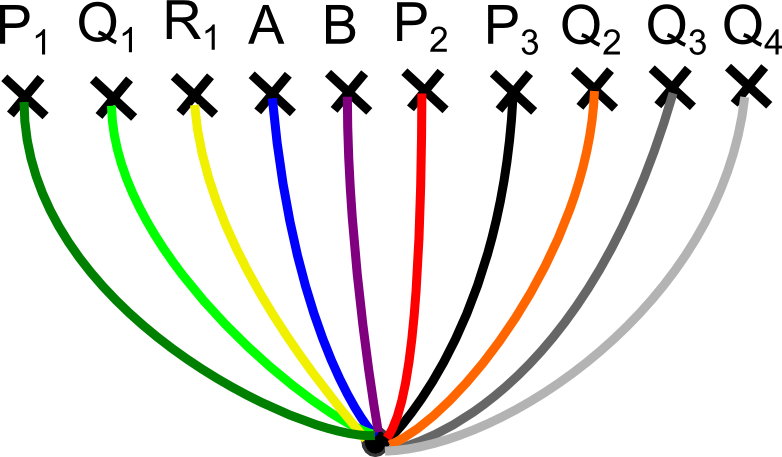}
\caption{Vanishing paths for $\sigma_{4,5}$ after further mutations (all trivial).}
\label{fig:RS245MutatedBase}
\end{center}
\end{figure}

For different values of $p$ and $q$, one can proceed similarly. We label cycles as suggested by Figure \ref{fig:RS245Mutated}. In general the indices of the $P_i$ run from $1$ to $p-1$, and those of the $Q_i$ from 1 to $q-1$. For all of these cases, we have $r=2$, and there is only $R_1$.

\subsubsection{Stabilization to $T_{p,q,2}$}\label{sec:Tpq2vcycles}
We now use Section \ref{sec:Gabrielovcyclic} to get a description of the vanishing cycles for $\tau_{p,q,2} (x,y,z) = \sigma_{p,q}(x,y) + z^2$. Use the Morsification of $\sigma_{p,q}$ to get a Morsification of $\tau_{p,q,2}$. Under the projection $\pi$ of a Milnor fibre to the $z$ coordinate, the resulting vanishing cycles are given by the matching paths of Figure \ref{fig:245MutatedBase}.
\begin{figure}[htb]
\begin{center}
\includegraphics[scale=1.2]{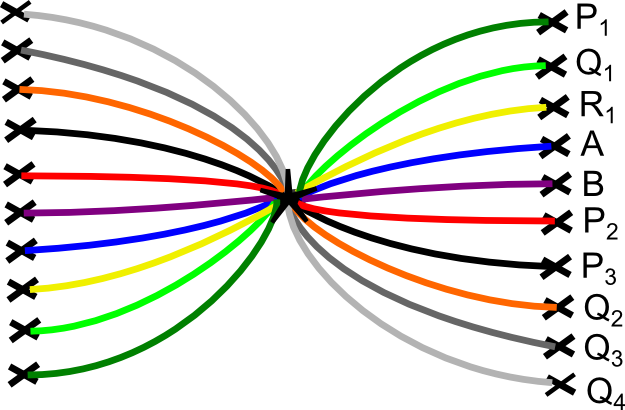}
\caption{Matching paths for the Milnor fibre of $\tau_{4,5,2}$, after isotopy.}
\label{fig:245MutatedBase}
\end{center}
\end{figure}
Their order is the same as the order for the vanishing cycles of $\sigma_{p,q}$, and we shall use the same notation. After isotopy, we may assume that the matching paths all intersect in exactly one point, say $\star$. The fibre above that point, $\pi^{-1} (\star)$, is the Milnor fibre of $\sigma_{p,q}$, and each vanishing cycle restricts to the corresponding simple closed curve on $\pi^{-1} (\star)$ (as on Figure  \ref{fig:RS245Mutated}). In particular, our cycles precisely give the intersection form described by Gabrielov (Figure \ref{fig:Tpqr'}).

\subsection{Case of $T_{3,3,3}$}\label{sec:T333}

\subsubsection{Overview and strategy}

The key will be the following construction:

\begin{proposition}

There exists a one parameter family of polynomial maps $M(x,y,z;t)$, $t \in [0,1]$, such that:
\begin{itemize}
\item $M( x,y,z ;t): \C^3 \to \C$ has distinct non degenerate critical points for all $t$.
\item $M(x,y,z;t)$, for all $t \in [0,1)$, has 14 critical points.
\item $M(x,y,z;1)$ is a deformation of $T_{3,3,3}$, and has eight critical points. (The six other critical values have increasing large absolute values as $t \to 1$, and have `escaped' to infinity at $t=1$.)
\item $M(x,y,z;0)$ is of the form 
$$
\check{m}(x,y) + Q(z)
$$
where $Q$ is a Morsification of $z^3$, and $\check{m}(x,y)$ is a real polynomial with seven critical points, whose real zero locus gives a divide satisfying the conditions used in section \ref{sec:ACampo}. 
\end{itemize}

\end{proposition}

This family will be constructed further down. Given this, we  proceed in several steps. We start by combining the techniques of Sections \ref{sec:ACampo} and \ref{sec:Gabrielovcyclic} to give an explicit description of vanishing paths and cycles in the generalized Milnor fibre of $\check{m}+Q$; the vanishing cycles are given by matching paths in an auxiliary Lefschetz fibration.

For all $t<1$, the generalized Milnor fibres (as defined in Section \ref{sec:generalizedMilnorfibre}) of $M(x,y,z;t)$ are all exact symplectomorphic to one another.  Moreover, 
we get an exact embedding of the Milnor fibre of $T_{3,3,3}$, say $\mathcal{T}_{3,3,3}$, into this space. 
In particular, the Floer cohomology (or more generally, higher $A_\infty$--products) between any two objects in the Fukaya category of $\mathcal{T}_{3,3,3}$ can be computed inside the generalized Milnor fibre of $\check{m}+Q$ instead. 

We want to realise a distinguished collection of vanishing cycles for $T_{3,3,3}$ as a subset of a collection of distinguished vanishing cycles of the generalized Milnor fibre. (This is completely analogous to the process with Milnor fibres of two singularities, one of which is adjacent to the other.) Pick vanishing paths for $M(x,y,z;0)$; as $t$ increases, these deform to vanishing paths for  $M(x,y,z;t)$, any $t<1$. To get vanishing cycles with the desired property, we just need to pick paths that also deform to vanishing paths for $t=1$ (that is, the paths for the eight critical values that remain at $t=1$ do not get `broken' as the six other critical values go off to infinity). We pick such paths.

 We then work backwards: starting with the \emph{known} configuration of vanishing cycles for $M(x,y,z;0)$ (as matching paths in an auxiliary fibration), together with the corresponding vanishing paths, we make mutations to get to the configuration of vanishing paths that is compatible with the deformation in $t$. We track these mutations in the auxiliary fibration. This gives us a description of the vanishing cycles for $T_{3,3,3}$ as matching paths in the auxiliary Lefschetz fibration. Finally, we make a few further mutations, modelled on the case of $T_{p,q,2}$, to get to a nicer collection of matching paths, which is a basis giving Gabrielov's intersection form (Figure \ref{fig:Tpqr'}).

\begin{remark}
We use a deformation argument for $T_{3,3,3}$, even though one could get a description of its Milnor fibre and vanishing cycles using solely A'Campo and Gabrielov's techniques: one representative for $T_{3,3,3}$ is $x^3+ y ^3+ z^3$. The reasons for using a deformation are two-fold: one the one hand, it readily gives a configuration in which we can perform surgery to get an exact torus, and on the other hand, it allows us to extend to the case of higher $r$.
\end{remark}

\subsubsection{A technical lemma}

We shall make repeated use of the following lemma, which is really an observation:

\begin{lemma}\label{th:trivialisotopy}
Suppose that two matching paths intersect in a point  (as in the left-hand side of Figure \ref{fig:trivialisotopy}), and that after Hamiltonian isotopy, the corresponding vanishing cycles do not intersect in the fibre above that point. Then we can make a local change to one of the matching paths so that they no longer intersect, as in the right-hand side of Figure \ref{fig:trivialisotopy}. The resulting matching cycle is Hamiltonian isotopic to the original one.
\begin{figure}[htb]
\begin{center}
\includegraphics [scale=0.85]{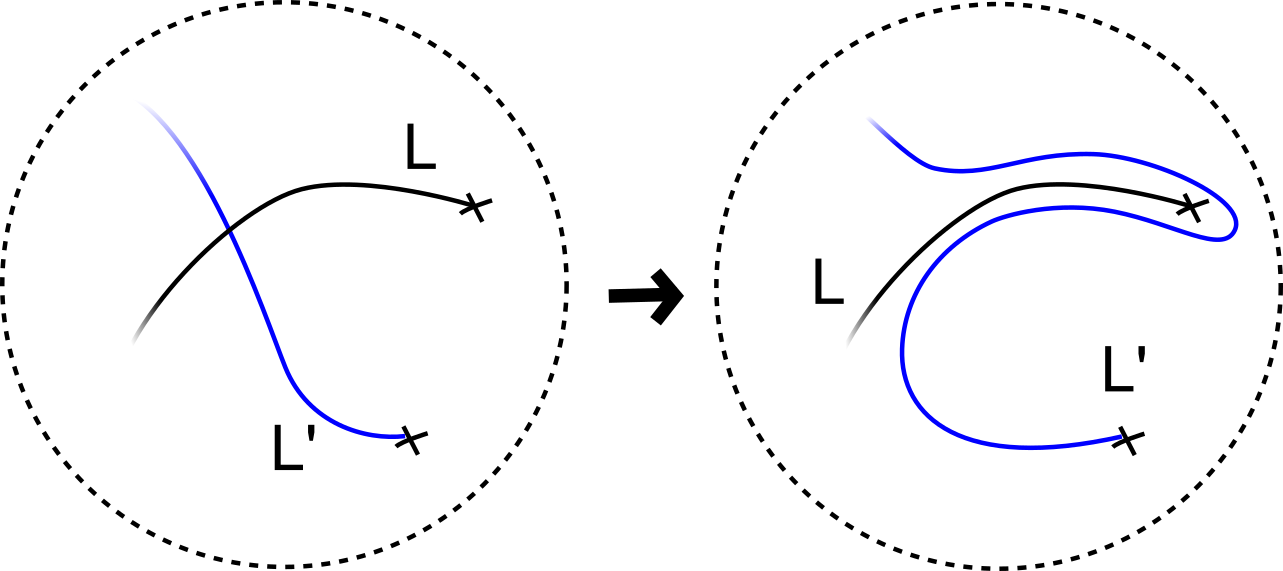}
\caption{Local change of matching paths for vacuous intersections
}
\label{fig:trivialisotopy}
\end{center}
\end{figure}

\end{lemma}

\subsubsection{Initial configuration}

Consider the function
\bq
\tilde{m}(x,y) = 2 \big( (x+0.25)^2 -0.5(y+0.25)-2  \big)  \big(  0.5(x+0.25) +2 -(y+0.25)^2 \big)
\eq
Note that $\frac{1}{2} \tilde{m}$ is a representative for $h_{3,3}$, following the notation of Section \ref{sec:Tpq2}. 
This has seven critical points. The real locus of $\{ \tilde{m}(x,y) =0 \}$ is given in Figure \ref{fig:checkmxy}.
\begin{figure}[htb]
\begin{center}
\includegraphics [scale=0.30]{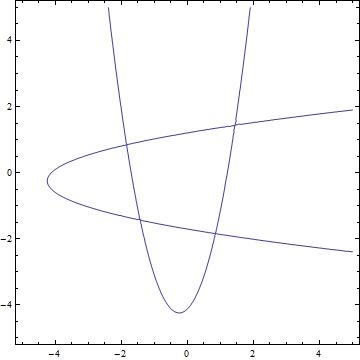}
\caption{ Real locus of $\{ \tilde{m}(x,y) =0 \}$ 
}
\label{fig:checkmxy}
\end{center}
\end{figure}
There are three critical values: a minimum, zero (multiplicity four) and a maximum (multiplicity two). Note that we've arranged for the minimum to be at $(0,0)$.  It realises the `kernel' of the good real deformations we were considering for $T_{p,q,2}$, with the caveat that $\tilde{m}$ is actually \emph{not} the deformation of \emph{any} isolated singularity. (One way to see this is to use the fact that the only singularities with Milnor number seven are $A_7$, $D_7$ and $E_7$; the intersection form associated to the divide of $\tilde{m}$ has non-trivial nullspace, so it is none of these. ) 

Consider a small deformation of $\tilde{m}$, say $m$. It defines a Lefschetz fibration 
\bq
m : V \subset \C^2 \to \C
\eq
with seven critical points and values ($V$ is a suitably chosen large open set). We can still use A'Campo's work to get the topological type of the fibre, and vanishing cycles and paths (see Section \ref{sec:generalizedextensions}). The fibre is a twice--punctured genus three surface; if we choose, as for $T_{p,q,2}$, a regular value of the form $-i \eta$ ($\eta \in \R_+$), and vanishing paths given by straight lines, the vanishing cycles are precisely the coloured cycles in Figure \ref{fig:RS245ACampo} (labelled $a$ through $g$). 

Now let $\tilde{Q}(z)$ be a Morsification of the function $Q(z)=z^3$. As  in Section \ref{sec:Gabrielovcyclic}, we get a Lefschetz fibration
\bq
m+ \tilde{Q}: V' \subset \C^3 \to \C
\eq
with $14$ critical points and values (w.l.o.g.~distinct). We can apply the symplectic version of Gabrielov's techniques to understand vanishing paths and cycles for this. 
To make calculations a little cleaner (with symmetry considerations), we will use the deformation of $\tilde{m}$ given by:
\bq
\check{m}(x,y) = \tilde{m}(x,y) - 2 xy.
\eq
This has seven critical points, and four critical values. One can check that in ascending real order, they correspond to $a$, then $b$ and $c$, then $d$ and $e$, then $f$ and $g$. The minimum -- corresponding to $a$ -- is still $(0,0)$. 
Now consider the function
$$
M(x,y,z) = \check{m}(x,y) +( 18 + 8i) z^2 +  \frac{16i}{3} z^3.
$$
We have chosen a different deformation for $Q$ than in Section \ref{sec:cyclicsymmetry}, so that it would have critical values with different imaginary parts; as the critical values of $\check{m}(x,y)$ are all real, this helps with visualization. 

Critical values are given in Figure \ref{fig:T333atstart}.
The critical values were plotted using the Mathematica code in Appendix \ref{ap:mathematicacode}.
 Vanishing paths given by Gabrielov's algorithm are also on that figure; these were superimposed onto the plot produced by Mathematica.
The yellow--and--purple paths (third and fourth) each give two vanishing paths, though of course the corresponding vanishing cycles do not intersect, so the order does not matter. The same goes for the green (fifth and sixth) and orange (seventh and eighth) paths. 
Note that 
\begin{itemize}

\item
 We have  made trivial mutations to pick the `other' lexicographic ordering between vanishing cycles of $m$ and $Q$ -- as described in Remark \ref{rmk:changinglexicographicordering}, as this was graphically cleaner.

\item Instead of picking a `central' base point with imaginary value roughly between those of the two critical values of $\tilde{Q}$, we have picked one with a smaller imaginary value. This will help make matters graphically cleaner at the next step.

\end{itemize}

\begin{figure}[htb]
\begin{center}
\includegraphics[scale=0.30]{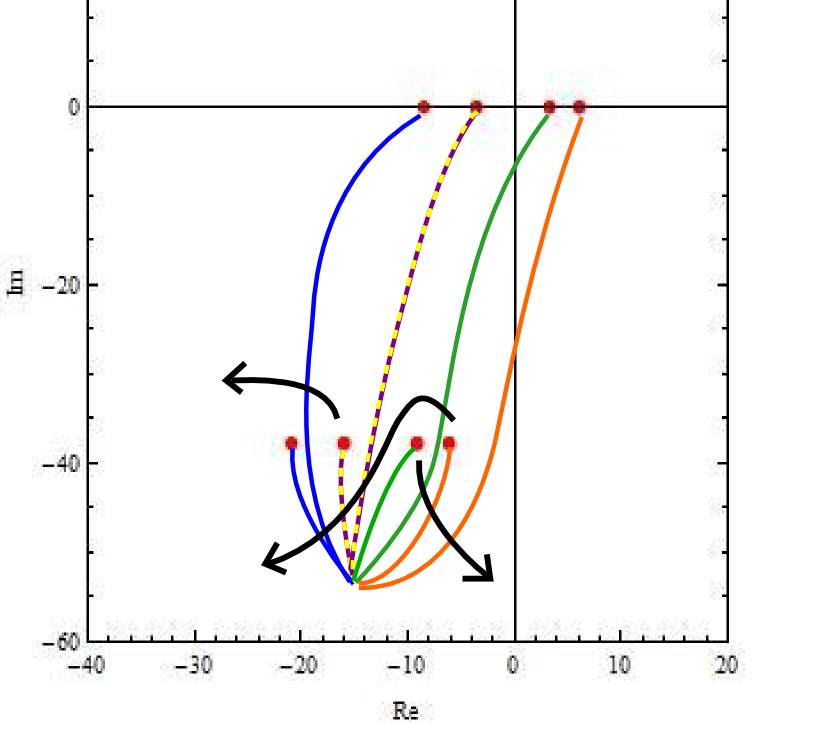}
\caption{ Critical values and vanishing paths for $M(x,y,z)$, with behaviour of the critical values under deformation.}
\label{fig:T333atstart}
\end{center}
\end{figure}

We label the vanishing cycles in increasing order as $(a,1), (a,2), (b,1), \ldots, (g,2)$, consistently with the letter labelling previously used for the kernel. Note $(a,1)$ and $(a,2)$ -- both associated to minima of $\check{m}$ -- correspond to critical points of the form $(0,0, \, \cdot \,)$. By construction, the critical points and values associated to $(a,1)$ and $(a,2)$ remain fixed as $t$ varies.

\subsubsection{Deformation to $T_{3,3,3}$}

Consider the deformation
\bq
M(x,y,z; t) = \check{m}(x,y) + 2(3z+txy)^2 + 8iz^2+ \frac{16i}{3} z^3.
\eq
for $t \in [0,1]$. For $t=0$, this is simply $M(x,y,z)$. For $t=1$, the $-2x^2y^2$ term in $\check{m}(x,y)$ is cancelled by $2t^2x^2y^2$; moreover, one can check  that 
\bq
M(x,y,z;1) = x^3+y^3 + 12xyz + 8iz^2+ 18z^2 + \frac{16i}{3} z^3 + \delta(x,y)
\eq
for some function $\delta(x,y)$. One can check that $M(x,y,z;1)$, a function with 8 critical points, and five critical values,   is a Morsification of $x^3+y^3+z^3+12xyz$ (a representative of $T_{3,3,3}$), as follows: consider the function  
\bq
L(x,y,x;t) = x^3+y^3+12xyz+8iz^2+ 18z^2+\frac{16i}{3}z^2 + (1-t) \delta(x,y).
\eq
As $t$ varies between zero and one, the critical values get deformed smoothly (none escapes to infinity, and none comes in from infinity). These can be plotted using the Mathematica code of Appendix \ref{ap:B}.

For each $t$,  $M(x,y,z;t)$ has only non-degenerate critical points. As $t$ increases from $0$ to $1$, six of the critical points of $M(x,y,z;t)$ go off to infinity; they correspond to three double critical values. The exit paths for the critical values are given in Figure \ref{fig:T333atstart}; to check this is correct, the reader might want to use a computer software program. Mathematica code can be found in Appendix \ref{ap:mathematicacode}. 

\subsubsection{Vanishing paths compatible with the deformation}

We change the distinguished collection of vanishing paths of $M(x,y,z;0)$ so as to have paths that deform to vanishing paths for $M(x,y,z;t)$ for all $t$. (As mentioned before, this just means the paths must, after isotopy, avoid the trajectories of the critical values that go off to infinity.) We choose the collection given by Figure \ref{fig:T333later}.
\begin{figure}[htb]
\begin{center}
\includegraphics[scale=0.30]{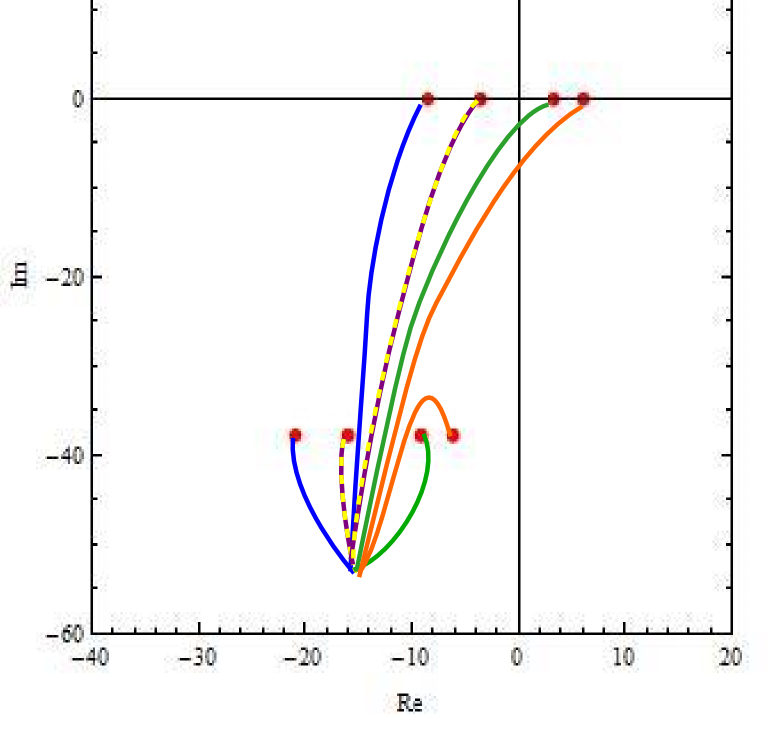}
\caption{New distinguished collection of vanishing paths for $M(x,y,z;0)$
}
\label{fig:T333later}
\end{center}
\end{figure}
This can be obtained from the previous distinguished collection through a series of mutations. The non-trivial ones are
\begin{eqnarray}
(d,2)  & \big(  \mapsto (d,2)^1 = \tau_{(e,1)} (d,2) \big) &\mapsto (d,2)^2= \tau_{(d,1)} (d,2)' 
 \\
(e,2) & \mapsto (e,2)^1= \tau_{(e,1)} (e,2) & \big( \mapsto (e,2)^2 = \tau_{(d,1)} (e,2)' \big)
\\
(f,2) & \big( \mapsto (f,2)^1 = \tau_{(g,1)} (f,2)  \big) & \mapsto (f,2)^2 = \tau_{(f,1)} (f,2)^1  
\\
(g,2) & \mapsto (g,2)^1 = \tau_{(g,1)} (g,2) &  \big( \mapsto (g,2)^2 = \tau_{(f,1)} (g,2)^1 \big)
 \\
(f,2)^2  &  \big( \mapsto (f,2)^3= \tau_{(e,1)} (f,2)^2 \big)  & \big( \mapsto (f,2)^4= \tau_{(d,1)} (f,2)^3 \big)
\\
(g,2)^2  & \big( \mapsto (g,2)^3= \tau_{(e,1)} (g,2)^2 \big) & \big( \mapsto (g,2)^4= \tau_{(d,1)} (g,2)^3 \big)
 \\
(f,1) & \big( \mapsto (f,1)^1= \tau_{(e,1)} (f,1)  \big) & \mapsto (f,1)^2= \tau_{(d,1)} (f,1)^1
 \\
(g,1) & \mapsto (g,1)^1= \tau_{(e,1)} (g,1)& \big( \mapsto (g,1)^2= \tau_{(d,1)} (g,1)^1 \big)
\\
(a,2) & \big( \mapsto (a,2)^1 = \tau^{-1}_{(b,1)}(a,2) \big) & \big( \mapsto (a,2)^2 = \tau^{-1}_{(c,1)}(a,2)^1 \big)
\end{eqnarray}\label{eq:T333mutations}

The critical values that do not go to infinity are critical values for the Lefschetz fibration associated to $M(x,y,z;1)$. We re-label the corresponding distinguished collection of vanishing paths, as follows (with order):
$$
\square = (a,1) \quad a=(a,2) \quad b=(b,2) \quad c=(c,2) \quad
d=(d,2)^2 \quad e=(e,2)^2 \quad f=(f,2)^4 \quad g=(g,2)^4
$$

\subsubsection{Vanishing cycles associated to the deformation--compatible vanishing paths}\label{sec:vcyclesforT333a}
Starting with the configuration of matching cycles which we know to be associated to the collection of vanishing paths of Figure \ref{fig:T333atstart} (the ones we get using techniques from Section \ref{sec:Gabrielovcyclic}), similar to Figure \ref{fig:joincyclic}, we can perform the mutations of \ref{eq:T333mutations}, and track the corresponding matching cycles. The result is in Figure \ref{fig:T333laterplane}. We have deleted the matching paths corresponding to critical values that go off to infinity.
\begin{figure}[htb]
\begin{center}
\includegraphics[scale=1]{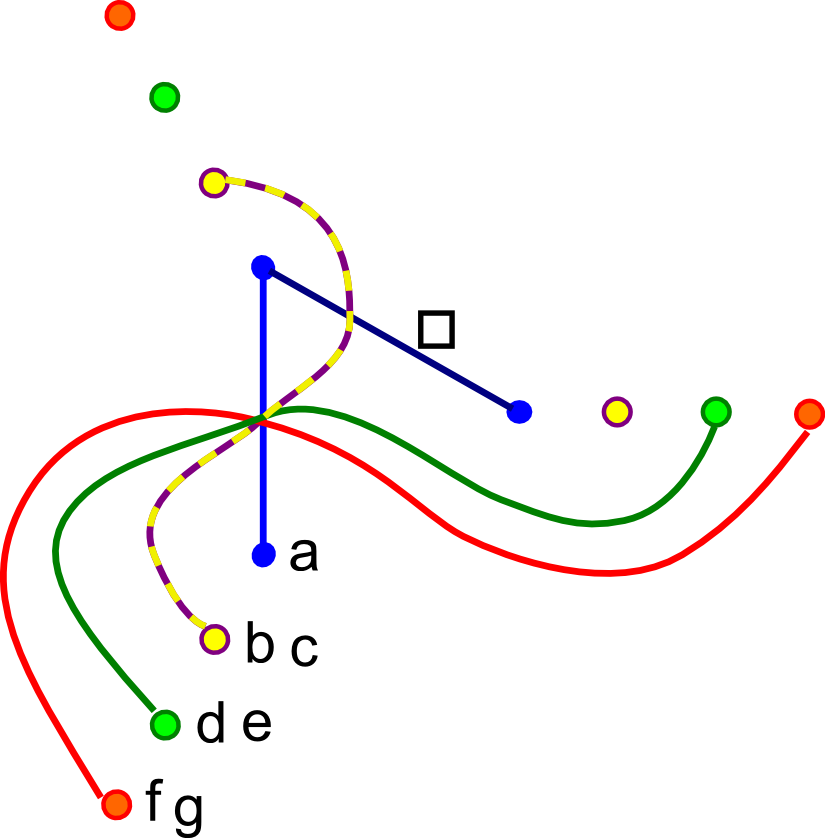}
\caption{Matching paths for $T_{3,3,3}$. The colours codings are as before (note yellow/purple, green and orange each represent two paths), and dark blue corresponds to the cycle $\square$.
}
\label{fig:T333laterplane}
\end{center}
\end{figure}
\subsubsection{Further mutations}\label{sec:vcyclesforT333b}
We perform further mutations on this distinguished basis for $T_{3,3,3}$, inspired by the mutations  made earlier for $T_{p,q,2}$. They are as follows (as always, trivial mutations are in parenthesis):
\begin{eqnarray}
 f  \mapsto ( f^1 = \tau_e f ) \mapsto f^2 = \tau_d f^1  \\
g  \mapsto g^1 = \tau_e g   \\
\square  \mapsto ( \square^1 = \tau_e \square ) \mapsto (\square^2 = \tau_{g^1} \square^1 )
\mapsto ( \square^3 = \tau_d \square^2 )   \mapsto ( \square^4 = \tau_{f^2} \square^3 ) \\
\mapsto  \square^5 =   \tau_c \square^4 \\
b  \mapsto b^1 = \tau_a b \mapsto b^2 = \tau_e b^1 
\mapsto ( b^3 = \tau_{g^1} b^2 )
\mapsto b^4 = \tau_d b^3 
\mapsto ( b^5 = \tau_{f^2} b^4 ) \\
\mapsto b^6 = \tau_c b^5
\end{eqnarray}
These are exactly the mutations as for $T_{p,q,2}$, with the mutation of $\square$ added in.
Set 
\begin{displaymath}
A=a \quad B = b^6 \quad
R_2 = \square^5 \quad R_1 = c \quad P_2 = f^2 \quad P_1 = d \quad 
Q_2 = g^1 \quad Q_1 = e.
\end{displaymath}
These are an ordered basis of vanishing cycles. 
After also making simplifying isotopies as in Lemma  \ref{th:trivialisotopy}, the resulting configuration of matching cycles is given by Figure \ref{fig:T333endplane}, where we have performed an isotopy of the base for further clarity. We can add trivial mutations (analogously to the end of Section \ref{sec:spqmutations}) so that the order of the vanishing paths is that of the right-hand side of  Figure \ref{fig:T333endplane}. 
\begin{figure}[htb]
\begin{center}
\includegraphics[scale=1]{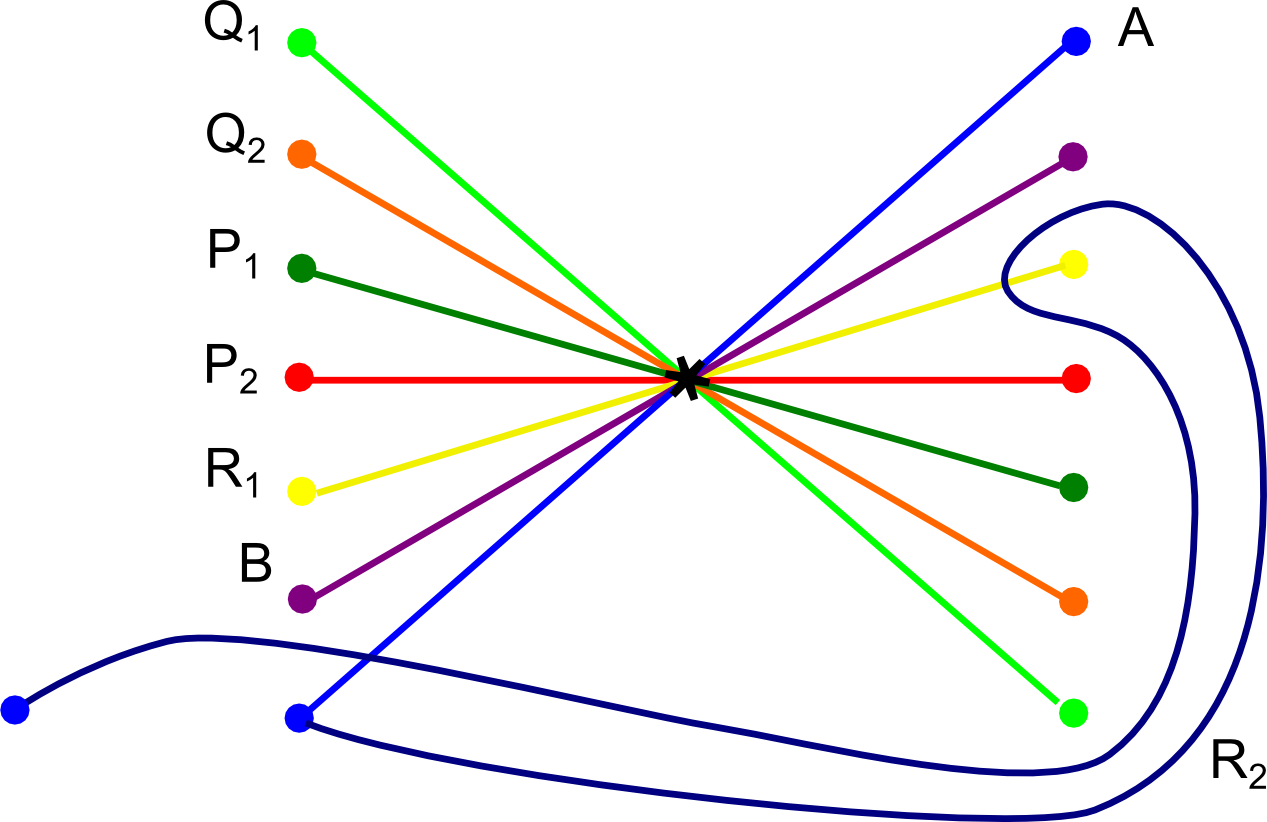}
\caption{Matching paths for $T_{3,3,3}$, after mutations and isotopies.
}
\label{fig:T333endplane}
\end{center}
\end{figure}
We can arrange for seven of the matching paths to intersect in exactly one point, $\star$. In the fibre above $\star$ (a twice-puncture genus three Riemann surface), the vanishing cycles that are restrictions of the seven matching paths intersect precisely as in the $T_{p,q,2}$ case, i.e.~Figure \ref{fig:RS245ACampo}. One can track Dehn twists and check that $R_1$ and $R_2$ intersect exactly in one point. Similarly, $A$ and $R_2$ intersect at two points, with opposite orientations. In particular, we have a configuration of vanishing paths recovering Gabrielov's presentation of the intersection form in Figure \ref{fig:Tpqr'}.

\subsubsection{Isotopy to make $A$ and $R_2$ disjoint}

\begin{lemma}\label{th:AR2disjoint}
After Lagrangian isotopy, $A$ and $R_2$ can be arranged to be disjoint (without affecting intersections with any of the other vanishing cycles). 
\end{lemma}

\begin{proof}
We start by describing a local model for the the intersections of $A$, $R_1$ and $R_2$.
Consider the hypersurface $\{ (x,y,z) \, | \, x^3+y^3+z^2 =1   \} $. This is the Milnor fibre of $D_4$. It can be viewed as the total space of a Lefschetz fibration in many different ways. Consider
\begin{eqnarray}
\Pi_1: &\{ (x,y,z) \, | \, x^3+f(y,z) =1   \}   & \to \C \\
& (x,y,z) \to x
\end{eqnarray}
where $f(y,z)$ is a Morsification of $y^3+z^2$. There are six critical values; the generic fibre is a once-punctured torus, with vanishing cycles given by the standard $a$ and $b$ curves. See Figure \ref{fig:D4model}.
\begin{figure}[htb]
\begin{center}
\includegraphics[scale=0.75]{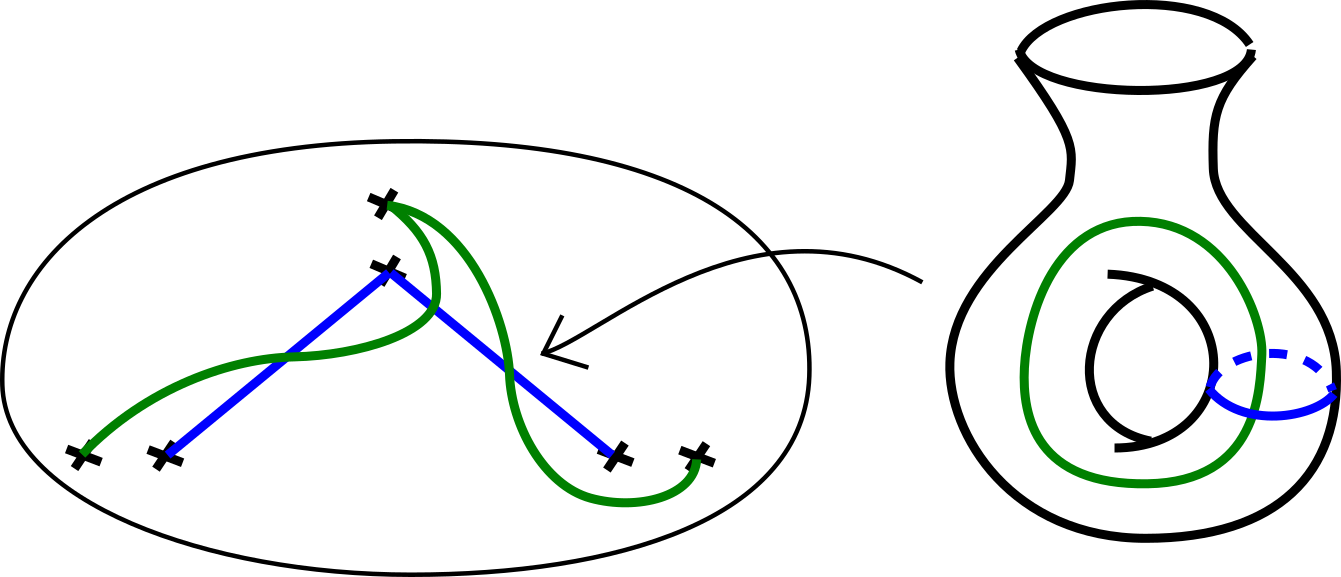}
\caption{Lefschetz fibration $\Pi_1$ on the Milnor fibre of $D_4$: base and fibre.
}
\label{fig:D4model}
\end{center}
\end{figure}

Matching paths corresponding to a distinguished collection of vanishing cycles for $D_4$ are given by Figure \ref{fig:AR2disjoint1} (this is the collection given by Section \ref{sec:Gabrielovcyclic} techniques after one mutation). Curves $\a$, $\b$ and $\gamma$ on the figure give a local model for the intersections of $A$, $R_1$ and $R_2$. 
\begin{figure}[htb]
\begin{center}
\includegraphics[scale=0.75]{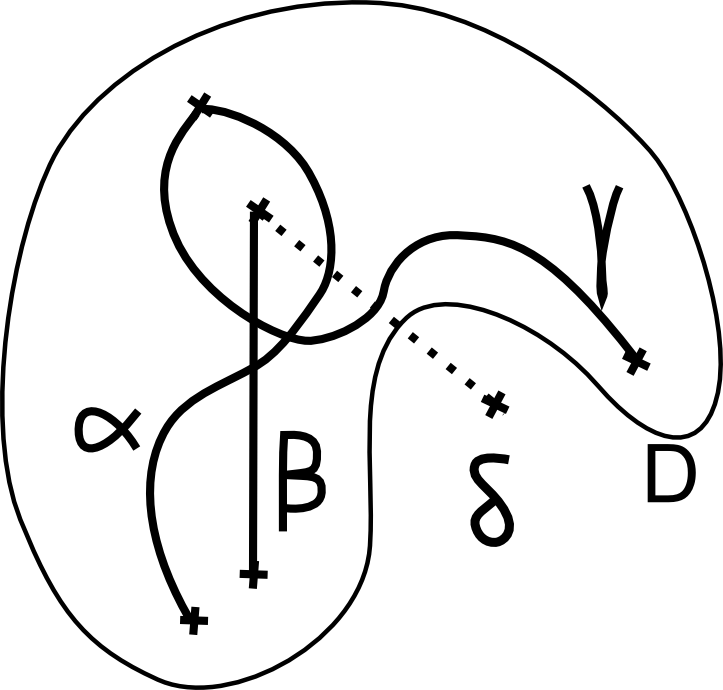}
\caption{First Lefschetz fibration for a local model for the intersections of $A$ (which corresponds to $\a$), $R_1$ ($\b$) and $R_2$ ($\gamma$).
}
\label{fig:AR2disjoint1}
\end{center}
\end{figure}
Now consider a different Lefschetz fibration, with, again, the Milnor fibre of $D_4$ (with three variables) as total space:
\begin{eqnarray}
\Pi_2: &\{ (x,y,z) \, | \, g(x,y)+z^2 =1   \}   & \to \C \\
& (x,y,z) \to z
\end{eqnarray}
where $g(x,y)$ is a Morsification of $x^3+y^3$.This has eight critical values; matching paths can be arranged to intersect in exactly one point. The fibre above that point is  the Milnor fibre of $x^3+y^3$ ($D_4$ with two variables); let's call is $S$. It is a three-punctured torus (and a distinguished collection of vanishing cycles is given by the `standard' $D_4$ configuration of $S^1$'s on that torus).
Also, we know by Lemma \ref{th:mutationseq} that through some (a priori unknown) sequence of mutations, we can get the same configuration of vanishing cycles as in Figure \ref{fig:AR2disjoint1}.
(By ``same'', we mean that there is an exact symplectomorphism between the two representatives of the Milnor fibre that takes one ordered collection of vanishing cycles to the other, possibly after Lagrangian isotopies of some of the cycles.) 
The sequence of mutations modifies the distinguished collection of vanishing cycles. On the other hand, one can keep track of the effects of mutations (i.e.~the change in the distinguished collection) in each of our Lefschetz fibration models for the $D_4$ Milnor fibre. In both bases, using Section \ref{sec:matchingmutations}, matching paths get taken to different matching paths. In the case of $\Pi_2$, where the matching paths initially only intersect at one point, this remains true. 

The vanishing cycles corresponding to $\alpha$, $\beta$, $\gamma$ and $\delta$ are given by matching paths as in Figure \ref{fig:AR2disjoint2}, where they are labelled as $\alpha'$, $\beta'$, $\gamma'$ and $\delta'$. 

\begin{figure}[htb]
\begin{center}
\includegraphics[scale=0.75]{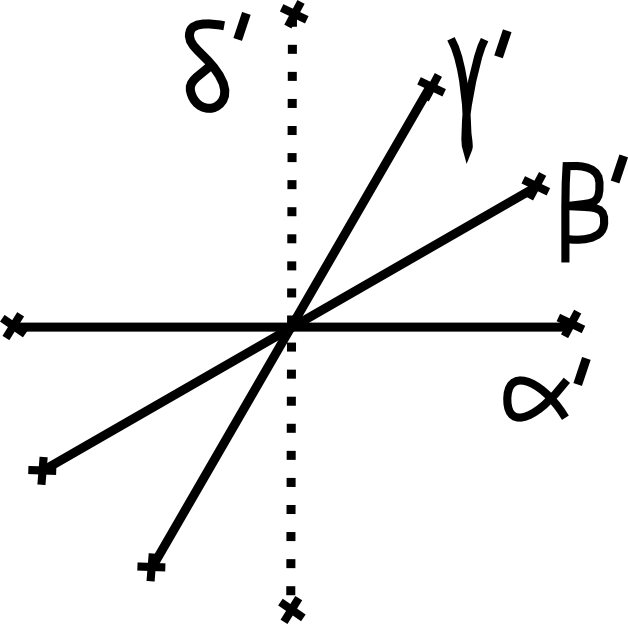}
\caption{Second Lefschetz fibration for a local model for the intersections of $A$ (which corresponds to $\a'$), $R_1$ ($\b'$) and $R_2$ ($\gamma'$).
}
\label{fig:AR2disjoint2}
\end{center}
\end{figure}
In particular, we can arrange for the matching paths for $\a', \ldots, \delta'$ to meet in exactly one point (w.l.o.g., $\Pi_2(S)$). 
The restrictions of the vanishing cycles $\alpha', \ldots, \delta'$ to $S$ are themselves vanishing cycles for the singularity $g$. 
Consider the simple closed curves $\a'|_S, \ldots, \delta'|_S$ on the three-punctured torus $S$. None can be contractible in the closure of $S$ to a torus (this is true of any vanishing cycle for $g$, e.g.~because of exactness). 
By construction, the \emph{signed} intersection number between $\a'|_S$ and $\gamma'|_S$ is zero: 
$\a$ and $\a'$ are Lagrangian isotopic, and $\gamma$ and $\gamma'$ are too; $\a$ and $\gamma$ have homological intersection zero, so $\a'$ and $\gamma$, which only meet on $S$, must also have intersection zero. 
Thus they must be isotopic in the closure of $S$,
using the fact that said closure is a torus.
 (They are not isotopic in $S$ itself: they are vanishing cycles for different critical values.) Hence there exists an isotopy of $S$ such that $\a'|_S$ and the image of $\gamma'|_S$ are disjoint: such an isotopy certainly exists in the closure of $S$, and we can arrange for it to have support on $S$. 
W.l.o.g.~$\a'|_S$ and $\gamma'|_S$ intersect transversally. To displace $\gamma'_S$ from $\a'|_S$, we need to (possibly iteratively) cancel out finitely many discs between them, each with finite symplectic area. In order for the isotopy to be Hamiltonian, we need to compensate for these. As $\a'|_S$ and $\gamma'|_S$ are not isotopic in $S$ (in particular, there are punctures ``between'' them on $S$), and the symplectic area of any subset of $S$ containing a puncture is infinite, there must be a compactly-supported \emph{Hamiltonian} isotopy of $S$ after which $\a'|_S$ and the image of $\gamma'|_S$ are disjoint.  
 Extending this over the pre-image of a neighbourhood of $\Pi_2(S)$, we see that there is a compactly supported Hamiltonian isotopy such that $\a'$ and the image of $\gamma'$ do not intersect. 
Thus the same is true of $\a$ and $\gamma$. Moreover, we can arrange for such an isotopy to have support in the preimage of the region $D$ in Figure \ref{fig:AR2disjoint1} (this equivalent to asking that none of the Lagrangian isotopic copies of $\gamma$ intersect a fixed point of $\delta$). 
In such a situation, the isotopy is contained to the local model. This implies that after a Hamiltonian isotopy of $R_2$, $A$ and $R_2$ can be made disjoint; the image of $R_2$ still intersects $R_1$ at one point. The other vanishing cycles lied away from this local model, so intersections with them are unchanged. 
\end{proof}

\subsection{Case of a general $T_{p,q,r}$}\label{sec:Tpqr}

\begin{proposition}\label{th:Tpqr}

 Let $\mathcal{T}_{p,q,r}$ denote  the Milnor fibre of $T_{p,q,r}$. Assume our subscripts are ordered so that $p \geq 3$, $q \geq 3$, and $r \geq 2$. The space $\mathcal{T}_{p,q,r}$ can be described as the corner-smoothing of the total space of a Lefschetz fibration $\pi$, whose fibre is the same as the Milnor fibre of $\sigma_{p,q}$ (a Riemann surface), and with $2p+2q+r$ critical points and values. 
 Of these critical points, $2(p+q+1)$ of them correspond to critical points for the description of $\mathcal{T}_{p,q,2}$ as the total space of a Lefschetz fibration. There are $p+q+1$ corresponding vanishing cycles, which are given by matching paths for that Lefschetz fibration; they all intersect in one point, $\star$. 
 
In the fibre above $\star$, which we shall call $M_\star$, the matching paths restrict to the configuration of vanishing cycles for $\sigma_{p,q}$ that we had already met (e.g.~Figure \ref{fig:RS245Mutated}). 

 As for the remaining $r-2$ critical points, there is an $A_{r-2}$--type chain of matching paths between them; these give the remaining vanishing cycles for $\mathcal{T}_{p,q,r}$. The case of $p= 3$, $q= 4$ and $r=5$ is given in Figure \ref{fig:matchingTpqr}.  
For $r=6$, the chain would be extended by the matching path for $R_5$, which intersects only the matching path for $R_4$, at their shared critical value, and so on. 

Moreover, in all cases, there is a Hamiltonian isotopy such that the image of $R_2$ does not intersect $A$, and its intersections with other cycles are unchanged.
 
\begin{figure}[htb]
\begin{center}
\includegraphics[scale=0.85]{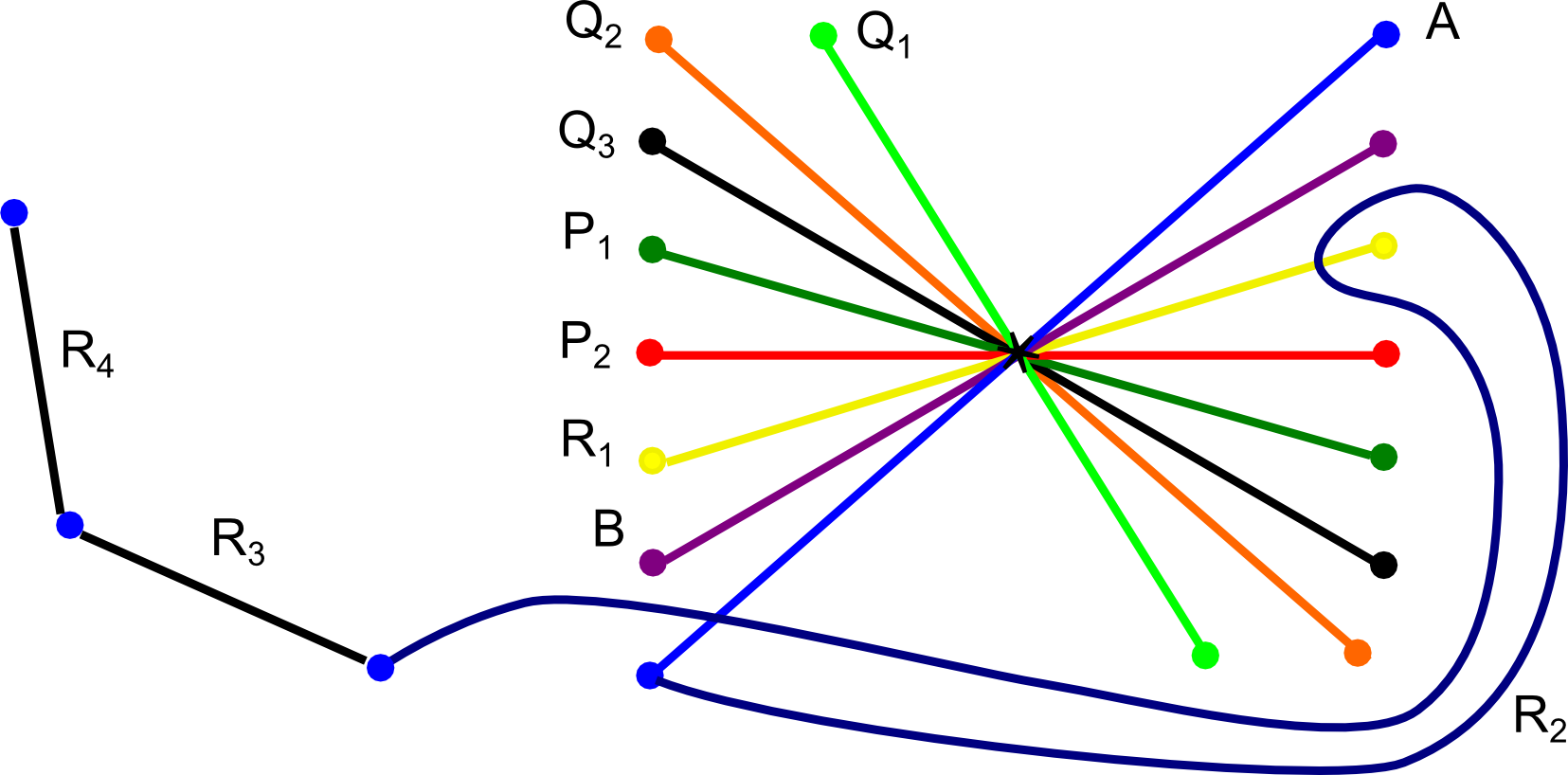}
\caption{ Matching paths giving a distinguished configuration of vanishing cycles for $T_{3,4,5}$; the labelling respects the previous ones, with analogous ordering.
}
\label{fig:matchingTpqr}
\end{center}
\end{figure}
 
\end{proposition}

The reader might have noticed that for the cases $(p,q) = (3,3)$, $(3,4)$ and $(3,5)$, there is no such isolated singularity as $\sigma_{p,q}$. What do we mean by the ``Milnor fibre'' of $\sigma_{p,q}$ in this case? We take the generalized Milnor fibre of functions constructed so as to have analogous features. We already did such a thing for the case $(p,q)=(3,3)$: use the generalized Milnor fibre of $h_{3,3} := \frac{1}{2} \check{m}$ (see also Figure \ref{fig:checkmxy}). For $(3,4)$,  take the generalized Milnor fibre of $h_{3,4}$,  a real deformation of $(x^2-y^2)(x-y^2)$ with real zero locus given by Figure  \ref{fig:h3435}.   Similarly, for  $h_{3,5}$, use a real deformation of $(x^2-y^3)(x-y^2)$ with real zero locus given by Figure \ref{fig:h3435}.  Notice that these fit into the pattern of the functions $h_{p,q}$ of Section \ref{sec:Tpq2}. 
 \begin{figure}[htb]
 \begin{center}
\includegraphics[scale=0.85]{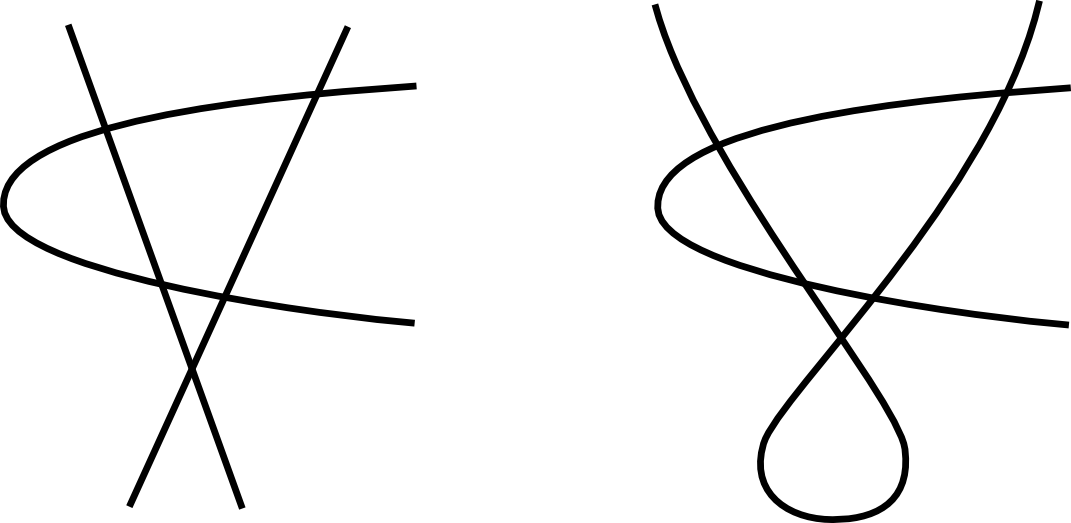}
\caption{Real zero locus of  $h_{3,4}$ (left) and $h_{3,5}$ (right).
}
\end{center}
\label{fig:h3435}
\end{figure} 
(Compare with Figure \ref{fig:Tpq2Morse}.) As with $\sigma_{p,q}$ and the functions $h_{p,q}$ of Section \ref{sec:Tpq2}, A'Campo's algorithm will give a Riemann surface with vanishing cycles consisting of seven simple closed curves  -- the `kernel' introduced in Section \ref{sec:Tpq2} -- and chains of the appropriate length attached to it. 

\begin{proof} For $T_{p,q,2}$ and $T_{3,3,3}$, this agrees with the description that we already have. We will first show that the theorem holds for $T_{3,3,r}$, with $r \geq 4$, by building on the work of the previous section (describing $T_{3,3,3}$). We will then combine this with what we know about $T_{p,q,2}$ to get the description for a general $T_{p,q,r}$. 

We make the following preliminary choices and additional assumptions:

\begin{itemize}
\item Let $Q_3(z) = 9z^2 + 4iz^2 + \frac{8i}{3} z^3$, and $\widetilde{Q}_3(z) = 4iz^2 + \frac{8i}{3} z^3$. Thus in the notation of Section \ref{sec:T333}, we have 
\bq
M(x,y,z) = \check{m}(x,y) + 2Q_3(z)
= \check{m}(x,y) + 2(3z)^2 + 2\widetilde{Q}_3(z).
\eq

\item For $r>3$, let $Q_r(z)$ be a polynomial of degree $r$, with no constant term, such that $2$ of its critical points are the same as those of $Q_3(z)$. Assume $Q_{3}(z) + t Q_r(z)$ is Morse for all $t \in [0,1]$. 

\item Recall that for every pair $(p,q)$, the function $h_{p,q}$ is a good real deformation  of a function of the form 
\bq
(x^{p-2} - y^2) ( x^2 - c y^{q-2})
\eq
for some constant $c$. We further assume that $h_{p,q}$ is also a polynomial, and that the coefficient of $x^2y^2$ is still $-1$. For $h_{3,3}$ we use instead $0.5 \check{m}$.

\item For every pair $(p,q)$, we already have that $h_{p,q}$ has a real minimum to which the vanishing cycle $a$ (one of the seven in the `kernel') is associated. Assume that the corresponding critical point is $(x,y)=(0,0)$. (Note that this is already the case for $h_{3,3}$.)

\item Our arguments will use some one-parameter families of polynomials involving that $h_{p,q}$ and $Q_r$ (the parameter lies in $[0,1]$). One can check that we can arrange for the $h_{p,q}$ and $Q_r$ to be such that these deformations are through Morse functions.
\end{itemize}
(Without being truly restrictive, these are all somewhat technical assumptions. They are principally designed to allow us to easily track critical points as we deform polynomials.)

  We'll start by studying $\mathcal{T}_{3,3,r}$. The function $Q_3(z) + Q_r(z)$ has $r-1$ critical points. Label them as $z_1, \ldots, z_{r-1}$, where $z_{r-1}=0$, and $z_{r-2}$ is the other critical point of $Q_3$. Pick any collection of vanishing paths giving that order. The function $h_{3,3}(x,y)$ has seven critical points; choose as ordered collection of vanishing paths and labels $a, \ldots, g$ following Section \ref{sec:T333}. The function 
     \begin{equation}
      h_{3,3}(x,y) + Q_3(z) + Q_r(z)
     \end{equation}
   has $7(r-1)$ critical points. From Gabrielov's algorithm, we get a distinguished collection of vanishing paths, together with a description of the corresponding vanishing cycles as matching cycles in an auxiliary Lefschetz fibration. (The fibre of this Lefschetz fibration is the generalized Milnor fibre of $\sigma_{3,3}$.) Label the vanishing cycles as $(\alpha, j)$, where $\alpha = a, b, \ldots , g$, and $j=1, \ldots, r-1$, and the corresponding critical points as $C_{(\a, j)}$. This is a  generalization of the scenario considered in Section \ref{sec:T333}, where we had $r=3$.
The left-hand side of Figure \ref{fig:T33r} gives, schematically, the configuration of critical points vanishing paths.
\begin{figure}[htb]
\begin{center}
\includegraphics [scale=0.32]{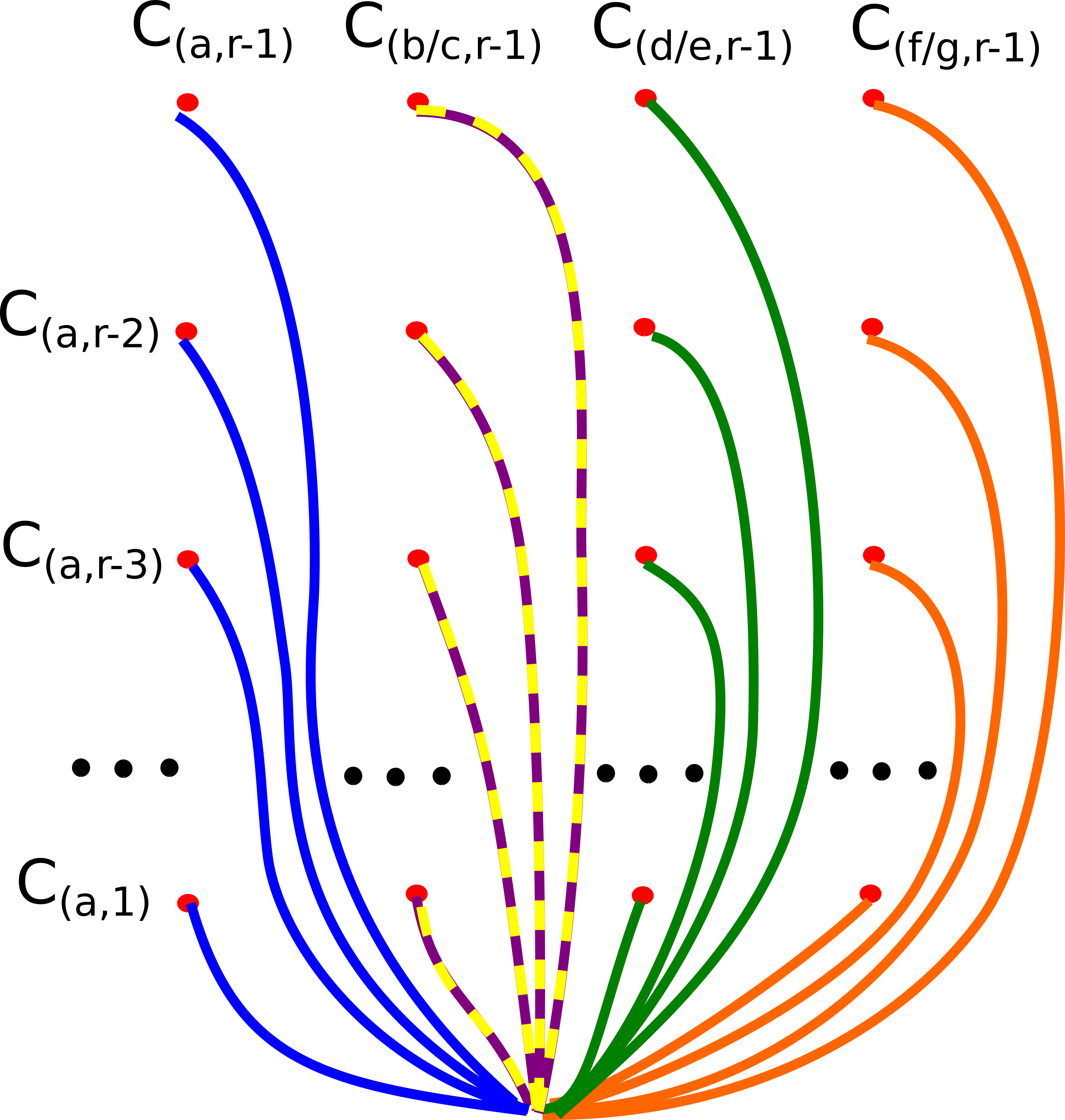} \qquad
\includegraphics[scale=0.32]{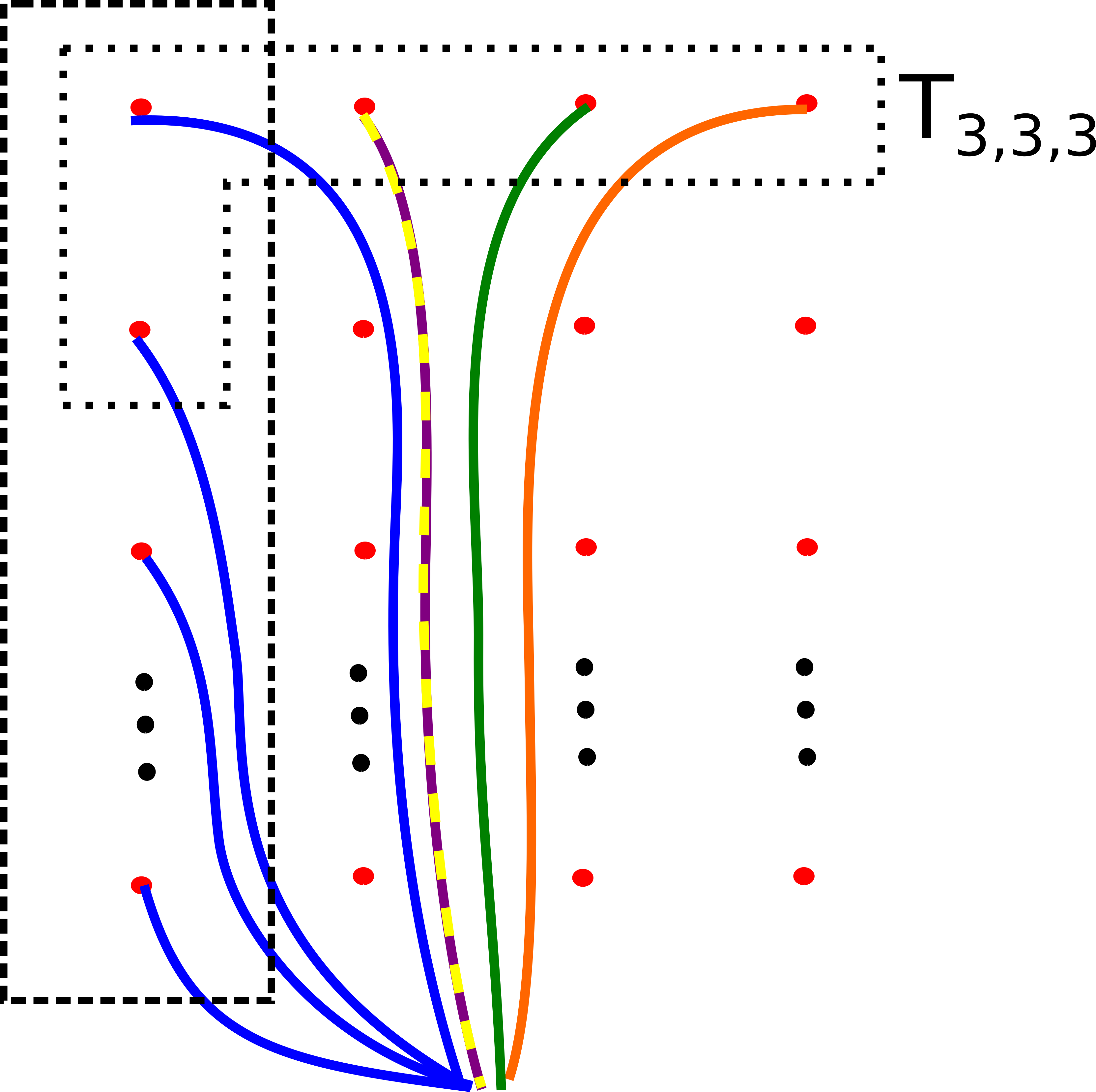}
\caption{Finding vanishing paths for $T_{3,3,r}$: on the left-hand side, the initial configuration ($t=l=0$); on the right-hand side, vanishing paths for the critical points that survive.}
\label{fig:T33r}
\end{center}
\end{figure}
 Now consider the function
\begin{equation}
   M_r(x,y,z;t;l) = h_{3,3}(x,y) + (3z+txy)^2 + \widetilde{Q}_3(z) + (1-l) Q_r(z)
\end{equation}
   where $t, l \in [0,1]$. Fixing $l=0$ and deforming $t$ from 0 to 1, we get (near the origin) a Morsification of $T_{3,3,r}$. We would like to know which of the critical points of $M_r(x,y,z;0;0)$ contribute to this. We expect $p+q+r-1=5+r$ of those points in total. 
First, note that they must include $C_{(a,j)}$ for all $j=1, \ldots, r-1$, as those critical points, whose $x$ and $y$ coordinates are zero, remain fixed as we increase $t$. These points are in the left-hand column of the right-hand half of 
Figure \ref{fig:T33r}, in a dashed box.
 One the other hand, suppose you start with $M_r(x,y,z;0;0)$, and increase $l$ to 1. 
By construction, this fixes $C_{(\a, j)}$ for all $\a =a, \ldots, g$, and $j=r-2, r-1$.  Moreover, we already understand the deformation from $M_r(x,y,z;0;1)$ to $M_r(x,y,z;1;1)$ given by increasing $t$, as it is precisely what we studied in Section \ref{sec:T333}. In particular, the critical points that are not sent to infinity are the images of $C_{(a, r-2)}$, $C_{(a,r-1)}$,  $C_{(b,r-1)}$, $C_{(c,r-1)}$, \ldots, $C_{(g,r-1)}$ (the other dashed part of the right-hand diagram in Figure \ref{fig:T33r}), 
and the generalized Milnor fibre of $M_r(x,y,z; 1;1)$ is $\mathcal{T}_{3,3,3}$.    
 On the other hand, notice that deforming  from $M_r(x,y,z;1;0)$ to $M_r(x,y,z;1;1)$ realizes the adjacency $T_{3,3,r} \to T_{3,3,3}$.  Putting all of this together, we see that the Milnor fibre $\mathcal{T}_{3,3,r}$ is a subset of the generalized Milnor fibre of $M(x,y,z;0,0)$, and that its critical points correspond to $C_{(a,1)}$, $C_{(a,2)}$, \ldots, $C_{(a,r-2)}$, $C_{(a,r-1)}$,  $C_{(b,r-1)}$, \ldots, $C_{(g,r-1)}$. 
 
We need to find vanishing paths that are compatible with the deformations.
  Start with $C_{(a,r-2)}$, $C_{(a,r-1)}$,  $C_{(b,r-1)}$, \ldots, $C_{(g,r-1)}$. From the study of $T_{3,3,3}$, we already know what vanishing paths to choose in order to avoid the exit trajectories of $C_{(b,r-2)}$, $C_{(c, r-2)}$, \ldots, $C_{(g,r-2)}$. Moreover, notice that none of the cycles $C_{(a,r-2)}$, $C_{(a,r-1)}$,  $C_{(b,r-1)}$, \ldots, $C_{(g,r-1)}$, have intersection points any of the other cycles $C_{(\a, j)}$ with $j \leq r-3$ and whose critical values exit the picture. (This can just be read off from the Gabrielov description, see Section \ref{sec:Gabrielovcyclic}.) 
  Thus, after the mutations that we found for the $T_{3,3,3}$ case (Section \ref{sec:T333}) and maybe some extra trivial mutations, we can get vanishing paths for  $C_{(a,r-2)}$, $C_{(a,r-1)}$,  $C_{(b,r-1)}$, \ldots, $C_{(g,r-1)}$ avoiding \emph{all} exit trajectories. Their order will still follow that for $T_{3,3,3}$. As for $C_{(a,1)}, \ldots, C_{(a,r-3)}$, one can simply start with the vanishing paths given by the Gabrielov configurations, and modify them all by the same sequence of (trivial) mutations as the one that is applied to $C_{(a,r-2)}$.  (There are different ways of seeing this; the simplest might be to observe that one can deform $Q_r(z)$ so that all $r-2$ first critical points almost coincide, and consider the effect for critical points and vanishing paths for $M_r(x,y,z;t;l)$.) 
Altogether, this gives the vanishing paths on the right-hand side of Figure \ref{fig:T33r}, possibly up to trivial mutations.
We're now almost done, although the description we have isn't yet as nice as the one stated in the theorem: we are at the analogue of the end of Section \ref{sec:vcyclesforT333a} for the $T_{3,3,3}$ case. To get the claimed configuration, we make some further mutations following Section \ref{sec:vcyclesforT333b} for
$T_{3,3,3}$. 
    
    We are left with the case of a `general' $T_{p,q,r}$. One can proceed similarly to before, considering the function
\begin{equation}
    N_{p,q,r}(x,y,z;t,l,\mu) = (1-\mu)h_{p,q}(x,y) +\mu h_{3,3}(x,y) 
    + (3z+txy)^2+ (1-l)Q_r(z)
  \end{equation}
    with $t, l, \mu \in [0,1]$. We can use the Gabrielov and A'Campo techniques to describe the generalized Milnor fibre of $N_{p,q,r}(x,y,z; 0;0;0)$ as the total space of a Lefschetz fibration, with fibre the Milnor fibre of $\sigma_{p,q}$. The critical points are represented by red and green dots in Figure \ref{fig:Tpqr-end}. For $t=1$, $l=\mu=0$, we get a Morsification of $T_{p,q,r}$. (This is where the assumption about the coefficient of $x^2y^2$ in the expression for $h_{p,q}$ comes in, as the two terms involving $x^2y^2$ in $N_{p,q,r}(x,y,z;1;0;0)$ cancel out.) Fixing $t=1, \mu=0$ and deforming $l$ from 0 to 1  realises the adjacency $T_{p,q,r} \to T_{p,q,2}$, or the appropriate generalization in the cases where $(p,q)=(3,4)$ or $(3,5)$. The surviving critical points are the first line in Figure \ref{fig:Tpqr-end} (in a dashed box). 
Fixing $t=1, l=0$ and deforming $\mu$ from 0 to 1 realises to adjacency $T_{p,q,r} \to T_{3,3,r}$. The surviving critical points are in the other dashed box in Figure \ref{fig:Tpqr-end}.
 Combining the information from both of these and proceeding similarly to above, we get the configuration of vanishing paths of Figure \ref{fig:Tpqr-end}, avoiding (possibly up to trivial mutations) the exit paths of the other points. We then use the same mutation sequences as before to obtain the desired description of vanishing cycles for $T_{p,q,r}$. 
\begin{figure}[htb]
\begin{center}
\includegraphics [scale=0.32]{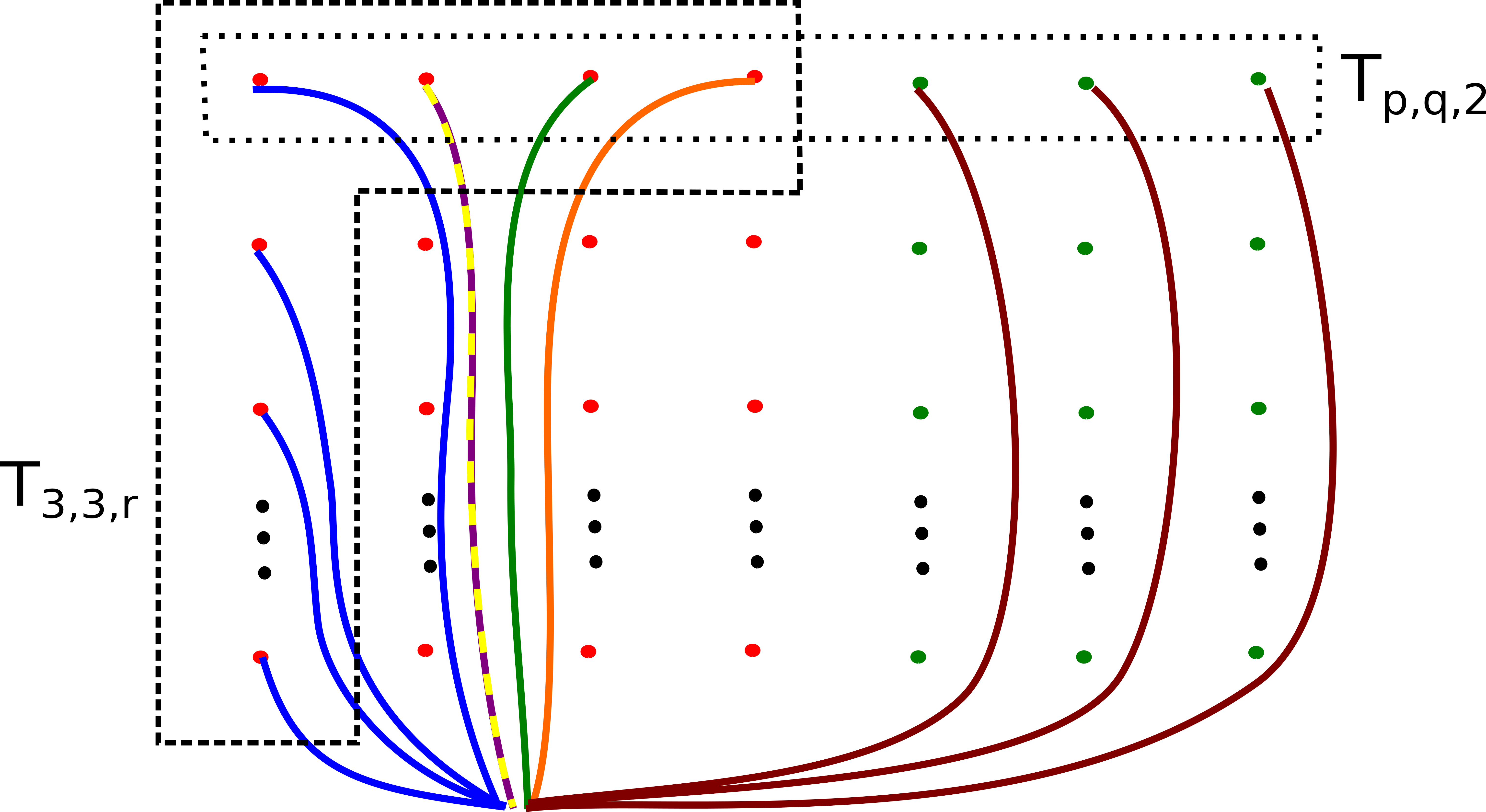} 
\caption{Finding vanishing paths for $T_{p,q,r}$.}
\label{fig:Tpqr-end}
\end{center}
\end{figure}

Finally, in all cases, the Lefschetz fibration in a neighbourhood of $R_2$ will have the same local model as before; as the proof of Lemma \ref{th:AR2disjoint} was local (and yields a compactly supported Hamiltonian isotopy), the claim about $R_2$ still holds.
\end{proof}

\subsubsection{Essentially local changes of the symplectic form and product structure near $M_\star$.}

Throughout the previous section, as well as Sections \ref {sec:Tpq2}  and \ref {sec:T333}, we considered Lefschetz fibrations $\pi: E \to B$ where $E$ is an open subset of $X$, some smooth hypersurface $X \subset \C^{3}$, $B$ is an open subset of $\C$, and $\pi$ is given by a complex polynomial. The Milnor fibre of $T_{p,q,r}$ was given either by $\mathcal{T}_{p,q,r} = E$, or by an open subset $\mathcal{T}_{p,q,r}$ of $E$. We repeatedly displaced vanishing cycles in the fibre $\pi^{-1}(\star)$ by Hamiltonian isotopies, obtained by making essentially local changes of the symplectic form. (See Lemma \ref{th:changeomega} and thereafter, as well as the remarks in Section \ref{sec:moregeneralLefschetzfibrations} about using these tools for more general Lefschetz fibrations.) One might worry about the effect of these changes on $E$ as a symplectic manifold. Thankfully, we have the following:

\begin{lemma}
Let $\o$ be the original symplectic form on $\mathcal{T}_{p,q,r}$, and $\o'$ the form after all the essentially local changes. We claim that $(\mathcal{T}_{p,q,r}, \o')$ is also a copy of the Milnor fibre of the  same singularity, possibly defined with a different holomorphic representative of the singularity, and different cutoffs. In particular, they are exact symplectomorphic when completed with cylindrical ends.
\end{lemma}

\begin{proof} Changes by compactly supported one-forms are taken care of by a Moser argument. 
Observe that in the cases we are concerned about, $\pi: E \to B$ is always given by projecting to the third complex coordinate, $z$. Recall that $\o$, the symplectic form on $E$, is the restriction of usual Kaehler form on $\C^3$: 
\bq
\o_{\C^{3}} = \frac{i}{4} d d^c \,\big( ||x||^2 + ||y||^2 + ||z||^2  \big).
\eq
This means that $\o + c \pi^\ast \o_b$ is given by the restriction of the Kaehler form
\bq
\o_{\C^{3}} = \frac{i}{4} d d^c \, \big( ||x||^2 + ||y||^2 + ||c' z||^2  \big)
\eq
for some constant $c'$. This is the same effect as a holomorphic reparametrization, which does not affect completed Milnor fibres, by Lemma \ref{th:Milnorfibreindepreparametrisation}.
\end{proof}

With the preceding observation under our belt, using Lemma \ref{th:changeomegatoproduct},  we can assume that the description of the Milnor fibre $\mathcal{T}_{p,q,r}$ given by Proposition \ref{th:Tpqr} has the following additional feature: 

\begin{assumption}\label{ass:productnhood}
Fix a large compact subset of $M_\star$, say $K$, such that each of the vanishing cycles on $M_\star$ are contained in the interior of $K$. Fix $r>0$ such that the closure of $B_r(\star)$ does not contain any critical values. We assume that near $K \subset M_\star$ our symplectic form restricts to `product' symplectic form, say $\o_{pr}$, following the description of Lemma \ref{th:changeomegatoproduct}. Moreover, we assume that all the matching paths through $\star$ are given by straight-line segments  inside $B_{r/2}(\star)$, and that no other matching paths intersect $B_{r}(\star)$. 
\end{assumption}

We modify our original $\o$--compatible almost complex structure $J$ following Lemmas \ref{th:changej} and \ref{th:changejtoproduct}. The result in an $\o_{pr}$--compatible almost complex structure, say $J_{pr}$, which on a neighbourhood of $K$ is given by the product of a complex form on fibres and a complex form on the base.



\section{Torus construction}\label{sec:torusconstruction}

\subsection{Lagrangian surgery}\label{sec:lagrangiansurgery}

We shall use the operation of Lagrangian surgery, introduced by Polterovich  \cite{Polterovich}. We use the same version as \cite[Appendix A]{Seidel99}; note that this construction allows more general cut-off functions, and has the opposite ordering convention to Polterovich's. 
Suppose you have two Lagrangians $L_1$ and $L_2$ that intersect transversally at one point $u$. Fix an order of $L_1$ and $L_2$, say $(L_1, L_2)$. Lagrangian surgery is a local procedure for obtaining a new Lagrangian $L_1 \# L_2$, which agrees with the union of $L_1$ and $L_2$ outside an arbitrarily small neighbourhood of $u$.
Pick a Darboux neighbourhood of $u$, say $\phi: U \ni u \to \R^4$ such that
\begin{itemize}
\item $\phi(u)=0$;
\item $\phi(L_1 \cap U ) = \big( \R \times \{0 \} \times \R \times \{0 \} \big) \cap \phi(U)$;
\item $\phi (L_2 \cap U ) = \big( \{0 \} \times \R \times \{0 \} \times \R \big) \cap \phi(U)$.
\end{itemize}
(See \cite[Section 4]{Polterovich}.)
Let $h: \R \to \R^2$ be a smooth embedding such that $\text{Im}(h)$ agrees with $\R_+ \times \{ 0\} \cup \{ 0 \} \times  \R_-$ outside $\phi(U)$, and such that there is no $z \in \R^2$ such that both $z$ and $-z$ lie in $\text{Im}(h)$. 
The Lagrangian handle associated to $h$ is
\bq
H = \{ (x \, \text{cos } t , y \, \text{cos } t, x \, \text{sin }t , y \, \text{sin } t) \, | \, (x,y) \in \text{Im}(h), t \in S^1 \}.
\eq
Now $L_1 \# L_2$ is defined by replacing the neighbourhood of $u$ with $H$:
\bq
L_1 \# L_2 := \big( ( L_1 \cup L_2 ) \backslash U \big) \cup \phi^{-1} \big(H \cap \phi(U) \big).
\eq
Up to Lagrangian isotopy, this only depends on our choice of ordering of $L_1$ and $L_2$. Nevertheless, define the \emph{parameter} of the surgery to be the area $\e$ between the image of $h$ and the union of the real and imaginary axes. (This will matter later for exactness.) Note $\e$ could be negative. Also, note that the intersection of the Lagrangian handle $H$ with the symplectic subspace $\R^2 \times \{ (0,0 )\}$ is $\{ \pm(x,y,0,0) | (x,y) \in \text{Im} (h) \} $ -- this will be used for e.g.~Figure \ref{fig:RSlocalfourdiscs} later.

\subsection{Main construction}

The vanishing cycles $A$ and $B$ intersect at two points, $u$ and $v$. 
As the signs of the two intersection points a, performing Lagrangian surgeries at both $u$ and $v$ gives a Lagrangian torus. This does not depend on the order chosen for each of the surgeries; on the other hand, if the two orientations differed, the result of the surgeries would always be a Klein bottle.

Choose two Darboux charts for each of $u$ and $v$ (i.e.~four charts in total): one for the surgery with order $(A,B)$, and one for the surgery with order $(B,A)$. 
By Assumption \ref{ass:productnhood} (product symplectic form near $M_\star$), we can take our Darboux charts at $u$ to be given by the product of a Darboux chart for some open neighbourhood of $u$ in $M_\star$ with a Darboux chart for a neighbourhood of $\pi(u)$ in $B_r(\star)$. We order our coordinates so that open subsets of fibres correspond to subsets of $\R^2 \times \{ (a,b)\}$, and lifts of open neighbourhoods of $\star$ to subsets of $\{ (0,0) \} \times \R^2$. (In particular, $\pi$ is projection to the final two coordinates.) We pick similar Darboux charts for $v$.

\subsubsection{Exactness}

Define $f_A$ to be a smooth function on $A|_{M_\star}$ such that $df_A = i^\ast _{A|M_\star} \theta$, and similarly for
  $f_B$. 
  
  \begin{proposition}
  Lagrangian surgeries on $u$ and $v$ with the same ordering of $A$ and $B$, and the same parameter $\e_u=\e_v=\e$, produce  an exact torus if and only if 
  \bq
  f_A (u)- f_B(u) = f_A(v)- f_B(v).
  \eq
  \end{proposition}
  
  \begin{proof}
  It's equivalent to find conditions under which $\theta$ integrates to zero about two distinct primitive simple closed curves. The meridional $S^1$ of the torus (which vanishes when deforming back to $A\cup B$) bounds a Lagrangian disc (which can be viewed by shrinking the $S^1$ in the handle $H$ of the surgery), so $\theta$ integrates to zero by Stokes' theorem. It remains to check the other direction. Let $S^1_l$ be any such curve.  It can be taken to be inside $M_\star$  near the surgery region.
 Assume we choose the surgery order $(A,B)$. Integrating along a curve that goes from $u$ to $v$ on $A$, then $v$ to $u$ on $B$, and correcting for the surgery, we find that
  \bq
  \int_{S^1_l} \theta = f_B(u)-f_B(v) + \e_u + f_A(v)-f_A(u) -\e_v.
  \eq
  In the case where $\e_u=\e_v=\e$, setting this expression equal to zero yields the desired result.
  \end{proof}
  
  Let $r_A$ and $r_A$ be the intersection points of $R_1$ with, respectively, $A$ and $B$. 
  Let $D_1$ and $D_2$ be the holomorphic discs between $u$, $r_A$ and $r_B$, and $v$, $r_A$ and $r_B$ (shaded in Figure \ref{fig:RS2localdiscs}).  

\begin{figure}[htb]
\begin{center}
\includegraphics [scale=0.85]{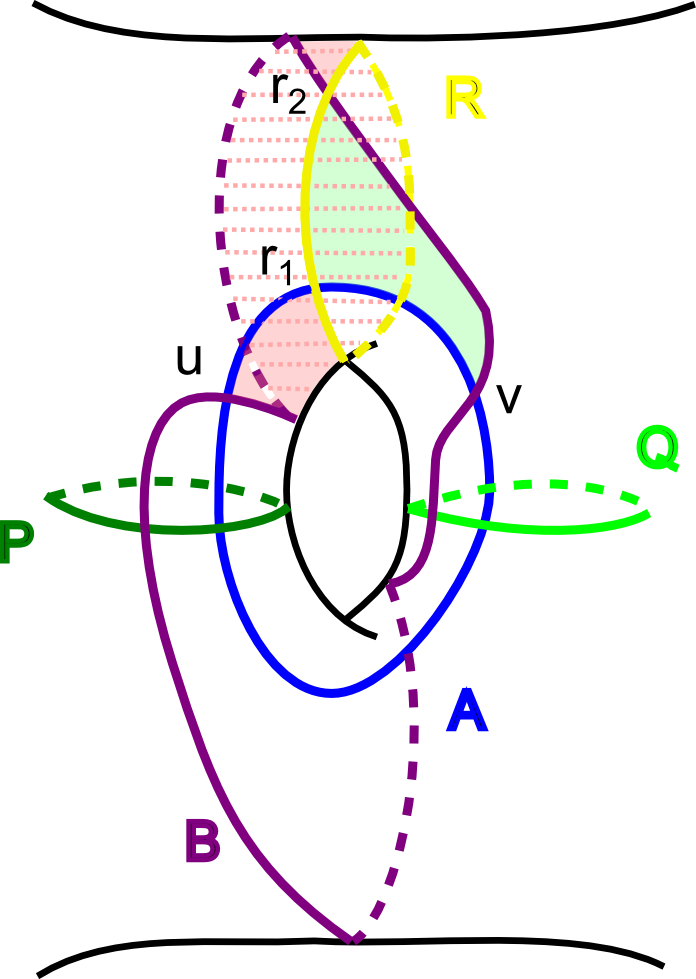}
\caption{Discs $D_1$ and $D_2$.}
\label{fig:RS2localdiscs}
\end{center}
\end{figure}
    Using Stokes' theorem, we get the following:
  
   \begin{corollary}
  Lagrangian surgery on $u$ and $v$ with same orderings, and same surgery parameter $\e$, produces an exact torus if and only if the discs $D_1$ and $D_2$ have the same symplectic area.  
  \end{corollary}
  
Notice that after exact Hamiltonian isotopy of the vanishing cycles on the Riemann surface $M_\star$, we can arrange for the two holomorphic discs $D_1$ and $D_2$ to have the same symplectic areas, while keeping minimal intersection between the curves. Thus, after an essentially local change of the symplectic form, we can arrange for this to be the case.
We shall assume hereafter that these areas are equal, for any of the $T_{p,q,r}$. 
  
 \subsubsection{Maslov class}
 
 \begin{lemma}\cite[Lemma 2.14]{Seidel00}
 Suppose you have two graded Lagrangians $\tilde{L}_0$ and $\tilde{L}_1$ which intersect transversally  at a single point $x$. If  we have
 \bq
 \tilde{I} (\tilde{L}_0, \tilde{L}_1; x) = 1
 \eq
 then there is a grading on $L_0 \# L_1$ which agrees with $\tilde{L}_0$ on $\tilde{L}_0 \cap (L_0 \# L_1)$, and with   $\tilde{L}_1$ on $\tilde{L}_1 \cap (L_0 \# L_1)$.
 \end{lemma}
 
 As a  corollary, we get:
 
 \begin{corollary}\label{th:Maslovsurgery}
 Suppose you have two Lagrangians $L_0$ and $L_1$ with vanishing Maslov classes, and which intersect transversally at two points $u$ and $v$, with agreeing orientations. Let $\Sigma$ be the result of the Lagrangian surgeries at $u$ and $v$, with same order $(L_0, L_1)$. Then $\Sigma$ has vanishing Maslov class if
 \bq
  \tilde{I} (\tilde{L}_0, \tilde{L}_1; u) - \tilde{I} (\tilde{L}_0, \tilde{L}_1; v)  =0.
 \eq
(The left hand side is independent of the choices of graded lifts $\tilde{L}_0$ and $\tilde{L}_1$.)
 \end{corollary} 
 
  In particular, we get the following result.
 
 \begin{lemma}
 Suppose we do Lagrangian surgery at $u$ and $v$ with the same order for $A$ and $B$. Then the resulting torus has vanishing Maslov class. 
 \end{lemma}
  
 \begin{proof}
 Pick a path $\gamma_1$  from  $u$ to $v$ along $A$, and $\gamma_2$  from  $u$ to $v$ along $B$. We  choose the ones that lie on the Riemann surface $M_\star$, and go through $r_A$ or $r_B$ exactly once. 
Deform a trivialization of the tangent bundle of the Milnor fibre so that when restricted to the tangent space of $\pi^{-1}(B_\star)$, for some small open neighbourhood of $\star$, it is given by the product of the standard trivialization of the tangent space of the base with the restriction of the original trivialization to the tangent space of$M_\star$. 

Let $\mathcal{L}(2)$ be the Grassmannian of Lagrangian planes in $\R^4$. The path $\gamma_i$ induces a path $\Gamma_i$ in $\mathcal{L}(2)$.  By Corollary \ref{th:Maslovsurgery}, it's enough to show that the Maslov index of $\Gamma_1$ relative to $\Gamma_2$ vanishes. The two base components of the Lagrangians are always transverse.
Given a vanishing cycle on $M_\star$ (and the same trivialization of $TM_\star$ as above), the associated path in the Grassmanian of Lagrangian planes in $\R^2$ must be zero; by inspection, e.g.~further deforming the trivialization using parallel copies of $A$ and $R$ on a neighbourhood of $R$  (where we are using the notation of Figure \ref{fig:RS2localdiscs}), we see that we can arrange for the fibrewise components of the $\Gamma_i$ to also always be transverse.
In particular, the relative Maslov index of the paths that we are considering vanishes.
 \end{proof}

\subsubsection{Conclusion}

Putting together the results from the previous subsections, we now see that we have:
\begin{theorem}\label{th:toriTpqr}
There exists an exact Lagrangian torus $T$ of vanishing Maslov class in the Milnor fibre of any  $T_{p,q,r}$ singularity. Its homology class is given by the difference of the classes of the vanishing cycles $A$ and $B$.
\end{theorem}

As a corollary, we obtain:

\begin{theorem}\label{th:tori}
The Milnor fibre of any positive modality isolated hypersurface singularity of three variables contains an exact Lagrangian torus, primitive in homology, and with vanishing Maslov class.
\end{theorem}

\begin{proof}
This follows from Theorem \ref{th:toriTpqr}, together with Theorem  \ref{th:adjacenttoparabolic} (any positive modality singularity is adjacent to a parabolic singularity, which are the three simplest of the $T_{p,q,r}$), and Lemma \ref{th:adjacentembedding} (if a singularity $f$  is adjacent to another one, say $g$,  there is an exact embedding of their Milnor fibres: $M_g \hookrightarrow M_f$). 
The homology class of the torus $T$ in $\mathcal{T}_{p,q,r}$ is the difference of the classes of two vanishing cycles. This remains true under the embedding given by adjacency (Lemma \ref{th:adjacentembeddingcycles}). Thus all the tori we construct have primitive homology classes.
This leaves the Maslov class claim. 
Let $f$ be positive modality singularity under consideration, and $g$ a parabolic singularity that it is adjacent to.
As $c_1(M_f)=0$, there is a lift from Lagrangian Grassmanian $LGr(M_f)$ to the Grassmanian of graded Lagrangian planes, $\widetilde{LGr}(M_f)$. 
This restricts to a lift from  $LGr(M_g)$ to $\widetilde{LGr}(M_g)$. Consider any closed path on $T$. This gives a path $\gamma: S^1 \to LGr(M_g)$. 
As $T$ has Maslov class zero in $M_g$, $\gamma$ lifts to a closed path $ S^1 \to \widetilde{LGr}(M_g)$. 
 (Of course, the lift of $LGr(M_g)$ to $\widetilde{LGr}(M_g)$ might be different from the one
  coming from trivialization of $T(M_g)$ used to calculate the Maslov class of $T$ -- but $\gamma$ will lift to a closed path for \emph{any} lift of $LGr(M_g)$ to $\widetilde{LGr}(M_g)$.) 
 Thus the image of $T$ in $M_f$ also has vanishing Maslov class.
\end{proof}

\subsection{A local model for the Lagrangian surgeries}\label{sec:localmodel}

For some Floer--theoretic computations, it will later be useful to have the following local model for the Lagrangian surgeries we make. 
Consider the Lefschetz fibration
\begin{eqnarray}
\chi: & \mathcal{C}:=\{ (x,y,z) \in \C^2 \times \C^\ast \, | \, x^2+y^2+z^2=1\} & \to \C^\ast \\
& (x,y,z) &\mapsto z.
\end{eqnarray}
The smooth fibre is a cone. There are two critical values, $z= \pm 1$, and the fibre above each of them is the union of two lines.  Equip this with with the exact symplectic form associated to the plurisubharmonic function
\bq
h(x,y,z) = |x|^2 + |y|^2 + (\text{log}|z|)^2 .
\eq
In particular, we have
\bq
\omega = \frac{i}{2} \Big(  dx \wedge d \bar{x} +  dy \wedge d \bar{y}  + \frac{ dz \wedge d \bar{z}} {z \bar{z}}\Big).
\eq
(The reader might want to think of the $z$ coordinate as an infinite annulus. Setting $z = e^{s+it}$, the $z$--term of $\omega$ is  $ds\wedge dt$, the standard symplectic form on $T^\ast S^1$.) 
Also, the $z$--terms of $\theta$ (say $\theta_z$) are a multiple of $\text{log} |z|$. In particular, 
\begin{equation} \label{eq:exactnessforlocalmodel}
\int_{|z|=1} \theta_z =0.
\end{equation}
The two unit half-circles in $\C^\ast$ give matching paths between $z=1$ and $z=-1$. (See Figure \ref{fig:localmodel}.)
\begin{figure}[htb]
\begin{center}
\includegraphics[scale=0.75]{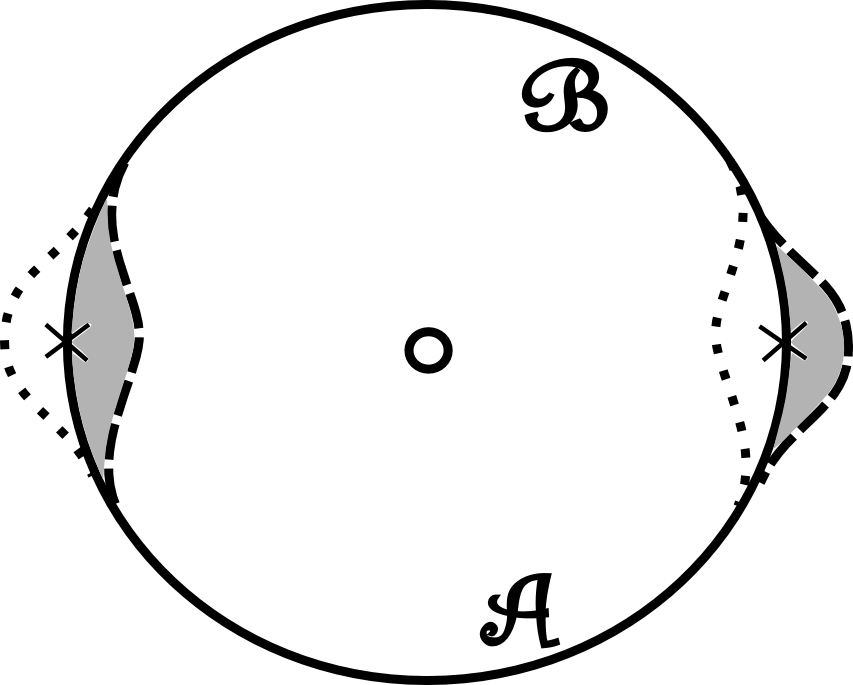}
\caption{Local model for Lagrangian surgeries on the intersections of $A$ and $B$.
}
\label{fig:localmodel}
\end{center}
\end{figure}
Using symplectic parallel transport, to each of these corresponds a Lagrangian sphere in $\mathcal{C}$.   (See e.g.~\cite{KhovanovSeidel}.) By for instance symmetry considerations, we do not even need to modify the symplectic form to get these  -- see e.g.\cite{Auroux07}, Section 5.1, for a similar argument. Call them $\mathcal{A}$ and $\mathcal{B}$. A Darboux-type argument gives:

\begin{proposition} There is an exact symplectomorphism from an open neighbourhood of the union of $A$ and $B$ to an open neighbourhood of the union of $\mathcal{A}$ and $\mathcal{B}$, mapping $A$ to  $\mathcal{A}$ and $B$ to  $\mathcal{B}$. 
\end{proposition}

(Exactness follows from Equation \ref{eq:exactnessforlocalmodel}, which ensures both intersection points have equal action.)
Near $1$ or $-1$, replace the union of the two matching paths by a curve segment avoiding the singular value (in each case, either to its left or its right). 
These are the dotted and dashed segments on our figure. The result, an $S^1$ in the base, gives a Lagrangian torus in the total space by taking the symplectic parallel transport of the vanishing cycle around $S^1$ (with no need to modify the symplectic form, again by symmetry considerations). 
 Up to Lagrangian isotopy, there are four choices. 
 If you make the perturbations  towards to same side (either left or right -- both dashes or both dots), the result will be exact if and only if the two displaced areas in the base agree. 
 (These are the shaded regions in our figure. Their areas are calculated with respect to the symplectic form given by the $z$--terms of $\omega$.) One can check that:

\begin{proposition}
The four Lagrangian tori obtained by surgery on $u$ and $v$ corresponds to the four matching tori described above. Moreover, the displaced area and the surgery parameter are monotone continuous functions of each other (cf \cite[Section 2.2]{LekiliMaydanskiy}). 
\end{proposition}

\subsection{Tori in parabolic singularities}

\subsubsection{Main result and discussion}

\begin{proposition} \cite{Gabrielov1}
The three parabolic singularities have a semi-definite intersection form, with a two-dimensional nullspace.
\end{proposition}

So far, we've been considering the null-class given by $[5]-[4]$ in the notation of the Dynkin diagram of Figure \ref{fig:Tpqr'} -- that is, the class $[A]-[B]$ in the notation of e.g.~Section \ref{sec:Tpqrvcycles}. 
There is a second independent class with a nice description in terms of the Dynkin diagram  \ref{fig:Tpqr'}: after quotienting out the class $[A]-[B]$,
it is given by a weighted combination of vertices of the quotient diagram. See Figure \ref{fig:T632}  for the case of $T_{6,3,2}$. (While the reader might find the weights instructive, we shall not use them to prove the theorem below.)

\begin{figure}[htb]
\begin{center}
\includegraphics[scale=0.7]{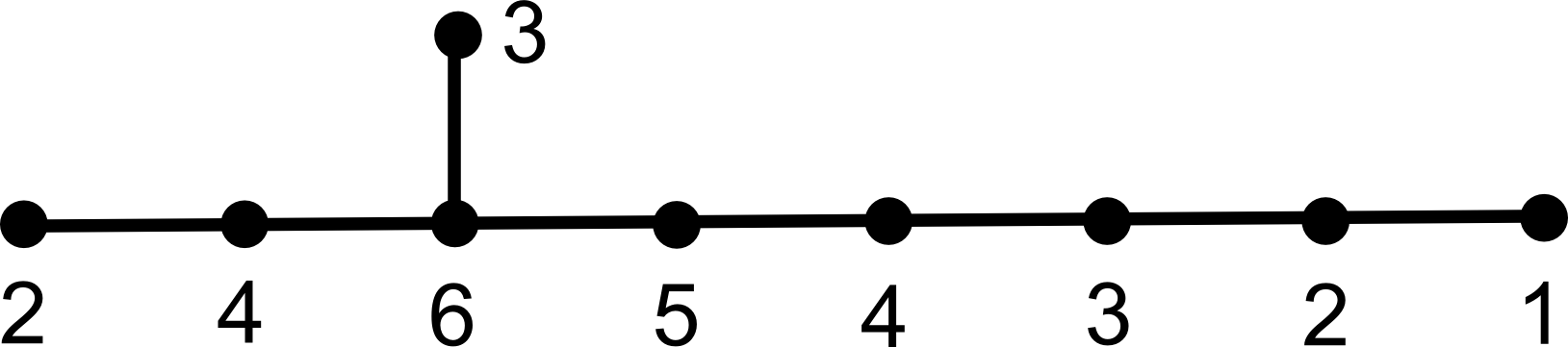}
\caption{Second null-class for the intersection form of $T_{6,3,2}$. The integers denote weights.
}
\label{fig:T632}
\end{center}
\end{figure}

Was the null-class that we constructed our torus in special in any way? 

\begin{thm}\label{th:allparabolic}
For parabolic singularities, there is an exact Lagrangian torus with vanishing Maslov class in each primitive homology class in the null-space of the intersection form.
\end{thm}

This means that there are tori in all the topologically permissible primitive homology classes for parabolic singularities. 
The case of non-primitive classes remains open.

The proof presented here uses compactifications of the three relevant Milnor fibres to del Pezzo surfaces, which might also prove of independent interest.

\subsubsection{Points in almost general position and del Pezzo surfaces}

 Let us start with some preliminaries. 
 
  \begin{definition}
  A del Pezzo surface is a smooth projective surface  with an ample anti-canonical bundle.  The rank of a del Pezzo is the self-intersection of its anti-canonical class. 
    \end{definition} 
    
 To check that a surface is del Pezzo, it is enough to calculate the self-intersection of an anti-canonical divisor $D$, and to check that for any irreducible curve $C$, $D\cdot C >0$ (Nakai--Moishezon criterion). Moreover, there is a short classification of these surfaces by rank, using the following notion:

\begin{definition}\label{def:gnlposition}
A collection of points $\Gamma$ in $\P^2$ are in general position if
no three points lie on a line; no six points lie on a conic; and
 there does not exist a cubic curve passing through seven of the points and with a double point at the eighth. (See e.g.~\cite[Expos\'e II]{777}.)
\end{definition}
Note that there can be at most eight points in general position in $\P^2$. We shall use:

\begin{theorem}\cite[Expos\'e II, Theor\`eme 1]{777}
Any rank one del Pezzo surface is the blow-up of $\P^2$ at eight points in general position; rank two, seven points; and rank three, six points. 
\end{theorem}

Additionally, we shall make use of the following results about points in general position.

\begin{lemma}\cite[Expos\'e III, Theor\`eme 1]{777}
Consider any collection of points in $\P^2$ in general position. Then there is a smooth cubic curve that contains all of them.
\end{lemma}

Conversely, fix any smooth cubic curve $E$ in $\P^2$, and an integer $n$, $1\leq n \leq 8$. Consider the collection of all sets of $n$ distinct (unordered) points on $E$. This is an algebraic variety: $(E^n-V)/Sym^n$, where $V$ is a closed subvariety of $E^n$. 
\begin{lemma}
The collection of sets of $n$ points on $E$ in  general position, say $E^n_{gp}$, is an open, connected, non-empty subvariety of  $(E^n-V)/Sym^n$.
\end{lemma}

\begin{proof} Points satisfying each of the conditions in Definition \ref{def:gnlposition} give proper closed subvarieties of $E^n$, invariant under the action of $Sym^n$. 
\end{proof}

We shall actually need a version of this for a one-parameter family of elliptic curves. 

\begin{definition}
Let $\mathcal{E}$ be the collection of all smooth cubic curves in $\P^2$ (that is, the space of coefficients, with the discriminant, corresponding to singular curves, removed).
\end{definition}

As a smooth complex variety, $\E$ has a natural smooth structure. Let $\lambda: S^1 \to \mathcal{E}$ be any smooth loop (we use the parametrization $S^1=\R/\Z$). Suppose that you have a collection $\Gamma_0$ of $n$ points in general position on $\lambda(0)$. 

\begin{lemma}\label{th:gnlloop}
We can find a smooth map 
\bq
\Gamma: \bigsqcup_{n} \, [0,1] \to \P^2
\eq such that $\Gamma(\{0 \}, \ldots \{0\})=\Gamma(\{1 \}, \ldots \{1\}) = \Gamma_0$  (set-wise), and $\Gamma(\{t \}, \ldots \{t\})$  is a collection of points in general position in the cubic curve $\lambda(t)$ for every $t$.
\end{lemma}

\begin{proof} Starting with $\Gamma_0$, an extension to a path of $n$ points certainly exists. Also, for any cubic $E$, $(E^n-V)/Sym^n$ is a smooth complex manifold, and $E^n_{gp}$ a (Zariski) open subvariety. Thus, given any path of cubic curves, the space of paths of $n$ points in almost general position is a (topologically) dense open subset of the space of paths of $n$ points. Moreover, in the case of a loop of cubic curves, as $E^n_{gp}$ is connected, we can arrange for the set of $n$ points at $t=1$ to match the initial one.
\end{proof}

\subsubsection{Compactification to del Pezzo surfaces}

Each of the parabolic singularities 
 has an isolated singularity at the origin, and no other singular values. Thus by Remark \ref{rk:weightedcutoffs},   the Milnor fibre of $x^3+y^3+z^3$ is represented by the hypersurface 
 \bq x^3+y^3+z^3+1=0
 \eq
 and similarly for the other two.

\begin{prop} We have that:
\begin{itemize}
\item $x^3+y^3+z^3+1=0$ compactifies to a rank three del Pezzo surface inside $\P^3$.

\item  $x^4+y^4+z^2+1=0$ compactifies to a rank two del Pezzo surface inside the weighted projective space $\P^3(1,1,2,1)$.

\item $x^6+y^3+z^2+1=0$ compactifies to a rank one del Pezzo surface inside the weighted projective space $\P^3(1,2,3,1)$.

\end{itemize}
In each case, the anti-canonical divisor given by intersecting with the hyperplane at infinity is a smooth elliptic curve.
\end{prop}

\begin{proof}

{\bf Case of (3,3,3).}
Compactify $x^3+y^3+z^3+1=0$ inside $\P^3$ to $x^3+y^3+z^3+w^3=0$, say $\overline{P}_8$. The divisor  at infinity (i.e.~the intersection with $w=0$), say $D$, is an elliptic curve: a cubic inside $\P^2$. It is an anti-canonical divisor: the form defined by 
\bq \label{eq:trivialcanonicalbundle}
\Omega \wedge df= {dx \wedge dy \wedge dz} \eq
where $f(x,y,z)=x^3+y^3+z^3$, trivialises the canonical bundle of its complement.
As $D$ is anticanonical, we have that 
 $D \cdot D =3$. As the complement if $D$ is affine, any irreducible curve $C$ in $\overline{P}_8$ must intersect $D$. By positivity of intersection, we have  that  $ D \cdot C >0 $.

{\bf Case of (4,4,2).}
Compactify $x^4+y^4+z^2+1=0$ inside the weighted projective space $\P^3(1,1,2,1)$ to $x^4+y^4+z^2+w^4=0$, say $\overline{X}_9$. This avoids the singular points of the weighted projective space. Again, the intersection with $w=0$ is an anticanonical divisor, say $D$ (because the complement, an affine hypersurface, has trivial canonical bundle -- one can still use Equation \ref{eq:trivialcanonicalbundle}). It's also an elliptic curve; one can for instance use the fact that it is branch-double-covered by the quartic $x^4+y^4+z^4$ in $\P^2$, together with the Riemann-Hurwitz theorem for curves. We claim the total space is a del Pezzo surface of rank 2. Here is an elementary way of computing the self-intersection of $D$: consider the map 
$ \pi: \P^3(1,1,2,1) \backslash [0:0:1:0] \to \P^2$ given by $[x:y:z:w] \mapsto [x:y:w]$. The surface $\overline{X}_9$ double covers $\P^2$, with branching over a quartic curve $Q$. 
By the generalized Riemann--Hurwitz theorem, 
\bq
K_{\overline{X}_9} = \pi^\ast K_{\P^2} + R
\eq where $R$ is the
ramification divisor upstairs. Notice that $R$ has the same support as $\pi^\ast Q$, but different multiplicity: $\pi^\ast Q = 2R$. Thus 
\bq
2K_{\overline{X}_9} = \pi^\ast(2 K_{\P^2} + Q) = \pi^\ast 
\mathcal{O} (-2)
\eq
Thus $-K_{\overline{X}_9}$ is ample, and \bq D \cdot D = \frac{1}{4}\text{ deg}(\pi) \times 4 = 2.\eq

{\bf Case of (6,3,2).} Similarly, compactify $x^6+y^3+z^2+1=0$ inside the weighted projective space $\P^3(1,2,3,1)$ to $x^6+y^3+z^2+w^6=0$, say $\overline{J}_{10}$. The anti-canonical divisor given by the intersection with $w=0$, say $D$, is still an elliptic curve (for instance, one can consider its branch-covering by the sextic $x^6+y^6+z^6$ in $\P^2$). The total space is a del Pezzo surface of rank 1. Why? You can use a different auxiliary projection to $\P^2$: the map $\pi: \P^3(1,2,3,1) \backslash [0:0:1:0] \to \P^2$ given by $[x:y:z:w] \mapsto [x:y^2:w]$. The surface $\overline{J}_{10}$ double covers $\P^2$, with branching over two copies of $\P^1$. Proceeding similarly to the case of $\overline{X}_9$, we find that $K_{\overline{J}_{10}} = \pi^\ast \mathcal{O}(-1)$.
\end{proof}

\subsubsection{Variation operator and the nullspace of the intersection form}

Our next ingredient is a characterization on the nullspace of the intersection form. 
Let $S$ be the compactification of a parabolic Milnor fibre to a del Pezzo surface, $D$ the preferred anti-canonical divisor in $S$, and $M^c:=S \backslash \nu(D)$ the complement of a neighbourhood of $D$ (that is, the compact version of the Milnor fibre $M=S \backslash D$).

\begin{prop}\label{th:nullclasses}
Let $i: H_2(\partial M^c) \to H_2(M)$ be the map induced by inclusion. It is injective, and the classes in the nullspace of the intersection form are precisely its image. 
\end{prop}

\begin{proof}

The space $M^c$ is a retract of $M$ (which itself can be obtained by gluing a half-infinite cylinder to the boundary of $M^c$). Thus it is equivalent to prove this for $i:  H_2(\partial M^c) \to H_2(M^c)$. Injectivity follows from the long exact sequence of the pair $(M^c, \partial M^c)$, together with the fact that $ H_3 (M^c, \partial M^c)  \cong H^1(M^c) = 0$. 
For $\a \in H_2 (M^c)$, the form $\langle \cdot, \a \rangle \in Hom (H_2(M^c), \Z)$ is given by the image of $\a$ under the standard maps 
\bq
H_2 (M^c) \to H_2 (M^c, \partial M^c) \cong (H_2 (M^c))^\ast.
\eq
Using the long exact sequence of the pair $(M^c, \partial M^c)$ again, we see that $\langle \cdot, i(\b) \rangle$ vanishes for any $\b \in H_2 (\partial M^c)$. 
Moreover, the same long exact sequence shows that if the form $\langle \cdot, \a \rangle$ is identically zero, then there must exist $\beta \in H_2 (\partial M^c)$ such that $i (\beta) = \a$. 
\end{proof}

\subsubsection{Monodromy of loops of smooth cubics curves}
Here is the final ingredient.

\begin{lemma}\label{th:makeloop}
Fix a loop $\lambda: S^1 \to \mathcal{E}$, and let $E=\lambda(0)$. The isomorphism $H_1(E) \to H_1(E)$ that this induces can correspond to any element  $\gamma \in SL_2(\Z)$. 
\end{lemma}

\begin{proof} This follows from the usual action of $SL_2(\Z)$ on the upper-half plane. As the space of cubic curves is path-connected, it is enough to understand this for any fixed cubic curve. Pick a point $\tau$ in the upper-half plane. It determines an (abstract) elliptic curve $F$. Three distinct ordered marked points (say $p$, $q$ and $r$) in the fundamental domain associated to $\tau$ specify an embedding into $\P^2$: they determine a degree three line bundle on $F$ with a preferred ordered basis of global sections. 
(Note that if $E$, with the data of its embedding into $\P^2$ -- rather than as an abstract curve -- has e.g.~a single triple intersection point with the hyperplane $\{ x=0 \}$, that embedding cannot arise from this construction.)
 Remember $\gamma$ is any element of $SL_2(\Z)$. 
We would get the same cubic in $\P^2$ by using $\gamma (\tau)$ and marked points $\gamma(p)$, $\gamma(q)$ and $\gamma(r)$. Now pick a path $\tau(t)$ in the upper-plane between $\tau(0)=\tau$ and $\tau(1)=\gamma(\tau)$, and paths of points $p(t)$, $q(t)$ and $r(t)$ in the fundamental domain of $\tau(t)$, with similar conditions. This determines a path of cubic curves in $\P^2$; by construction, it has the required property.
\end{proof}

\begin{remark} There is more to the homotopy group $\pi_1(\E)$ than $SL_2(\Z)$: for instance, there is a semi-direct product with $\Z/3 \times \Z/3$ (think about which triples of points determine the same embedding), which doesn't get detected by the action on $H_1$. For a full description of $\pi_1(\mathcal{E})$, see \cite{Lonne}. 
\end{remark}

\subsubsection{Conclusion of argument}

As before, let $S$ be a del Pezzo surface  that is the compactification of a parabolic Milnor fibre, and let $D$ be the preferred anti-canonical divisor in $S$. Blow down six, seven or eight  (say $n$) exceptional curves on $S$ to get to $\P^2$. The exceptional curves blow down to a collection of $n$ points in almost general position in $\P^2$, say $\Gamma$.
 Moreover, $D$ is the proper transform on a cubic curve $E$ in $\P^2$ passing through all points of $\Gamma$.
 Now given any element $\gamma \in SL_2(\Z)$, by Lemma \ref{th:makeloop}, we can find a loop $\lambda: S^1 \to \E$ with $\lambda(0)=E$ such that the resulting isomorphism $H_1(E) \to H_1(E)$ is given by $\gamma$. Moreover, by Lemma \ref{th:gnlloop}, we can smoothly extend $\Gamma$ to a one-parameter family of $n$ points in almost general position $\Gamma(t) \subset \lambda(t)$. 

Now consider $S^1 \times \P^2$. In $\{t \} \times \P^2$, blow up $\Gamma(t)$. The result is a smooth $S^1$--bundle with fibre a del Pezzo surface of fixed rank. Moreover, each fibre, say $S(t)$, comes with a complex structure and favourite anti-canonical divisor, say $D(t)$, which is the proper transform of the cubic $\lambda(t)$. Notice that $S(0)=S(1)=S$, and similarly with $D$. Call this fibre bundle $\mathcal{S}$.

Choosing smooth families of global sections (holomorphic above each point in $S^1$), we can construct an embedding $\mathcal{S} \subset S^1 \times \P^k$ such that $S(t) \subset \{ t \} \times \P^k$ is a projective embedding determined by the square of the anti-canonical bundle. Moreover, composing with a loop in $PGL_{k+1}(\C)$, we can assume that the image of $D(t)$ is always the intersection of $Im(S(t))$ with the hyperplane $x_{k+1}=0$.  Removing these, we get a smooth $S^1$--bundle $\M$ with fibres diffeomorphic to $M$, the Milnor fibre we started with. Moreover, each fibre $M(t)$ comes equipped with an embedding to $\C^k$; in particular, it comes with a favourite exact symplectic form, inherited from the Kaehler form on $\C^{k}$. By construction, $M(0)=M(1)$ is just the Milnor fibre $M$ we started with (as an exact symplectic manifold). 

\begin{claim} 
The monodromy map $f: M \to M$ of $\mathcal{M}$ acts on the nullspace of the intersection form by our chosen element $\gamma \in SL_2(\Z)$.
\end{claim}

\begin{proof} 
By construction, we know that $\gamma$ gives a map $H_1(E) \to H_1(E)$; within the fibre bundle $\mathcal{S}$, one can consider the smooth automorphism of $D$ obtained by following the path of favourite anti-canonical divisors $D(t)$. By construction, the action on first homology is also given by  $\gamma$. Now consider the compact version of the Milnor fibre $M^c:= S \backslash \nu(D)$. (This is the same notation as Proposition \ref{th:nullclasses}.) We claim that there is a natural isomorphism $H_1(D) \cong H_2(\partial M^c)$. Consider the Gysin sequence associated to $D$ and $\partial M^c$. We get that
\bq
0 \to H^1 (D) \to H^1(\partial M^c) \to H^0(D) \to H^2(D) \to \ldots
\eq
We know that $H^0 (D) \cong \Z \cong H^2 (D)$, and that the above map $H^0(D) \to H^2(D)$, given by cupping with the Euler class, is injective. Thus there is a natural isomorphism $H^1(D) \cong H^1 (\partial M^c)$, and, by applying Poincar\'e duality on both sides, a natural isomorphism $H_1(D) \cong H_2 (\partial M^c)$. 

In particular, the monodromy action on $H_1(D)$ is the same as the monodromy action on $H_2(\partial M^c)$. By 
Proposition \ref{th:nullclasses}, we are done. 
\end{proof}

The final step is to show that we can use this monodromy to take the exact Lagrangian torus that we already have to another one. Of course, the monodromy is only a smooth map; we need to proceed with a little caution.

\begin{claim} Start with the exact Lagrangian torus $T \subset M$ of Theorem \ref{th:tori}. We claim we can find an exact Lagrangian torus $T'$ that is isotopic to $f(T)$, and has vanishing Maslov class.
\end{claim}

\begin{proof}  Let $\omega(t)$ be the Kaehler symplectic form on $M(t)$, and set $\omega = \omega(0) = \omega(1)$ on $M$. 
Pick a smooth family 
\bq
\psi: [0,1]\times M \to \mathcal{M}
\eq such that
$\psi(t)$ is a diffeomorphism $M \to M(t)$, with $\psi(0)=\text{Id}$ and $\psi(1)=f$. 

Now $\omega_t:= (\psi(t))^\ast \omega(t)$ is a smooth one-parameter family of symplectic forms on $M$ ($t\in [0,1]$) such that $\omega_0 = \omega(0)$ and $\omega_1 = f^\ast (\omega)$. We can associate to this a Moser vector-field $X(t)$ on $M$, whose flow, \emph{when defined}, will preserve exact Lagrangians. 

One each fibre $M(t)$, the Liouville vector field $Z(t)$ of $\omega(t)$ is just the restriction of the gradient vector field of $\sum |x_n|^2$ (the square of the distance function). In particular, for any sufficiently large $R$, $M(t)$ is transverse to the sphere $S(R)$, and $Z(t)$ points outwards on the boundary of $M(t)^R := M(t) \cap B_0(R)$ for all $t \in [0,1]$.

Pick an increasing smooth  function $\sigma: [0,1] \to \R_+$ such that $\sigma(0)=0$ and for any $t_1 < t_2$,
\bq
(\psi(t_1))^{-1} \big(M(t_1)^{\sigma(t_1)+R} \big) \subset 
(\psi(t_2))^{-1} \big(M(t_2)^{\sigma(t_2)+R} \big) .
\eq
We claim that one can find sufficiently large constants $R$ and $c$ such that $T \subset \text{Int} (M(0)^R)$, and the vector field 
\bq
X(t) - (\psi^{-1})_\ast \big( c Z(t) \big)
\eq
points inwards along the boundary of $M(t)^r$ for all $r$ with $R \leq r \leq R+\sigma(1)$, and all $t$. (Why? Such constants certainly exist for any $t$; now use uniform continuity.)

Now consider the \emph{time-dependent} vector field $X(t) - (\psi^{-1})_\ast \big( c Z(t) \big)$. By construction, its flow is defined for any starting point in the interior of $M(0)^R$. Let $T'$ be the image of $T$; it is an exact Lagrangian with respect to $\omega_1$ (with the pull-back primitive one-form). Now $f(T')$  is an exact Lagrangian with respect to the original symplectic structure, and isotopic to $f(T)$. Finally, $T'$ has vanishing Maslov class (as do the images of $T$ at each time $t$), and so $f(T')$ has too. 
\end{proof}

This concludes the proof of Theorem \ref{th:allparabolic}. 


 
\section{The Fukaya category of unimodal Milnor fibres} \label{sec:Fukayaunimodal}

\subsection{Statements of results}

\subsubsection{Preferred torus}

For the remainder of this section, we focus on the torus $T$ obtained by making both surgeries at $u$ and $v$ with the order $(A,B)$, with equal parameters $\e$. The restriction of $T$ to the fibre $M_\star$ is given in Figure \ref{fig:RSlocalfourdiscs}.

\begin{figure}[htb]
\begin{center}
\includegraphics[scale=1.1]{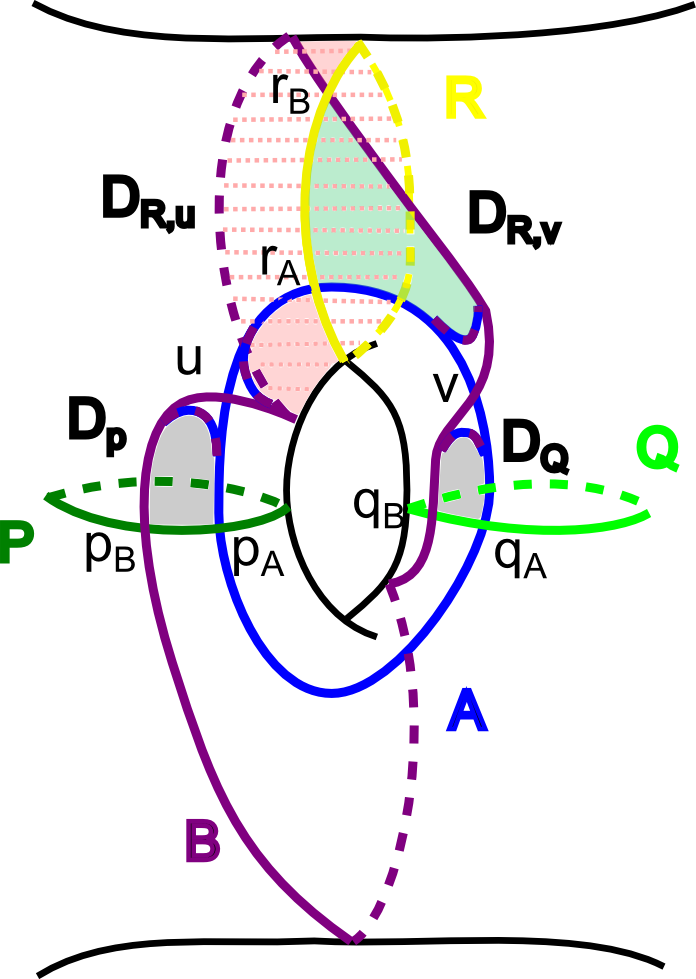}
\caption{Intersection of $T$ and $M_\star$: union of the $A$ and $B$ curves, with surgeries at $u$ and $v$. The intersection points with $P$ are $p_A$ and $p_B$, similarly for $Q$ and $R$. The shaded and dotted holomorphic discs are labelled $D_P$, $D_Q$, $D_{R,u}$ and $D_{R,v}$. 
}
\label{fig:RSlocalfourdiscs}
\end{center}
\end{figure}

In particular, there is a total of four visible holomorphic discs between $T$ and $P=P_1$, $Q=Q_1$ and $R=R_1$. We label these as in Figure \ref{fig:RSlocalfourdiscs}. They shall turn out to be the only discs. Recall that to identify the space of equivalence classes of pairs (spin structure, flat complex line bundles) on $T$ with $(\C^\ast)^2$, it is enough to pick a basis of $H_1(T)$. We do so as follows:
\begin{itemize}
\item $T$ has a preferred meridional direction (which vanishes when deforming to $A\cup B$). We take this as the first coordinate. 
\item For the second coordinate, we need to choose a longitudinal $S^1$. We choose the $S^1$ component of $T|_{M_\star}$ that is parallel to $R|_{M_\star}$ (top component in Figure  \ref{fig:RSlocalfourdiscs}).
\end{itemize}

\subsubsection{Floer cohomologies}

\begin{proposition}\label{th:floercohomology} Pick any lift $T$ to the Lagrangian Grassmanian of $\mathcal{T}_{p,q,r}$, say $\widetilde{T}$; similarly for  $A$, $B$ and $R_1$. 
The Floer cohomology between the associated branes is as follows:
\bq
HF^\ast \big( (\widetilde{T}, (\a,\b)), \widetilde{A} \text{ or } \widetilde{B} \big) =
\begin{cases}
 0 & \a \neq 1 \\
 H^{\ast+k} (S^1) & \a =1
\end{cases}
\eq
\bq
HF^\ast \big( ( \widetilde{T}, (\a,\b)), \widetilde{R_1} \big) =
\begin{cases}
 0 & \b \neq 1 \\
 H^{\ast+k} (S^1) & \b =1
\end{cases}
\eq
where the grading shift $k$ depends on the choices of absolute gradings. (For $\Z_2$--gradings, forget about the choices of lifts, and take $k=0$.)
Floer cohomology groups with all other vanishing cycles are identically zero.
\end{proposition}

The proof will be given later in this section.

\subsubsection{(Split-)generation of the Fukaya category}
\begin{thm}\label{th:nofukayageneration}
The Fukaya category of the Milnor fibre of $T_{p,q,r}$ is not split-generated by any collection of vanishing cycles.
\end{thm}

\begin{proof}This is a consequence of Proposition \ref{th:floercohomology} and Lemma \ref{th:nosplitgeneration}, together with the fact that 
\bq
HF^\ast \big( (\widetilde{T}, (\a, \b)) , ( \widetilde{T}, (\a, \b)) \big) \cong HF^\ast \big( (\widetilde{T}, (1, 1)) , (\widetilde{T}, (1, 1)) \big) \cong H^\ast (T^2) \neq 0 \eq
where the first isomorphism follows from definitions, and the second one from a Puinikhin--Salamon--Schwatz type argument; see e.g.~work of Albers \cite{Albers}. The $\Z_2$--graded case is similar.
\end{proof}

In particular, Seidel's result (Theorem \ref{th:Seidelweighted}) is strict. 

\begin{remark}  There are versions of Theorem \ref{th:nofukayageneration} for the following `flavours' of Fukaya categories:
\begin{itemize}
\item $\Z_2$--graded, where objects are pairs $(L, \mathfrak{s}_L)$, with $L$ compact orientable exact Lagrangian and $\mathfrak{s}_L$ a spin structure on it. 
\item absolutely $\Z$--graded, where objects are pairs $(\widetilde{L}, \mathfrak{s}_L)$, with $L$ compact orientable exact Lagrangian of Maslov class zero, $\widetilde{L}$ a lift of $L$ to the Lagrangian Grassmanian and $\mathfrak{s}_L$ a spin structure on $L$.
\end{itemize}
With our notation, in the case of $T$, this is the same as restricting ourselves to the cases where $\a = \pm 1$, $\beta = \pm 1$. 
\end{remark}

\subsection{Floer cohomology of $T$ with the vanishing cycles $P_1$, $Q_1$ and $R_1$}

Here we prove Proposition \ref{th:floercohomology} when the vanishing cycle involved is one of $P_i$, $Q_j$ or $R_k$. The result is immediate unless $i$, $j$ or $k$ is one.

Recall  that on an open neighbourhood $U$ of $K$, some large compact subset in $M_\star$, we have a product symplectic form $\o_{pr}$ and compactible almost complex structure $J_{pr}$. (The set $U$ is contained in $\pi^{-1}(B_{r/2}(\star))$. See Assumption \ref{ass:productnhood}.) The Darboux charts about $u$ and $v$ in which we perform Lagrangian surgeries are both contained in $U$.
With respect to these choices, there are four visible holomorphic discs between $T$ and these vanishing cycles: one with $P_1$, one with $Q_1$, and two with $R_1$. See Figure \ref{fig:RSlocalfourdiscs}. 
To get the claimed Floer cohomology computations, we will show that with respect to this almost complex structure:
\begin{itemize}
\item these are the only discs between $T$ and $P_1$, $Q_1$ or $R_1$;
\item the discs all have Maslov index one;
\item the discs are regular.
\end{itemize}
We will address these points one by one. Why is this enough? It is certainly the case for $P_1$ and $Q_1$, where there is only one disc. 
In the case of $R_1$,  the two discs contribute 
\bq
(1-\beta)r_A
\eq
to the differential
$
\partial r_B \in CF \big( R_1, (T, (\alpha, \beta) ) \big).
$

\subsubsection{Uniqueness of discs}\label{sec:uniqueness}
Give $\R \times [0,1] \subset \C$ the usual complex structure $i$, and consider a holomorphic map $\sigma: \R \times [0,1] \to \mathcal{T}_{p,q,r}$, such that
\begin{itemize}
\item $\R \times \{0\}$ maps to $P_1$ and $\R \times \{1\}$ maps to $T$ (Lagrangian boundary conditions)
\item $\text{lim}_{s \to -\infty} \sigma(s, \cdot) = p_1$ and $\text{lim}_{s \to -\infty} \sigma(s, \cdot) = p_2$ (asymptotic conditions). 
\end{itemize}
These are precisely the maps used to count the coefficient of $p_1$ in $\partial p_2 \in CF \big(P_1, ( T, (\a, \b) \big)$. The disc $D_P$ gives one such map (or, rather, a one parameter family of maps), say $\tau$. We would like to argue that there are no other maps. 

Suppose we have such a map $\sigma$.
First, using the open mapping theorem, we see that $\text{Im}(\sigma)$ must lie in $\pi^{-1}(B_{r/2}(\star))$.
 Also, as they only depend on $p_1$ and $p_2$, the symplectic areas for $\sigma$ and $\tau$ agree:
\begin{equation}\label{eq:areasagree}
\int_{\R \times [0,1]} \sigma^\ast \o_{pr} = \int_{\R \times [0,1]} \tau^\ast \o_{pr}.
\end{equation}
   Let $\pi_2: U \to M_\star$ be the projection to the fibre above $\star$. By construction, this is $(J_{pr}, J_\star)$--holomorphic, where $J_\star$ is the complex structure on $M_\star$. Also, $\pi_2$ projects $A$ onto $A|_{M_\star}$, and similarly for $B$ and $P_1$. Now notice that
\begin{equation}
\int_{\R \times [0,1]} \sigma^\ast \o_{pr} \geq \int_{\R \times [0,1]} (\pi_2 \circ \sigma)^\ast \o_{M_\star}.
\end{equation}
Moreover, $\pi_2 \circ \sigma: \R \times [0,1] \to M_\star$ is $(i, J_\star)$--holomorphic. By topological considerations, we see that $\mathit{Im} (\pi_2 \circ \sigma)$ must \emph{contain} $D_P$. More precisely, it must agree with $D_P$ everywhere apart from the neighbourhood $U\subset M_\star$ of $u$ where the surgery is performed. In $U$,  the boundary of $\mathit{Im} (\pi_2 \circ \sigma)$ must be some curve in the image of the handle $H$ under the projection $\pi_2$. In particular, the curve belonging to $T$, which already lies in $M_\star$, is the option that gives the smallest area.  Moreover, it can only be realised as a portion of the boundary of $\mathit{Im}(\pi_2 \circ \sigma)$ if the corresponding boundary portion of $\mathit{Im}(\sigma)$ already lies in $M_\star$. Thus the previous equation can be strengthened to:
\begin{equation}
\int_{\R \times [0,1]} \sigma^\ast \o_{pr} \geq \int_{\R \times [0,1]} (\pi_2 \circ \sigma)^\ast \o_{M_\star} \geq  \int_{\R \times [0,1]} \tau^\ast \o_{pr}
\end{equation}
The first equality holds if and only if $\text{Im } \sigma \subset M_\star$, which implies that the two equalities if and only if $\sigma$ and $\tau$ are the same map up to translation of the real coordinate on $\R \times [0,1]$. Equation \ref{eq:areasagree} implies that this must be the case.

The case of $D_Q$ is completely analogous. For $D_{R,u}$ and $D_{R,v}$, one proceeds similarly to show that the discs must lie in $M_\star$, though there are then two possibilities. (Both of them have the same symplectic area: we meet again our assumptions on the symplectic area of the discs $D_1$ and $D_2$, and the equality of the surgery parameters, which allowed exactness of $T$ in the first place.)

\subsubsection{Maslov index of discs}
Pick a trivialisation of the tangent space of $\T_{p,q,r}$ that on $\pi^{-1}\big(B_{r/2}(\star) \big)$, is given by the product of a trivialization of the tangent space of $M_\star$ with the standard trivialization of the base. Let us start by calculating the Maslov index of $D_P$.

Let $\g_1: [0,1] \to \T_{p,q,r}$ be the path from $p_A$ to $p_B$ on $P$ along the boundary of $D_P$, and $\g_2: [0,1] \to \T_{p,q,r}$ be the path from  $p_A$ to $p_B$ on $T$ along the boundary of $D_P$. These determine paths $\Gamma_i: [0,1] \to \mathcal{L}(2)$, where $\mathcal{L}(2)$ is the Lagrangian Grassmannian. 
The fibre-wise components of these paths contribute $+1$ to the Maslov index of $\Gamma_2$ relative to $\Gamma_1$. 

It remains to show that the base components of the paths do not contribute anything. Recall the Lagrangian handle $H$ (Section \ref{sec:lagrangiansurgery}) is given by 
\bq
H = \{ (x \, \text{cos } t , y \, \text{cos } t, x \, \text{sin }t , y \, \text{sin } t) \, | \, (x,y) \in \text{Im}(h), t \in S^1 \} \subset \R^4.
\eq
Under the identifications described in that section, $\gamma_2$ corresponds to the path
\bq
\{ (x , y , 0, 0) \, | \, (x,y) \in \text{Im}(h), t \in S^1 \}.
\eq
Tangent vectors to $H$ along $\gamma_2$ are given by $(t_1, t_2, 0,0)$, the extension of the tangent vector to $\text{Im}(h)$, and $(0,0,x,y)$. Thus the corresponding path in the Lagrangian Grassmannian of the base is described by Figure \ref{fig:baselagrangianpath}. It contributes zero to the relative Maslov index.

\begin{figure}[htb]
\begin{center}
\includegraphics[scale=0.85]{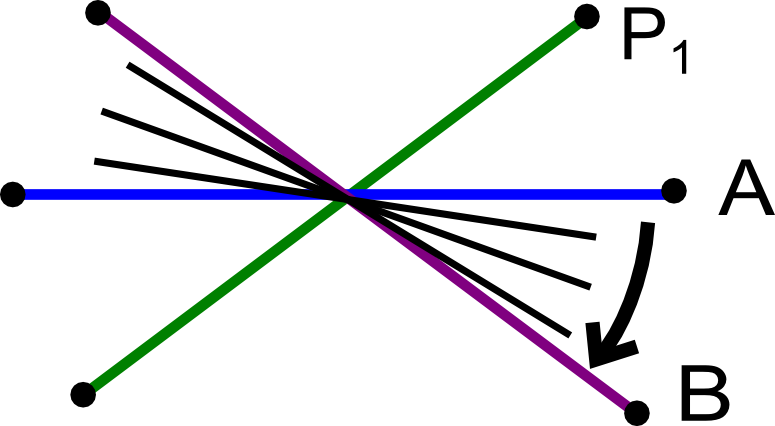}
\caption{Path of Lagrangian lines determined by $\gamma_2$, and Lagrangian line for $P_1$.}
\label{fig:baselagrangianpath}
\end{center}
\end{figure}
The other three discs are completely analogous. 

\subsubsection{Regularity of discs} Let $M$ denote $\mathcal{T}_{p,q,r}$, $L_0$ denote $P_1$ (or, depending on the disc we are considering, $Q_1$ or $R_1$),  and $L_1$ the torus $T$.
Let $R = \R \times [0,1]$, and fix $\sigma$ as above. The corresponding linearised Cauchy--Riemann operator $D_\sigma$ is a Fredholm operator:
\begin{multline}
D_\sigma: \, W^{k,p} \big(R \, ;\, \sigma^\ast (TM), \sigma^\ast (TL_0), \sigma^\ast(TL_1) \big) \\
\to W^{k-1,p} \big(R  \, ;  \, T^\ast R \otimes \sigma^\ast (TM), T^\ast R \otimes \sigma^\ast (TL_0), T^\ast R \otimes \sigma^\ast(TL_1) \big)
\end{multline}
for some $k$ and $p$. We want to show that $\sigma$ is regular, that is, that the operator $D_\sigma$ is onto. From our Maslov index calculation, we already know that this operator has index one. Thus it's enough to show that the kernel of $D_\sigma$ is one-dimensional. This means it must correspond to translations in the $\R$--direction of the domain of $R$, and nothing else. 
 Recall $D_\sigma$ is defined by extending an operator
\begin{multline}
D_\sigma: \, C^\infty \big(R \, ; \, \sigma^\ast (TM), \sigma^\ast (TL_0), \sigma^\ast(TL_1) \big) \\
\to C^\infty \big(R \, ; \, T^\ast R \otimes \sigma^\ast (TM), T^\ast R \otimes \sigma^\ast (TL_0), T^\ast R \otimes \sigma^\ast(TL_1) \big).
\end{multline}The kernel of the operator on the smooth spaces is the same as the kernel on the completion. (Why? This follows from boundedness in the injective case. In general, notice that the kernel in the $C^\infty$ space, which is finite dimensional, is already closed in the Sobolev norm. The question then reduces to the injective case.)
Consider a smooth one parameter family of $(j, J_{pr})$--holomorphic maps:
\bq
\sigma_t: \R \times [0,1] \to M = \mathcal{T}_{p,q,r}
\eq
such that $t\in (-1,1)$,  $\sigma_0 = \sigma$, and each $\sigma_t$ has the same boundary and asymptotic conditions as $\sigma$. Let $X$ be the vector field such that
\bq
\sigma_t = \sigma_0 + tX + \text{higher order terms in }t.
\eq
It is an element of 
$ C^\infty \big( R; \sigma^\ast (TM), \sigma^\ast (TL_0), \sigma^\ast(TL_1) \big)
$. By construction, it belongs to the kernel of $D_\sigma$. 
Moreover, using exponential maps, one can construct a one-parametre family $\sigma_t$ such that any smooth element of the kernel arises in this way. The uniqueness argument of Section \ref{sec:uniqueness} show that the map
\bq
\pi \circ \sigma_t: R \times (-1,1) \to \C
\eq
is the constant map to $\star$. Thus 
\bq
0 = \frac{d}{dt} (\pi \circ \sigma_t) = D \pi \Big( \frac{d}{dt} \sigma_t \Big)  = D\pi ( X).
\eq
This means the horizontal component of $X$ always vanishes. It remains to understand its vertical component. For this, we just use so called `automatic regularity' for discs on Riemann surfaces (see e.g.~\cite[Section 13a]{Seidel08}). 

\subsection{Floer cohomology of $T$ with the vanishing cycles $A$ and $B$}

To complete the proof of Proposition \ref{th:floercohomology}, it remains to consider the case where the vanishing cycle involved is $A$ or $B$. We will use the local model for our construction described in section \ref{sec:localmodel}. 
Recall we use the Lefschetz fibration:
\begin{eqnarray}
\chi: & {\mathcal{C}}:=\{ (x,y,z) \in \C^2\times \C^\ast \, | \, x^2+y^2+z^2=1\} & \to \C^\ast \\
& (x,y,z) &\mapsto z.
\end{eqnarray}
equipped with the standard symplectic structure on $\C^2 \times T^\ast S^1$, where $T^\ast S^1$ is identified with $\C^\ast$ via $(p,q) \mapsto e^{p+iq}$. 
 The cases of $A$ and $B$ are analogous; let us compute Floer cohomology between $T$ and $B$. 

After Hamiltonian isotopy, we can arrange for the curve defining $T$, and the matching path for $B$, to intersect in exactly one point in $\C^\ast$, say $w$, as in Figure \ref{fig:localmodelagain}. 

\begin{figure}[htb]
\begin{center}
\includegraphics[scale=0.85]{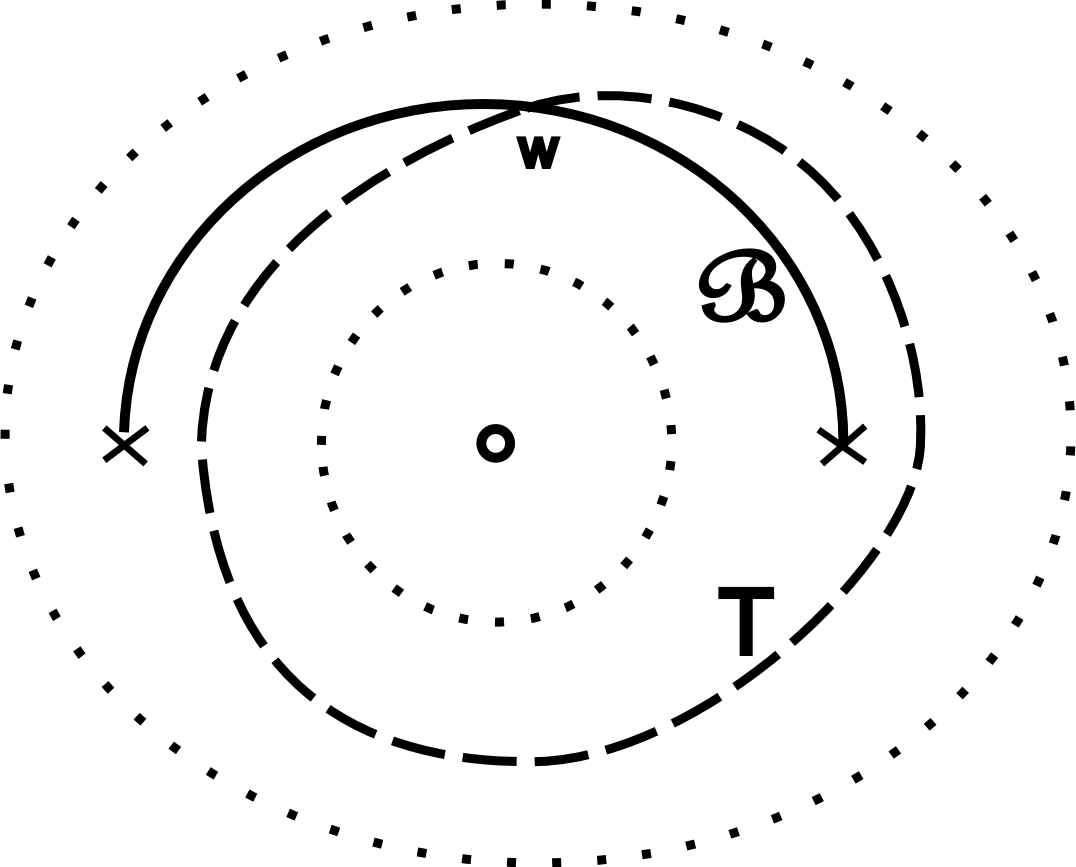}
\caption{Local model for the intersection of $T$ and $B$: base. 
}
\label{fig:localmodelagain}
\end{center}
\end{figure}
After a further Hamiltonian isotopy (with an essentially local change of the symplectic form, for instance), we can assume that on $\chi^{-1}(w)$, which is a copy of $T^\ast S^1$, $T$ and $B$ restrict to two meridional $S^1$'s that intersect in two points, and are Hamiltonian deformations of each other. There are two `immediate' holomorphic discs between the two intersection points, say $b_1$ and $b_2$, corresponding to the two closed regions bounded by the union of the two $S^1$'s. As before, we want to prove that these are the only holomorphic discs.

Consider the complex structure $J$ on $\mathcal{C}$ inherited from the standard complex structure on $\C^2 \times T^\ast S^1$. Choose  a tubular neighbourhood of the Lagrangian spheres $\mathcal{A} \cup \mathcal{B}$, say $\nu$, such that $\partial \nu$ is of contact type, and $J$ is of contact type near $\partial \nu$. (One can for instance use the dotted concentric circles of Figure \ref{fig:localmodelagain}. We assume that $T \subset \nu$.) This pulls back to an almost complex structure on a tubular neighbourhood of $A \cup B$, compatible with $\omega$. Moreover, we can extend it to an $\o$--compatible complex structure on the whole of $\T_{p,q,r}$, such that the extension also agrees with the usual complex structure outside a compact set. By convexity considerations, no holomorphic disc involved in the differential of the Floer cohomology between $B$ and $T$ can leave the tubular neighbourhood of $A \cup B$. In particular, the Floer cohomology in the local model agrees with the Floer cohomology in the total space.

Suppose we have a holomorphic map
\bq
\phi: \R \times [0,1 ] \to \mathcal{C}
\eq
with Lagrangian boundary conditions given by $T$ and $B$, and asymptotics $b_1$ and $b_2$. First, by the open mapping theorem, we must have that $\chi \circ \phi =0$. Thus the image of $\phi$ must lie in $\chi^{-1}(w)$. By topological considerations, the only possibilities are the two discs that we already had. One can then check that they have Maslov index one, and, for instance using analogous arguments to the previous case, they are regular.

\section{Mirror symmetry for $T_{p,q,r}$} \label{sec:mirrorsymmetry}

This section contains the proof of Theorem \ref{th:mirrorsymmetry}, announced in the introduction:
\begin{theorem} \label{th:mirrorsymmetry}
There is an equivalence 
\bq
D^b Fuk^{\to} (T_{p,q,r}) \cong D^b Coh (\P^1_{p,q,r})
\eq
where the left-hand side is the bounded derived directed Fukaya category of the singularity $T_{p,q,r}$, and the right-hand side is the bounded derived category of coherent sheaves on an orbifold $\P^1$, with orbifold points of isotropies $1/p$, $1/q$ and $1/r$.
\end{theorem}
To obtain this result, we  compare presentations of the two categories. The category of coherent sheaves on an orbifold $\P^1$, and its bounded derived extension, were already understood. We shall use the same definitions as \cite{ChenKrause}, and their description of the derived category $D^b \text{Coh} (\P^1_{p,q,r})$ (see below). Using our description of $\mathcal{T}_{p,q,r}$, we are able to calculate the bounded derived directed Fukaya category of $T_{p,q,r}$, and show that it is isomorphic to $D^b \text{Coh} (\P^1_{p,q,r})$. 

\subsection{The derived category of coherent sheaves on an orbifold $\P^1$}

We take the following description from \cite[Section 6.9]{ChenKrause}: there is an equivalence
\bq
D^b \text{Coh} (\P^1_{p,q,r}) \cong D^b(\text{mod} \mathcal{A})
\eq
where $\mathcal{A}$ is the finite dimensional associative algebra given by the following quiver:
\bq
\xymatrix{
 & & \bullet \ar[r]^{x_1} & \bullet \ar[r]^{x_2} & \bullet \ldots \bullet \ar[r]^{x_{p-2}} & \bullet \\
\bullet  \ar@/_/[r]_{a_2} \ar@/^/[r]^{a_1}  
& \bullet \ar[ur]^{b_1}  \ar[r]^{b_2} \ar[dr]_{b_3}
&
\bullet \ar[r]^{y_1} & \bullet \ar[r]^{y_2} & \bullet \ldots \bullet \ar[r]^{y_{q-2}} & \bullet \\
& & \bullet \ar[r]^{z_1} & \bullet \ar[r]^{z_2} & \bullet \ldots  \bullet \ar[r]^{z_{r-2}} & \bullet \\
}
\eq
modulo the relations 
\bq
b_1 \circ a_2 = 0 \qquad b_2 \circ a_1 = 0 \qquad b_3 \circ (a_1 - a_2) =0
\eq
(When comparing with \cite[Section 6.9]{ChenKrause}, note that we've assumed that the three orbifold points lie at $[0;1]$, $[1;0]$ and $[1;1]$.)
Note each vertex also comes with the identity morphism, which we suppress from the notation. 

\subsection{The derived directed Fukaya category of $T_{p,q,r}$}
Given any singularity $f$, one can associate to it an $A_\infty$ category, $D^b Fuk^{\to} (f)$. For a detailed introduction, see \cite[Chapter 3]{Seidel08}. Here's a definition. Suppose we have already fixed universal choices of regular perturbation data to define the Fukaya category of the Milnor fibre of $f$. (This is the version of the Fukaya category that we have been considering so far.) Start with a distinguished basis of vanishing cycles for $f$, say $V_1, \ldots, V_\mu$. (Here we use not a cyclic ordering, but some fixed choice of absolute ordering that agrees with it.)  The directed Fukaya category associated to our basis has objects the $V_i$. Morphism spaces are as follows:
\bq
hom(V_i, V_j) = 
\begin{cases}
0 & i > j \\
\C e_i & i = j \\
CF^\ast (V_i, V_j) & i < j
\end{cases}
\eq
 It is strictly unital. 
Suppose that $a_k \in CF(V_{k_1}, V_{k_2})$, with $k_1 <k_2$ for all $k$. Then the $A_\infty$--product 
\bq
 \mu^d (a_0, a_1, \ldots, a_d)
\eq
is just given by the $A_\infty$--product in the Fukaya category of the Milnor fibre. In all other cases, apart from the $\mu^2$ products implied by unitality, the $A_\infty$--products vanish. While each of the directed Fukaya categories depends a priori on the choice of distinguished basis, we get an invariant of the singularity when we pass to the bounded derived completion \cite[Theorem 18.24]{Seidel08}.

Let's compute this in the case of $T_{p,q,r}$. We'll use the geometric description \ref{th:Tpqr}, and the following ordered basis of vanishing cycles:
\bq
A, B, P_1, \ldots, P_{p-1}, Q_1, \ldots, Q_{q-1}, R_1, \ldots, R_{r-1}.
\eq
(This can be obtained from the previous order through trivial mutations.) 
By general considerations, the $A_\infty$--category $D^b \Fuk^{\to} (T_{p,q,r})$ is isomorphic to $D^b ( \text{mod} (\mathcal{B}))$, where $\mathcal{B}$ is the endomorphism algebra of 
$$
A \oplus B \oplus P_1 \oplus \ldots \oplus R_{r-1}.
$$
(This is proved in \cite{Bondal}; see \cite[Theorem 8.34]{Huybrechts} for an exposition.)
We shall first describe this algebra.
What are the $hom$ spaces between vanishing cycles? $A$ and $B$ intersect each of $P_1$, $Q_1$ and $R_1$ in one point, which lies on $M_\ast$. Label them as before as $p_A \in CF^\ast (A, P_1)$, and similarly with $p_B, q_q$, \ldots $r_B$ in Figure \ref{fig:RSfourlocaldiscsHMS}. $A$ and $B$ intersect each other twice, giving generators $u, v \in CF^\ast (A,B)$. 
\begin{figure}[htb]
\begin{center}x
\includegraphics[scale=1.1]{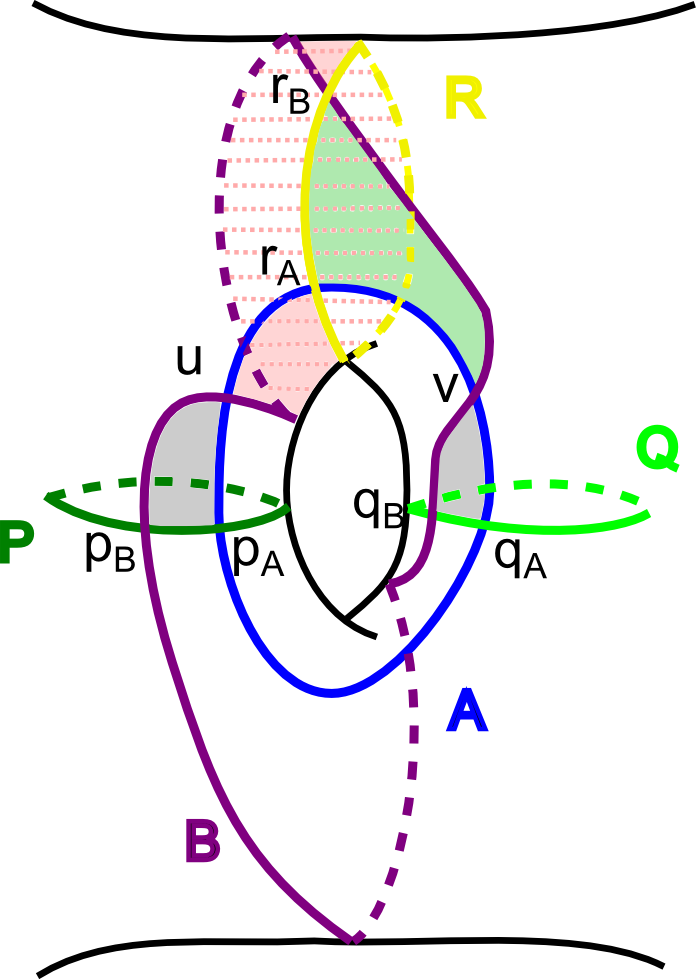}
\caption{Fibre $M_\star$, restricted to a neighbourhood of $A$ and $B$.
}
\label{fig:RSfourlocaldiscsHMS}
\end{center}
\end{figure}
Aside from these, there is one intersection point between $P_i$ and $P_{i+1}$, for each $i=1, \ldots, p-2$, say $p_i \in CF^\ast (P_i, P_{i+1})$, and similarly for the $Q_i$ and $R_i$. There are no other intersection points between vanishing cycles. 

  To calculate $\mathcal{B}$, we need the $A_\infty$--morphisms between the vanishing cycles.
We'll use the same symplectic form and almost complex structure as in Section \ref{sec:Fukayaunimodal}. There are four holomorphic discs in Figure \ref{fig:RSfourlocaldiscsHMS}. Using arguments completely analogous to Section \ref{sec:Fukayaunimodal}, we see that they are unique, of Maslov index zero, and regular. The intersections do not allow for $A_\infty$--morphisms between any other collection of vanishing cycles (even disregarding our choice of ordering). Thus the almost-complex structure we chose is regular. We have the following $A_\infty$--products:
\bq
\mu^2(p_B , u) = p_A \qquad \mu^2(q_B , v) = q_A \qquad \mu^2(r_B , u) = r_A \qquad \mu^2(r_B , v) = r_A
\eq (Note that we can always change signs of the generators so that these hold precisely.)
  Thus $\mathcal{B}$ is the 
  finite dimensional associative algebra given by the directed quiver:
  \bq
\xymatrix{
 & & \bullet_{P_1} \ar[r]^{p_1} & \bullet_{P_2} \ar[r]^{p_2} & \bullet \ldots  \bullet\ar[r]^-{p_{p-2}} & \bullet_{P_{r-1}} \\
\bullet_A  \ar@/_/[r]_{v} \ar@/^/[r]^{u}  
& \bullet_B \ar[ur]^{p_B}  \ar[r]^{q_B} \ar[dr]_{r_B}
&
\bullet_{Q_1} \ar[r]^{q_1} & \bullet_{Q_2} \ar[r]^{q_2} & \bullet \ldots \bullet \ar[r]^-{q_{q-2}} & \bullet_{Q_{q-1}} \\
& & \bullet_{R_1} \ar[r]^{r_1} & \bullet_{R_2} \ar[r]^{r_2} & \bullet \ldots \bullet \ar[r]^-{r_{r-2}} & \bullet_{R_{r-1}} \\
}
\eq
modulo the equations 
\bq
p_B \circ v = 0 \qquad q_B \circ u =0 \qquad r_B \circ (u-v)=0
\eq
and the fact that any sequence of the $p_i$, $q_i$ or $r_i$ composes to zero. (As with  $\mathcal{A}$, we suppress the identity morphism at each vertex from the notation.) Now consider the following elements of $D^b \Fuk^{\to}(T_{p,q,r})$, given as twisted complexes:
\begin{eqnarray}
\xymatrix{
P'_1 = \{ P_1 \ar[r]^-{p_1} &   P_2 \ar[r]^-{p_2} & \ldots \ar[r]^-{p_{p-2}} & P_{p-1} \} } \\
\xymatrix{
P'_2 = \{  P_1 \ar[r]^-{p_1} & \ldots \ar[r]^-{p_{p-3}} & P_{p-2} \} } \\
\ldots \quad  \nonumber \\
P'_{p-1} = P_{1}
\end{eqnarray}
and similarly for $Q'_i$ and $R'_i$. Let $\Pi_i$ be the projection $P'_{i+1} \to P'_i$. We'll use the same notation for the $Q'_i$ and $R'_i$. The collection $A$, $B$, $P'_1$, \ldots, $P'_{p-1}$, $Q'_1$, \ldots, $R'_{r-1}$ also generates $D^b \Fuk^{\to} (T_{p,q,r})$. Also, we have that:
\begin{itemize}
\item Each of the $P'_i$ has one-dimensional self cohomology. To see this, one can either perform an algebraic computation, or notice that each $P_i'$ corresponds to a vanishing cycle; for instance, $P_{p-2}'$ is the result of the Dehn twist of $P_2$ in $P_1$.
\item Morphisms between the $P_i'$'s are given by
\bq
Hom (P_i ', P_j ') = 
\begin{cases}
0 & i > j \\
\C  & i \leq j
\end{cases}
\eq
and generators are given by either the identity or by compositions $\Pi_i \circ \ldots \circ \Pi_{j-1}$. 
\end{itemize}
Thus the endomorphism algebra of 
$$
A \oplus B \oplus P'_1 \oplus \ldots \oplus R'_{r-1}
$$
is precisely $\mathcal{A}$, given this time by the algebra of the directed quiver
  \bq
\xymatrix{
 & & \bullet_{P'_1} \ar[r]^{\Pi_2} & \bullet_{P'_2} \ar[r]^{\Pi_3} & \bullet \ldots \bullet \ar[r]^-{\Pi_{r-1}} & \bullet_{P'_{r-1}} \\
\bullet_A  \ar@/_/[r]_{v} \ar@/^/[r]^{u}  
& \bullet_B \ar[ur]^{p_B}  \ar[r]^{q_B} \ar[dr]_{r_B}
&
\bullet_{Q'_1} \ar[r]^{\Pi_2} & \bullet_{Q'_2} \ar[r]^{\Pi_3} & \bullet \ldots \bullet \ar[r]^-{\Pi_{q-1}} & \bullet_{Q'_{q-1}} \\
& & \bullet_{R'_1} \ar[r]^{\Pi_2} & \bullet_{R'_2} \ar[r]^{\Pi_3} &\bullet  \ldots \bullet \ar[r]^-{\Pi_{r-1}} & \bullet_{R'_{r-1}} \\
}
\eq
   modulo the equations $p_B \circ v = 0$, $q_B \circ u =0$ and $ r_B \circ (u-v)=0$. Thus $$D^b \Fuk^{\to}(T_{p,q,r}) \cong D^b(\text{mod} \mathcal{A}).$$ This  completes the proof of Theorem \ref{th:mirrorsymmetry}.

\appendix

\section{Visualizing the singular values of $M(x,y,z;t)$}\label{ap:mathematicacode}

The following code should allow the reader to visualise the singular values of $M(x,y,z;t)$ as $t$ varies from 0 to 1. It was made and tested using Mathematica 8.0.
\begin{verbatim}
m[x_, y_] := -2((x + 0.25)^2 - 2 - 0.5 (y + 0.25)) ((y + 0.25)^2 - 2 - 
    0.5 (x + 0.25))
u = 0;
g[x_] = Integrate[x (x + 1), x];
h[x_] = (8 I)*g[x];
plots1 = Table[
   g[x_, y_, z_] := m[x, y] - 2x*y + 2*(3 z + u*x*y)^2 + 2*h[z];
    spts =  NSolve[D[g[x, y, z], x] == 0 && D[g[x, y, z], y] == 0 && 
      D[g[x, y, z], z] == 0, {x, y, z}];
   results = g[x, y, z] /. spts;
   ListPlot[{Re[#], Im[#]} & /@ results, AxesOrigin -> {0, 0}, 
    PlotRange -> {{-40, 20}, {-60, 20}}, ImagePadding -> 40, 
    AspectRatio -> 1, Frame -> True, 
    FrameLabel -> {{Im, None}, {Re, "complex plane"}}, 
    PlotStyle -> Directive[Red, PointSize[.02]]], {u, 0, 1, 0.02}];
ListAnimate[plots1]
\end{verbatim}

\section{Viewing $M(x,y,z;1)$ as a deformation of $T_{3,3,3}$}\label{ap:B}

The following code should allow the reader to check that $M(x,y,z;1)$ is indeed a deformation of $N(x,y,z) = x^3+y^3+12xyz+8iz^2+\frac{16i}{3}z^2$. We have that
\bq
M(x,y,z;1) = N(x,y,z) + \delta(x,y)
\eq
The code plots the critical values of the function
\bq
L(x,y,x;t) = N(x,y,z) + (1-t) \delta(x,y)
\eq
We see that $t$ varies between zero and one, the critical values get deformed continuously (none escapes to infinity, and none comes in from infinity). 
\begin{verbatim}
m[x_, y_] := -2 ((x + 0.25)^2 - 2 - 0.5 (y + 0.25)) ((y + 0.25)^2 - 
    2 - 0.5 (x + 0.25))
u = 0;
g[x_] = Integrate[x (x + 1), x];
h[x_] = (8 I)*g[x];
g[x_, y_, z_] = m[x, y] - 2 x*y + 2*(3 z + x*y)^2 + 2*h[z];
ExpandAll[g[x, y, z]]
d[x_, y_, z_] = g[x, y, z] - x^3 - y^3 - 12*x*y*z - 2*h[z];
ExpandAll[d[x, y, z]]
plots1 = Table[L[x_, y_, z_] := g[x, y, z] - t*d[x, y, z];
   spts = 
    NSolve[D[L[x, y, z], x] == 0 && D[L[x, y, z], y] == 0 && 
      D[L[x, y, z], z] == 0, {x, y, z}];
   results = L[x, y, z] /. spts;
   ListPlot[{Re[#], Im[#]} & /@ results, AxesOrigin -> {0, 0}, 
    PlotRange -> {{-40, 20}, {-60, 20}}, ImagePadding -> 40, 
    AspectRatio -> 1, Frame -> True, 
    FrameLabel -> {{Im, None}, {Re, "complex plane"}}, 
    PlotStyle -> Directive[Red, PointSize[.02]]], {t, 0, 1, 0.02}];
ListAnimate[plots1]
\end{verbatim}

\end{document}